\documentclass{amsart}

\usepackage{amsmath,amssymb,amsfonts,amsthm}
\usepackage[all]{xy}
\usepackage{enumerate}
\usepackage{tikz}
\usepackage{array}
\usepackage{mathrsfs}
\usepackage{csquotes}
\usepackage{fullpage}
\usepackage{hyperref}

\theoremstyle{theorem}
\newtheorem{thm}{Theorem}[section]
\newtheorem{prop}[thm]{Proposition}
\newtheorem{lem}[thm]{Lemma}
\newtheorem{coro}[thm]{Corollary}

\theoremstyle{remark}
\newtheorem{rem}[thm]{Remark}
\newtheorem{ex}[thm]{Example}
\newtheorem{notation}[thm]{Notation}

\theoremstyle{definition}
\newtheorem{defi}[thm]{Definition}

\numberwithin{equation}{section}

\renewcommand{\P}{\mathbb{P}}
\newcommand{\Q}{\mathbb{Q}}
\newcommand{\C}{\mathbb{C}}

\newcommand{\s}{\mathscr{S}}

\newcommand{\gr}{\mathrm{gr}}
\newcommand{\sgn}{\mathrm{sgn}}
\newcommand{\td}[1]{\widetilde{#1}}
\newcommand{\qD}{{}^{(q)}D}
\newcommand{\kA}{{}^{(k)}A}

\newcommand{\A}{\mathscr{A}}
\newcommand{\B}{\mathscr{B}}
\renewcommand{\L}{\mathscr{L}}
\newcommand{\M}{\mathscr{M}}

\newcommand{\hookrightpar}{\overset{1}{\underset{\parallel}{\hookrightarrow}}}
\newcommand{\hookrightperp}{\overset{1}{\underset{\perp}{\hookrightarrow}}}

\renewcommand{\leq}{\leqslant}
\renewcommand{\geq}{\geqslant}

\definecolor{navyblue}{rgb}{0.000000,0.000000,0.501961}



\title{Relative cohomology of bi-arrangements}
\author{Cl{\'e}ment Dupont}
\address{Max-Planck-Institut f\"{u}r Mathematik\\Vivatsgasse, 7 \\53111 Bonn, Germany}
\email{cdupont@mpim-bonn.mpg.de}

\begin{document}
\maketitle

\begin{abstract}
A \textit{bi-arrangement of hyperplanes} in a complex affine space is the data of two sets of hyperplanes along with a coloring information on the strata. To such a bi-arrangement, one naturally associates a relative cohomology group, that we call its \textit{motive}. The main reason for studying such relative cohomology groups comes from the notion of \textit{motivic period}. More generally, we suggest the systematic study of the motive of a \textit{bi-arrangement of hypersurfaces} in a complex manifold. We provide combinatorial and cohomological tools to compute the structure of these motives. Our main object is the \textit{Orlik--Solomon bi-complex} of a bi-arrangement, which generalizes the Orlik--Solomon algebra of an arrangement. Loosely speaking, our main result states that \enquote{the motive of an exact bi-arrangement is computed by its Orlik--Solomon bi-complex}, which generalizes classical facts involving the Orlik--Solomon algebra of an arrangement. We show how this formalism allows us to explicitly compute motives arising from the study of multiple zeta values and sketch a more general application to periods of mixed Tate motives.
\end{abstract}

\setcounter{tocdepth}{1}

\section{Introduction}

	Let us consider a set of hyperplanes in a complex affine or projective space, which we call an \textit{arrangement of hyperplanes}. A natural question, raised by Arnol'd~\cite{arnold}, is to understand the cohomology ring of the complement of the union of the hyperplanes in the arrangement. This question was settled in two steps by Brieskorn~\cite{brieskorn} and Orlik and Solomon~\cite{orliksolomon} and led to the introduction of the Orlik--Solomon algebra of an arrangement of hyperplanes, which has now become a classical tool in algebraic topology and combinatorics.
	
	In this article, we recast these classical results as part of a more general framework. We define a \textit{bi-arrangement of hyperplanes} in a complex affine or projective space to be the data of two sets~$\L$ and~$\M$ of hyperplanes, along with a coloring function~$\chi$ which associates to each stratum (intersection of some hyperplanes from~$\L$ and~$\M$) the color~$\lambda$ or the color~$\mu$. An arrangement of hyperplanes is then simply a bi-arrangement of hyperplanes for which~$\M=\varnothing$ and~$\chi$ only takes the color~$\lambda$.
	
	More generally, a bi-arrangement of hypersurfaces in a complex manifold~$X$ is the data of two sets of smooth hypersurfaces of~$X$ and a coloring function, which is a bi-arrangement of hyperplanes in every local chart on~$X$.
	
	The \textit{motive}\footnote{See \S\ref{parterminologymotive} for a discussion on this terminology.} of a bi-arrangement of hypersurfaces~$(\L,\M,\chi)$ in~$X$ is the collection of the relative cohomology groups (with coefficients in~$\Q$):
	\begin{equation}\label{eqmotiveintro}
	H^\bullet(\td{X}\setminus \td{\L},\td{\M}\setminus \td{\M}\cap\td{\L}).
	\end{equation}
	Here~$\pi:\td{X}\rightarrow X$ is a resolution of the singularities of~$\L\cup\M$ and~$\td{\L}\cup\td{\M}=\pi^{-1}(\L\cup\M)$ is a normal crossing divisor with a given partition of its irreducible components determined by the coloring function~$\chi$. In the case of an arrangement of hypersurfaces $(\L,\varnothing,\lambda)$, this is simply the cohomology of the complement:~$H^\bullet(X\setminus\L)$.
	
	Our motivation for studying the relative cohomology groups (\ref{eqmotiveintro}) mainly comes from the notion of motivic period, see \S\ref{parintroperiods} for more details. \\
	
	In this article, we introduce tools to compute the motive of a given bi-arrangement. 
	\begin{enumerate}[--]
	\item In the local context of hyperplanes in~$\C^n$, we define the \textit{Orlik--Solomon bi-complex} of a bi-arrangement of hyperplanes, generalizing the construction of the Orlik--Solomon algebra. This allows us to single out a natural class of bi-arrangements for which the Orlik--Solomon bi-complex is well-behaved, that we call \textit{exact}, and that includes all arrangements of hyperplanes.
	\item In the global context of hypersurfaces in a complex manifold~$X$, we define the \textit{geometric Orlik--Solomon bi-complex} of a bi-arrangement of hypersurfaces, which incorporates the combinatorial datum of the Orlik--Solomon bi-complexes and the cohomological datum of the geometric situation.
	\end{enumerate}
	
	Our main result can then be vaguely stated as follows.
	
	\begin{thm}
	The motive of an exact bi-arrangement is computed by its Orlik--Solomon bi-complex.
	\end{thm}
	
	In the special case of arrangements, we recover the classical Brieskorn--Orlik--Solomon theorem in the local context, and its global counterpart proved by Looijenga~\cite{looijenga} (see also~\cite{duponthypersurface}) in the global context.\\
	
	 Before we turn to a more detailed description of our results in \S\ref{intropar2} and \S\ref{intropar3}, we explain in \S\ref{parintroperiods} the motivation behind the study of bi-arrangements and their motives. Even though this motivation will not be apparent in most of this article, we think that it gives a good intuition on the objects that we are studying.
	 
	\subsection{Periods of bi-arrangements and relative cohomology}\label{parintroperiods}
	
		The value of the Riemann zeta function at an integer~$n\geq 2$ is defined by the series
		\begin{equation}\label{eqzetasum}
		\zeta(n)=\sum_{k\geq 1}\frac{1}{k^n}.
		\end{equation}
		These numbers have been first studied by Euler who showed that~$\zeta(2n)$ is a rational multiple of~$\pi^{2n}$, e.g.~$\zeta(2)=\frac{\pi^2}{6}$. Little is known about the arithmetic properties of the numbers~$\zeta(2n+1)$~\cite{apery,ballrivoal,zudilin}.
		An important fact about these numbers is that they have a representation as a multiple integral (expand~$\frac{1}{1-x_1}$ as a geometric series and integrate inductively with respect to~$x_1,\ldots,x_n$):
		\begin{equation}\label{eqzetaintegral}
		\zeta(n)=\int_{0<x_1< \cdots< x_n< 1}\frac{dx_1}{1-x_1}\frac{dx_2}{x_2}\cdots \frac{dx_n}{x_n}\cdot
		\end{equation}
		
		This representation allows us to view~$\zeta(n)$ as the \textit{period} of a certain cohomology group (motive). We now explain how this works for the case of~$\zeta(2)=\int_{0< x< y< 1}\frac{dx}{1-x}\frac{dy}{y}$. 
		 
		Let us consider the geometric situation pictured in the left-hand side of the figure below. In~$X=\P^2(\C)$ with affine coordinates~$(x,y)$, let~$\L$ (the dashed lines, in blue) be the divisor of poles of the form~$\omega=\frac{dx}{1-x}\frac{dy}{y}$. It is the union of the line at infinity and the lines~$\{x=1\}$,~$\{y=0\}$. Let now~$\M$ (the full lines, in red) be the Zariski closure of the boundary of the domain of integration~$\Delta=\{0< x< y< 1\}$ (the shaded triangle). It is the union of the lines~$\{x=0\}$,~$\{x=y\}$,~$\{y=1\}$.
		
		The divisor~$\L\cup\M$ is not normal crossing in~$X$. We let~$\pi:\td{X}\rightarrow X$ be the blow-up along the points~$P_1$,~$P_2$,~$Q_1$,~$Q_2$, and let~$E_1$,~$E_2$,~$F_1$,~$F_2$ be the corresponding exceptional divisors. We let~$\td{\L}$ be the union of~$ E_1$,~$E_2$, and the strict transforms of the three lines from~$\L$; we let~$\td{\M}$ be the union of~$F_1$,~$F_2$, and the strict transforms of the three lines from~$\M$. Now~$\td{\L}\cup\td{\M}=\pi^{-1}(\L\cup\M)$ is a normal crossing divisor in~$\td{X}$, pictured in the right-hand side of the figure below.
		
		\begin{figure}[h!!]
		\def\svgwidth{1\textwidth}
		\begin{center}
		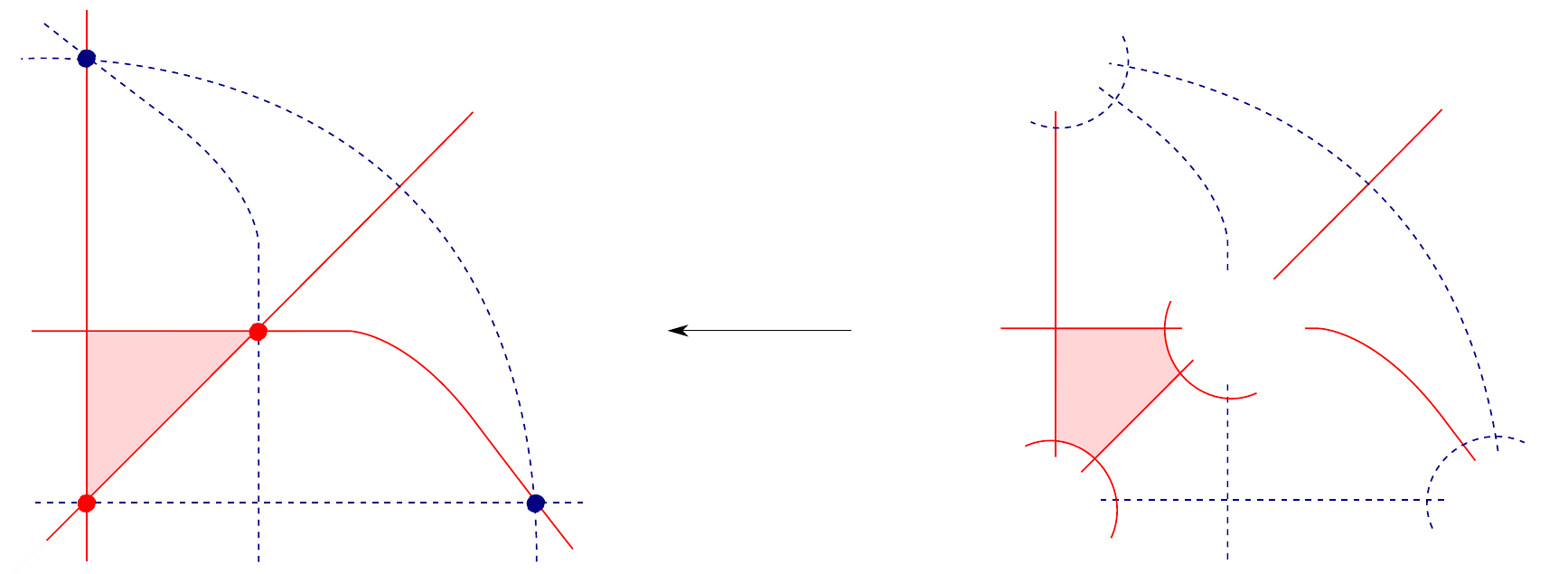
		\end{center}
		\end{figure}
		
		Let us introduce the relative cohomology group (with coefficients in~$\Q$)
		\begin{equation}\label{eqintromotivezeta2}
		H=H^2(\td{X}\setminus\td{\L},\td{\M}\setminus\td{\M}\cap\td{\L}).
		\end{equation}
		The differential form~$\pi^*(\omega)$ is closed and has poles along~$\td{\L}$, hence defines a cohomology class in~$H$. The domain~$\pi^{-1}(\Delta)$ (the shaded pentagon in the figure above) has its boundary on~$\td{\M}$, hence defines a homology class in~$H^\vee$. Hence,
		$$\zeta(2)=\int_\Delta\omega=\int_{\pi^{-1}(\Delta)}\pi^*(\omega)$$
		is a period of~$H$.\\ 
		More precisely,~$H$ is a mixed Tate motive over~$\mathbb{Z}$, the class of~$\pi^*(\omega)$ lives in the algebraic de Rham cohomology group~$H_{dR}$, the class of~$\pi^{-1}(\Delta)$ lives in the Betti (singular) homology group~$H_B^\vee$, and~$\zeta(2)$ appears as the pairing between these classes via the comparison isomorphism between de Rham and Betti cohomology. Note that the (equivalence class of the) triple 
		\begin{equation}\label{eqmotivicperiod}
		(H,[\pi^*(\omega)],[\pi^{-1}(\Delta)])
		\end{equation}
		is called the \textit{motivic period} corresponding to~$\zeta(2)$, and is an algebro-geometric avatar of the integral (\ref{eqzetaintegral}). The interested reader will find in the author's PhD thesis~\cite{dupontPhD} more details on the construction of more general motivic periods.\\
		
		At this point, we want to answer two natural questions.  \\
		
		\textit{Why work in the blow-up?} One could want to replace~$H$ with the (simpler) relative cohomology group~$H'=H^2(X\setminus\L,\M\setminus\M\cap\L)$. This is wrong because the boundary of~$\Delta$ intersects~$\L$, hence~$\Delta$ does not define a homology class in~$H'$. This is why we have to work in the blown-up situation. Furthermore, working with normal crossing divisors reveals a hidden (Poincar\'{e}--Verdier) duality: exchanging~$\td{\L}$ and~$\td{\M}$ corresponds to the linear duality between the cohomology groups (\ref{eqintromotivezeta2}). These two reasons should begin to convince the reader that the (\textit{a priori} tedious) blow-up process is the correct thing to do and that cohomology groups like~$H$ are more relevant than their simpler counterparts~$H'$. We hope that the results of this article will appear as another argument in favor of this point.\\
		
		\textit{How to deal with the exceptional divisors?} In the above example, it is crucial that~$E_1$ and~$E_2$ are part of the divisor~$\td{\L}$ since~$\pi^*(\omega)$ has poles along them; in the same fashion, it is crucial that~$F_1$ and~$F_2$ are part of the divisor~$\td{\M}$ since~$\pi^{-1}(\Delta)$ has boundary components on them. In higher dimensional situations, we may have a choice to make between~$\td{\L}$ and~$\td{\M}$ that is not imposed by the geometry. We keep track of these choices using a coloring function~$\chi$, which assigns the color~$\lambda$ (for~$\td{\L}$) or~$\mu$ (for~$\td{\M}$) to the strata that we blow up. Here we would then have~$\chi(P_1)=\chi(P_2)=\lambda$ and~$\chi(Q_1)=\chi(Q_2)=\mu$.\\
		
		To sum up, we have a triple~$(\L,\M,\chi)$ made of two sets of (projective) hyperplanes, and the coloring function. This triple is what we call a (projective) \textit{bi-arrangement}. The \textit{motive} of this bi-arrangement is the relative cohomology group (\ref{eqintromotivezeta2}).\\
			
		The idea that the study of cohomology groups like (\ref{eqintromotivezeta2}) and motivic periods like (\ref{eqmotivicperiod}) tells us something about integrals like (\ref{eqzetaintegral}) is (implicitly or explicitly) present at many places of the literature, including Deligne and Goncharov's theory of motivic fundamental groupoids~\cite{delignedroiteprojective,goncharovgaloissym,delignegoncharov}, Goncharov and Manin's description of the multiple zeta motives~\cite{goncharovmanin}, Brown's work on multiple zeta values~\cite{brownMTMZ}, the general theory of periods~\cite{kontsevichzagier,andregalois} and foundational work on the theory of motives~\cite{kontsevichoperadsmotives,hubermullerstach}. In physics, this point of view gives rises to the notion of Feynman motive, which is a precious tool in the study of the arithmetics and the analysis of Feyman integrals~\cite{blochesnaultkreimer,aluffimarcollibanana,marcollifeynmanintegralsmotives,doryngraphhypersurfaces,brownschnetzK3,mullerstachweinzierlzayadeh}.
		
		The question of computing the motives of projective bi-arrangements was implicitly asked by Beilinson \textit{et al.} in~\cite{bvgs} as part of the general programme of defining scissor congruence groups that compute the higher~$K$-theory of fields. As a special case of our result, we give a partial answer to their question by giving a combinatorial description of the weight-graded quotients of the motives of certain projective bi-arrangements, see Theorem \ref{thmprojective}. We will investigate further in that direction in a subsequent article.
	
	\subsection{From the Orlik--Solomon algebra to the Orlik--Solomon bi-complex}\label{intropar2}
	
		Following the pioneering work~\cite{arnold} of Arnol'd, the Orlik--Solomon algebra was introduced~\cite{orliksolomon} to understand the cohomology of the complement of a union of hyperplanes in an affine space~$\C^n$. Let~$\A=\{K_1,\ldots,K_k\}$ be an arrangement of hyperplanes in~$\C^n$, i.e. a finite set of hyperplanes of~$\C^n$ which pass through the origin. The Orlik--Solomon algebra~$A_\bullet=A_\bullet(\A)$ is a differential graded algebra (over $\mathbb{Q}$ and with differential of degree $-1$) which may be defined as an explicit quotient of the exterior algebra~$\Lambda^\bullet(e_1,\ldots,e_k)$ with~$d(e_i)=1$ for~$i=1,\ldots,k$. It has the remarkable property of admitting a direct sum decomposition
		$$A_r=\bigoplus_{S\in\s_r(\A)}A_r^S$$
		where~$\s_r(\A)$ denotes the set of strata of~$\A$ (intersections of hyperplanes from~$\A$) of codimension~$r$. There is a more geometric (but less explicit) way of defining the components~$A_r^S$ by induction on~$r$, starting with~$A_0=\Q$ and imposing exact sequences
		$$0\rightarrow A_r^\Sigma \stackrel{d}{\longrightarrow} \bigoplus_{\Sigma\stackrel{1}{\hookrightarrow}S}A_{r-1}^S \stackrel{d}{\longrightarrow} \bigoplus_{\Sigma\stackrel{2}{\hookrightarrow}T}A_{r-2}^T$$
		where~$\Sigma\stackrel{c}{\hookrightarrow} \Sigma'$ denotes an inclusion of strata of~$\A$ with~$\dim(\Sigma)=\dim(\Sigma')-c$. This means that~$A_r^\Sigma$ is defined as the kernel of a previously defined morphism.
	
		Now let us fix two arrangements of hyperplanes~$\L=\{L_1,\ldots,L_l\}$ and~$\M=\{M_1,\ldots,M_m\}$ in~$\C^n$. We add the datum of a \textit{coloring function}~$\chi$ which associates to each stratum~$S\neq\C^n$ of~$\L\cup\M$ a color~$\chi(S)\in\{\lambda,\mu\}$ such that~$\chi(L_i)=\lambda$ for each~$i$,~$\chi(M_j)=\mu$ for each~$j$, plus a technical condition (see Definition~\ref{defibiarrangement}). The triple~$(\L,\M,\chi)$ is called a \textit{bi-arrangement} of hyperplanes. 
		
		The \textit{Orlik--Solomon bi-complex} of~$(\L,\M,\chi)$ is a bi-complex~$A_{\bullet,\bullet}=A_{\bullet,\bullet}(\L,\M,\chi)$ with differentials~$d':A_{\bullet,\bullet}\rightarrow A_{\bullet-1,\bullet}$ and~$d'':A_{\bullet,\bullet-1}\rightarrow A_{\bullet,\bullet}$. By definition, there is a direct sum decomposition 
		$$A_{i,j}=\bigoplus_{S\in \s_{i+j}(\L\cup\M)}A_{i,j}^S$$
		and we impose exact sequences  
		\begin{equation}\label{exactlambda}
		0\rightarrow A_{i,j}^\Sigma \stackrel{d'}{\longrightarrow} \bigoplus_{\Sigma\stackrel{1}{\hookrightarrow}S}A_{i-1,j}^S \stackrel{d'}{\longrightarrow} \bigoplus_{\Sigma\stackrel{2}{\hookrightarrow}T}A_{i-2,j}^T \;\;\textnormal{ if~$\chi(\Sigma)=\lambda$;}
		\end{equation}
		\begin{equation}\label{exactmu}
		0\leftarrow A_{i,j}^\Sigma \stackrel{d''}{\longleftarrow} \bigoplus_{\Sigma\stackrel{1}{\hookrightarrow}S}A_{i,j-1}^S \stackrel{d''}{\longleftarrow} \bigoplus_{\Sigma\stackrel{2}{\hookrightarrow}T}A_{i,j-2}^T \;\;\textnormal{ if~$\chi(\Sigma)=\mu$.}
		\end{equation}
		
		Starting with~$A_{0,0}=\Q$, this is enough to define the components~$A_{i,j}^S$ by induction on~$S$, as kernels or cokernels of previously defined morphisms. If~$\M=\varnothing$ and~$\chi$ takes only the value~$\lambda$, we recover the inductive definition of the Orlik--Solomon algebra:~$A_{\bullet,0}(\L,\varnothing,\lambda)=A_\bullet(\L)$. In the world of bi-arrangements there is a duality that exchanges the roles of~$\L$ and~$\lambda$ on the one hand, and~$\M$ and~$\mu$ on the other hand. This duality translates as the linear duality of the Orlik--Solomon bi-complexes.\\
		
		We are mostly interested in the bi-arrangements~$(\L,\M,\chi)$ such that the exact sequences (\ref{exactlambda}) and (\ref{exactmu}) may be extended to exact sequences 
		
		$$0\rightarrow A_{i,j}^\Sigma \stackrel{d'}{\longrightarrow} \bigoplus_{\Sigma\stackrel{1}{\hookrightarrow}S}A_{i-1,j}^S \stackrel{d'}{\longrightarrow} \bigoplus_{\Sigma\stackrel{2}{\hookrightarrow}T}A_{i-2,j}^T \stackrel{d'}{\longrightarrow}\cdots\stackrel{d'}{\longrightarrow} \bigoplus_{\Sigma\stackrel{i}{\hookrightarrow}Z}A_{0,j}^Z \rightarrow 0 \;\;\textnormal{ if~$\chi(\Sigma)=\lambda$;}$$ 
		
		$$0\leftarrow A_{i,j}^\Sigma \stackrel{d''}{\longleftarrow} \bigoplus_{\Sigma\stackrel{1}{\hookrightarrow}S}A_{i,j-1}^S \stackrel{d''}{\longleftarrow} \bigoplus_{\Sigma\stackrel{2}{\hookrightarrow}T}A_{i,j-2}^T \stackrel{d''}{\longleftarrow}\cdots\stackrel{d''}{\longleftarrow} \bigoplus_{\Sigma\stackrel{j}{\hookrightarrow}Z}A_{i,0}^Z \leftarrow 0 \;\;\textnormal{ if~$\chi(\Sigma)=\mu$.}$$
		
		These bi-arrangements are called \textit{exact}, and form a natural class of bi-arrangements that includes the arrangements~$(\L,\varnothing,\lambda)$ - this is because the Orlik--Solomon algebras are exact as complexes. \\ 
		
		
		The drawback of the inductive definition of the Orlik--Solomon bi-complexes is that we lack an explicit description as in the case of the Orlik--Solomon algebra. We solve this problem for a subclass of exact bi-arrangements that we call \textit{tame}. The tameness condition (see Definition~\ref{defitame}) is a simple combinatorial condition on the coloring which ensures that the colors~$\lambda$ and~$\mu$ \textit{do not interfere too much}.
		
		\begin{thm}[see Theorem~\ref{thmtameOS} for a precise statement]\label{thmintrotame}
		All tame bi-arrangements are exact. Furthermore, we may describe the Orlik--Solomon algebra of a tame bi-arrangement~$(\{L_1,\ldots,L_l\},\{M_1,\ldots,M_m\},\chi)$ as an explicit subquotient of the tensor product~$\Lambda^\bullet(e_1,\ldots,e_l)\otimes \Lambda^\bullet(f_1^\vee,\ldots,f_m^\vee)$ of two exterior algebras.
		\end{thm}
	
	\subsection{Bi-arrangements of hypersurfaces}\label{intropar3}
	
		We now turn to a global geometric situation. Let~$X$ be a complex manifold and~$(\L,\M,\chi)$ be a bi-arrangement of hypersurfaces in~$X$. This means that~$\L=\{L_1,\ldots,L_l\}$ and~$\M=\{M_1,\ldots,M_m\}$ are sets of smooth hypersurfaces of~$X$ such that locally around every point of~$X$,~$(\L,\M,\chi)$ is a bi-arrangement of hyperplanes. In particular, every stratum (connected component of an intersection of hypersurfaces)~$S$ of~$\L\cup\M$ is given a color~$\chi(S)\in\{\lambda,\mu\}$.
		
		The formalism of the Orlik--Solomon bi-complexes immediately extends from bi-arrangements of hyperplanes to bi-arrangements of hypersurfaces, using the same inductive definition.\\
		
		Using repeated blow-ups along strata, we may produce an explicit resolution of singularities (\enquote{wonderful compactification})~$\pi:\td{X}\rightarrow X$ such that~$\pi^{-1}(\L\cup\M)$ is a normal crossing divisor inside~$\td{X}$. The strata that we have blown up give rise to exceptional divisors in~$\td{X}$. We define 
		\begin{enumerate}[--]
		\item~$\td{\L}\subset\td{X}$ to be the union of the strict transforms of the hypersurfaces~$L_i$ along with the exceptional divisors corresponding to strata~$S$ such that~$\chi(S)=\lambda$;
		\item~$\td{\M}\subset \td{X}$ to be the union of the strict transforms of the hypersurfaces~$M_j$ along with the exceptional divisors corresponding to strata~$S$ such that~$\chi(S)=\mu$.
		\end{enumerate}
		We then have a normal crossing divisor~$\pi^{-1}(\L\cup\M)=\td{\L}\cup\td{\M}\subset \td{X}$. We define the \textit{motive} of the bi-arrangement of hypersurfaces~$(\L,\M,\chi)$ to be the collection of relative cohomology groups (with coefficients in~$\Q$) 
		\begin{equation}\label{eqintromotive}
		H^\bullet(\L,\M,\chi)=H^\bullet(\td{X}\setminus\td{\L},\td{\M}\setminus\td{\M}\cap\td{\L}).
		\end{equation}
		In the case of an arrangement of hypersurfaces~$(\L,\varnothing,\lambda)$ we simply have~$H^\bullet(\L,\varnothing,\lambda)=H^\bullet(X\setminus \L)$ the cohomology of the complement.
		
		The main result of this article is the following (see Theorem~\ref{maintheorem}). It states that for exact bi-arrangements of hypersurfaces, we may compute the corresponding motive 	via a spectral sequence that involves the cohomology of the strata and the Orlik--Solomon bi-complex of the bi-arrangement.
		
		\begin{thm}\label{intromaintheorem}
		Let~$(\L,\M,\chi)$ be an exact bi-arrangement of hypersurfaces in a complex manifold~$X$, with its Orlik--Solomon bi-complex~$A_{\bullet,\bullet}$.
		\begin{enumerate}
		\item There is a spectral sequence \footnote{Here,~$(-i)$ denotes the Tate twist of weight~$2i$. It is important in the algebraic case; otherwise it should be ignored.}
		\begin{equation}\label{eqspectralsequenceOS}
		E_1^{-p,q}=\bigoplus_{\substack{i-j=p\\ S\in\s_{i+j}(\L\cup\M)}} H^{q-2i}(S)(-i)\otimes A_{i,j}^S \;\Longrightarrow \; H^{-p+q}(\L,\M,\chi).
		\end{equation}
		\item If~$X$ is a smooth complex variety and all hypersurfaces of~$\L$ and~$\M$ are divisors in~$X$, then this is a spectral sequence in the category of mixed Hodge structures.
		\item If~$X$ is a smooth and projective complex variety, then this spectral sequence degenerates at the~$E_2$ term and we have
		$$E_\infty^{-p,q}\cong E_2^{-p+q} \cong \gr_q^W H^{-p+q}(\L,\M,\chi).$$
		\end{enumerate}
		\end{thm}
		
		The differential of the~$E_1$ page of the above spectral sequence is explicit. It is induced by the differentials of the Orlik--Solomon bi-complex and the Gysin and pullback morphisms corresponding to inclusions of strata.
	
		In the case of an arrangement of hypersurfaces~$(\L,\varnothing,\lambda)$, this gives a spectral sequence
		$$E_1^{-p,q}=\bigoplus_{S\in\s_p(\L)}H^{q-2p}(S)(-p)\otimes A_p^S(\L)  \;\Longrightarrow \; H^{-p+q}(X\setminus \L)~$$
		which was first defined in~\cite{looijenga} and studied in~\cite{duponthypersurface} in the context of logarithmic differential forms and mixed Hodge theory.
		This is a global generalization of the Brieskorn--Orlik--Solomon theorem~\cite{brieskorn,orliksolomon}, which corresponds to an arrangement of hyperplanes~$\A$ in~$X= \C^n$ and states that there is an isomorphism~$H^\bullet(\C^n\setminus \A)\cong A_\bullet(\A)$.\\
		
		Coming back to hyperplanes, we may apply Theorem~\ref{intromaintheorem} to the case of projective bi-arrangements of hyperplanes in~$X=\P^n(\C)$ (we could also apply it to affine bi-arrangements of hyperplanes in~$X=\C^n$, but this would give a less symmetric statement). As a corollary of Theorem~\ref{intromaintheorem}, we get the following.
		
		\begin{thm}[see Theorem~\ref{thmprojective} for a more precise statement]\label{introtheoremprojective}
		Let~$(\L,\M,\chi)$ be an exact projective bi-arrangement of hyperplanes in~$\P^n(\C)$. For~$k=0,\ldots,n$, let~$\kA_{\bullet,\bullet}$ be the bi-complex obtained by only keeping the rows~$0\leq i\leq k$ and the columns~$0\leq j\leq n-k$ of the Orlik--Solomon bi-complex of~$(\L,\M,\chi)$, and let~$\kA_\bullet$ be its total complex. We then have isomorphisms 
		$$\gr_{2k}^WH^r(\L,\M,\chi)\cong H_{2k-r}(\kA_\bullet)(-k).$$
		\end{thm}
		
		The above theorem implies that the weight-graded pieces of the motive of an exact bi-arrangement are \textit{combinatorial invariants} of the bi-arrangement. In general, the motive itself is not at all a combinatorial invariant. Indeed, the extension data between different weights are given by integrals like (\ref{eqzetaintegral}), which are sensitive to the equations of the hyperplanes. In the case of an arrangement of hyperplanes, this distinction does not appear, since the motive~$H^k(\L,\varnothing,\lambda)=H^k(\P^n(\C)\setminus \L)$ is concentrated in weight~$2k$.\\
		
		To prove Theorem~\ref{intromaintheorem}, our main object of study is the \textit{geometric Orlik--Solomon bi-complex}
		$$\qD_{i,j}=\bigoplus_{S\in \s_{i+j}(\L\cup\M)}H^{q-2i}(S)(-i)\otimes A_{i,j}^S.$$
		The key technical result (Theorem~\ref{maintheoremtechnical}) is thus the fact that there is a quasi-isomorphism between the geometric Orlik--Solomon bi-complex of a bi-arrangement and that of its blow-up. This allows us to reduce to the case where~$\L\cup\M$ is a normal crossing divisor in~$X$, for which Theorem~\ref{intromaintheorem} is a classical fact (Proposition~\ref{appspectralsequenceNCD}).
		
	\subsection{About the terminology}\label{parterminologymotive}
	
		Following~\cite{blochesnaultkreimer}, we use the word \textit{motive} in a non-technical sense, as a substitute for \enquote{relative cohomology group with some more structure}. We chose this homogeneous (non-standard) terminology because the objects~$H^\bullet(\L,\M,\chi)$ have incarnations in different categories, depending on the context.
		\begin{enumerate}[--]
		\item In the general case where~$X$ is a complex manifold,~$H^\bullet(\L,\M,\chi)$ is just a collection of vector spaces over~$\Q$.
		\item If~$X$ is a smooth complex variety and all the hypersurfaces in~$\L$ and~$\M$ are divisors in~$X$, then each of these vector spaces is endowed with a mixed Hodge structure.
		\item If~$X$ is a projective or affine space and the hypersurfaces in~$\L$ and~$\M$ are hyperplanes, then these mixed Hodge structures are of Tate type (all the weight-graded quotients are pure Tate structures).
		\item If furthermore all these hyperplanes are defined over a number field~$F\hookrightarrow \C$, then these mixed Hodge structures are the Hodge realizations~\cite{huberrealization, huberrealizationcorrigendum} of a mixed Tate motive over~$F$~\cite{levinetatemotives}. In this case, Theorem~\ref{introtheoremprojective} precisely describes the weight-graded pieces of these mixed Tate motives.
		\end{enumerate}
		
		It would be interesting to generalize our results to other settings, by working in Nori's tannakian category of motives, or the tannakian category of mixed Tate motives over any field for which the Beilinson--Soul\'e vanishing conjecture holds, etc.
		
	\subsection{Perspectives}
	
		The objects and techniques introduced in this article raise different questions for further research.
		\begin{enumerate}[--]
		\item Find an explicit combinatorial characterization of exact bi-arrangements.
		\item Find an explicit combinatorial presentation of the Orlik--Solomon bi-complex of a bi-arrangement, in the spirit of the presentation for tame bi-arrangements (Theorem~\ref{thmintrotame}). Find bases of the Orlik--Solomon bi-complex of a bi-arrangement in the spirit of nbc-bases of Orlik--Solomon algebras.
		\item Study the Orlik--Solomon bi-complex $A_{\bullet,\bullet}(\L,\M,\chi)$ as a module over the Orlik--Solomon algebras $A_\bullet(\L)$ and $A_\bullet(\M)$ (this is not  a general fact that the Orlik--Solomon bi-complex is a module over the Orlik--Solomon algebras, but it may happen in certain cases). In particular, relate homological properties of this module, such as Koszulness, to combinatorial properties of the bi-arrangement.
		\end{enumerate}
	
	\subsection{Connections with other articles}
	
		We are indebted to the work of A. B. Goncharov, in particular the ideas of~\cite{goncharovperiodsmm} which introduces the main objects of study of this article. One should be able to reconcile our strategy and Goncharov's strategy based on perverse sheaves using the Orlik--Solomon bi-complexes in the spirit of E. Looijenga's approach~\cite{looijenga} in the case of arrangements. 
		
		In~\cite{zhao}, J. Zhao introduces bi-complexes which should play the role of the Orlik--Solomon bi-complexes in the case of projective bi-arrangements (they cannot be compared to the Orlik--Solomon bi-complexes, since there is no coloring datum in \textit{loc. cit.}). Unfortunately, no connection is made between his combinatorial setting and the corresponding motives, except in the case of a generic bi-arrangement, i.e. a normal crossing divisor.
		
		In~\cite{dupontdissection}, we have already proved and used a very particular case of our main result in order to study a combinatorial family of periods.
	
	\subsection{Conventions and notations}
	
		\begin{enumerate}
		\item \textit{(Coefficients)} Unless otherwise stated, all vector spaces and algebras are defined over~$\Q$, as well as the tensor products of such objects. All (mixed) Hodge structures are defined over~$\Q$. All (relative) cohomology groups have coefficients in~$\Q$.
		\item \textit{(Tate twists)} We allow ourselves an abuse of notations with the Tate twists, writing like~$H^k(X)(-r)$ for~$X$ a complex manifold which is not necessarily a smooth algebraic variety. This is because Tate twists are important in the algebraic case; otherwise they should be ignored, and $H^k(X)(-r)$ should simply be interpreted as $H^k(X)$.
		\item \textit{(Homological algebra)} Our convention on bi-complexes is not standard since we mix the homological and the cohomological convention. A bi-complex is a collection of vector spaces~$C_{i,j}$ with differentials~$d':C_{i,j}\rightarrow C_{i-1,j}$ and~$d'':C_{i,j-1}\rightarrow C_{i,j}$ such that~$d'\circ d'=0$,~$d''\circ d''=0$ and~$d'\circ d''=d''\circ d'$. Our convention is to view the total complex~$C_n=\bigoplus_{i-j=n}C_{i,j}$ as a complex in the homological convention.
		\end{enumerate}
		
	\subsection{Outline of the paper}
	
		In \S\ref{sectionOS} we introduce the formalism of bi-arrangements and Orlik--Solomon bi-complexes as a generalization of the Orlik--Solomon algebra of an arrangement.
		
		In \S\ref{sectionhypersurfaces} we introduce bi-arrangements of hypersurfaces in a complex manifold. We define the motive of a bi-arrangement of hypersurfaces and study the behaviour of the Orlik--Solomon bi-complexes with respect to blow-up.
		
		In \S\ref{sectiongeometric} we define the geometric Orlik--Solomon bi-complex of a bi-arrangement of hypersurfaces, study its behaviour with respect to blow-up, and state the main theorem.
		
		In \S\ref{sectionprojective} we study the particular case of projective bi-arrangements of hyperplanes, with an application to multizeta bi-arrangements.
		
		In \S\ref{parproof}, which is the most technical part of this article, we prove the main theorem.
		
		
		Appendix~\ref{parappA} recalls some (more or less) classical facts on relative cohomology  in the case of normal crossing divisors.
		
		Appendix~\ref{parappB} is a collection of cohomological identities related to Chern classes and blow-ups. They are used in the proof of the main theorem.
	
	\subsection{Acknowledgements}
		This work was begun at the Institut de Math\'{e}matiques de Jussieu - Paris rive gauche (IMJ-PRG), Paris, France, during the author's PhD, and completed at the Max-Planck-Institut f\"{u}r Mathematik (MPIM), Bonn, Germany. We thank these institutes for their hospitality. Many thanks to Francis Brown for his many valuable comments on a preliminary version of this article. This work was partially supported by ERC grant 257638 ``Periods in algebraic geometry and physics".

\section{The Orlik--Solomon bi-complex of a bi-arrangement of hyperplanes}\label{sectionOS}

	\subsection{The Orlik--Solomon algebra of an arrangement of hyperplanes}
	
		Here we recall a few definitions and notations from the theory of arrangements of hyperplanes. We refer the reader to the classical book~\cite{orlikterao} for more details.
	
		\subsubsection{Definitions and notations}
	
			An \textit{arrangement of hyperplanes} (or simply an \textit{arrangement})~$\A$ in~$\C^n$ is a finite set of hyperplanes of~$\C^n$ that pass through the origin. Let us write~$\A=\{K_1,\ldots,K_k\}$. For~$i=1,\ldots,k$, we may write~$K_i=\{f_i=0\}$ where~$f_i$ is a non-zero linear form on~$\C^n$.
			
			If~$\A'$ is an arrangement in~$\C^{n'}$ and~$\A''$ is an arrangement in~$\C^{n''}$, then we may define their \textit{product}~$\A=\A'\times\A''$, which is the arrangement in~$\C^{n'+n''}$ consisting of the hyperplanes~$K'\times\C^{n''}$, for~$K'\in\A'$, and~$\C^{n'}\times K''$, for~$K''\in\A''$.  
			
			A \textit{stratum} of~$\A$ is an intersection~$K_I=\bigcap_{i\in I}K_i$ of some of the~$K_i$'s, for~$I\subset\{1,\ldots,k\}$. By convention, we have~$K_\varnothing=\C^n$, and all other strata are called \textit{strict}. We write~$\s_m(\A)$ for the set of strata of~$\A$ of codimension~$m$,~$\s(\A)=\bigsqcup_{m\geq 0}\s_m(\A)$ for the set of all strata of~$\A$ and~$\s_{+}(\A)=\bigsqcup_{m>0}\s_m(\A)$ for the set of strict strata of~$\A$.
			
			It is classical to view the set of strata as a poset ordered via reverse inclusion. For~$S$ a stratum of~$\A$, we write~$\A^{\leq S}$ for the arrangement consisting of the hyperplanes that contain~$S$. 
			
			Let us write~$S^\perp\subset (\C^n)^\vee$ for the space of linear forms on~$\C^n$ that vanish on a stratum~$S$; it is spanned by the~$f_i$'s for~$i$ such that~$S\subset K_i$. We say that a family of strata~$S_1,\ldots,S_r$ \textit{intersect transversely} and write~$S_1\pitchfork\cdots\pitchfork S_r$ if~$S_1^\perp,\ldots,S_r^\perp$ are in direct sum in~$\C^n$.
			
			If~$S$ is a stratum of~$\A$, a \textit{decomposition} of~$S$ is an equality~$S=S_1\pitchfork\cdots\pitchfork S_r$ with the~$S_j$'s strata of~$\A$, and such that for every hyperplane~$K_i$ that contains~$S$,~$K_i$ contains some~$S_j$. Dually, this amounts to saying that we may write~$S^\perp=(S_1)^\perp\oplus \cdots\oplus (S_r)^\perp$ such that every~$f_i\in S^\perp$ is in some~$(S_j)^\perp$. Equivalently, we have a product decomposition~$\A^{\leq S}\cong \A^{\leq S_1}\times\cdots\times\A^{\leq S_r}$. We say that~$S$ is \textit{reducible} if it has a non-trivial decomposition, i.e. with all~$S_j$'s strict strata, and \textit{irreducible} otherwise. Every~$K\in\A$ is irreducible. A stratum~$S$ has a unique decomposition~$S=S_1\pitchfork \cdots\pitchfork S_r$ with the~$S_j$'s irreducible.

		\subsubsection{The Orlik--Solomon algebra}
		
			Let~$\A=\{K_1,\ldots,K_k\}$ be an arrangement of hyperplanes in~$\C^n$. We let~$E_\bullet(\A)=\Lambda^\bullet(e_1,\ldots,e_k)$ be the exterior algebra on generators~$e_i$,~$i=1,\ldots,k$ in degree~$1$. For~$I=\{i_1<\cdots<i_r\}\subset\{1,\ldots,k\}$ we write~$e_I=e_{i_1}\wedge\cdots\wedge e_{i_r}$ for the corresponding basis element of~$E_r(\A)$, with the convention~$e_\varnothing=1$. Let~$d:E_\bullet(\A)\rightarrow E_{\bullet-1}(\A)$ be the unique derivation of~$E_\bullet(\A)$ such that~$d(e_i)=1$ for all~$i$. It is given by 
			$$d(e_{i_1}\wedge\cdots\wedge e_{i_r})=\sum_{j=1}^r (-1)^{j-1}e_{i_1}\wedge\cdots\wedge\widehat{e_{i_j}}\wedge\cdots\wedge e_{i_r}.$$
			A subset~$I\subset\{1,\ldots,k\}$ is said to be \textit{dependent} if the hyperplanes~$K_i$, for~$i\in I$, are linearly dependent, and \textit{independent} otherwise. A \textit{circuit} of~$\A$ is a minimally dependent subset. Let~$R_\bullet(\A)$ be the homogeneous ideal of~$E_\bullet(\A)$ generated by the elements~$d(e_I)$ for~$I$ dependent. The Leibniz rule implies that it is generated by the elements~$d(e_I)$ for~$I$ a circuit.
			
			The \textit{Orlik--Solomon algebra} of~$\A$ is the quotient~$A_\bullet(\A)=E_\bullet(\A)/R_\bullet(\A)$. It is a differential graded algebra which is easily seen to be exact if~$\A$ is non-empty, a contracting homotopy~$h:A_\bullet(\A)\rightarrow A_{\bullet+1}(\A)$ being given by~$h(x)=e_1\wedge x$. An important feature of the Orlik--Solomon algebra is the following direct sum decomposition with respect to the set of strata:
			$$A_r(\A)=\bigoplus_{S\in\s_r(\A)}A_r^S(\A)$$
			where~$A_r^S(\A)$ is spanned by the classes of the elements~$e_I$ for~$I$ such that~$K_I=S$. 
			
			We will write~$S\stackrel{m}{\hookrightarrow}T$ for an inclusion of strata of codimension~$m$; for an inclusion~$S\stackrel{1}{\hookrightarrow}T$, we then have a component~$d_{S,T}:A_r^S(\A)\rightarrow A_{r-1}^T(\A)$ for the differential~$d$.
			
			If~$\Sigma$ is a strict stratum of~$\A$ of codimension~$r$, then the complex
			$$0\rightarrow A_r^\Sigma(\A) \stackrel{d}{\longrightarrow} \bigoplus_{\Sigma\stackrel{1}{\hookrightarrow}S}A_{r-1}^S(\A) \stackrel{d}{\longrightarrow} \bigoplus_{\Sigma\stackrel{2}{\hookrightarrow}T}A_{r-2}^T(\A) \stackrel{d}{\longrightarrow}\cdots\stackrel{d}{\longrightarrow} A_0^{\C^n}(\A)\rightarrow 0$$
			is the Orlik--Solomon algebra of the arrangement~$\A^{\leq\Sigma}$, hence is exact. This property allows one to uniquely define (as in~\cite[Lemma 2.2]{looijenga}) the groups~$A_r^S(\A)$ and the differentials~$d_{S,T}$ by induction on the codimension, starting with~$A_0^{\C^n}(\A)=\Q$. We will use this inductive point of view to generalize this construction to bi-arrangements.
		
	\subsection{Bi-arrangements of hyperplanes}
	
		
			\begin{defi}\label{defibiarrangement}
			A \textit{bi-arrangement of hyperplanes} (or simply a \textit{bi-arrangement})~$\B=(\A,\chi)$ in~$\C^n$ is the data of an arrangement of hyperplanes~$\A$ in~$\C^n$ along with a \textit{coloring function}
			$$\chi:\s_+(\A)\rightarrow \{\lambda,\mu\}$$
			on the strict strata of~$\A$, such that the \textit{K\"{u}nneth condition} is satisfied:
			\begin{equation}\label{kunnethcondition}
			\textnormal{for any non-trivial decomposition~$S=S'\pitchfork S''$,~$\chi(S)=\chi(S')$ or~$\chi(S)=\chi(S'')$.}
			\end{equation}
			\end{defi}
		
			\begin{rem}
			The K\"{u}nneth condition is trivially satisfied if~$\chi(S')\neq\chi(S'')$. More generally, let~$S=S_1\pitchfork\ldots\pitchfork S_r$ be the decomposition of a strict stratum~$S$ into irreducible strata~$S_k$. If~$\chi(S_1)=\cdots=\chi(S_r)$, then the K\"{u}nneth condition forces~$\chi(S)=\chi(S_1)=\cdots=\chi(S_r)$. Otherwise,~$\chi(S)$ is not constrained by the definition of a bi-arrangement of hyperplanes. To sum up, a coloring function that satisfies the K\"{u}nneth condition is uniquely determined by
			\begin{enumerate}[--]
			\item the colors of the irreducible strata;
			\item the colors of the strata~$S=S_1\pitchfork\cdots\pitchfork S_r$ with the~$S_k$'s irreducible which do not all have the same color.
			\end{enumerate}
			For all our purposes, only the colors of the irreducible strata will matter, thus we make the following definition.
			
			\begin{defi}\label{defiequivalence}
			Two bi-arrangements are \textit{equivalent} if their underlying arrangements are the same and if their coloring functions agree on the irreducible strata. 
			\end{defi}
			In most of the article, we will implicitly consider bi-arrangements up to this equivalence relation. In particular, we will allow ourselves to define a bi-arrangement by only specifying the colors of the irreducible strata.
			\end{rem}
			
			\begin{rem}			
			The hyperplanes~$L\in\A$ such that~$\chi(L)=\lambda$ (resp. the hyperplanes~$M\in\A$ such that~$\chi(M)=\mu$) form an arrangement denoted by~$\mathscr{L}$ (resp.~$\mathscr{M}$). In most geometric situations (see \S\ref{parintroperiods}) these two arrangements play very different roles, hence the union~$\A=\L\sqcup\M$ is an artificial object. In other words, one should not view a bi-arrangement as an arrangement with some coloring datum, but as \textit{two} arrangements with some coloring datum. To emphasize this point, we will use the following notational conventions.
			\end{rem}
			
			\begin{notation}\label{notationconvention1}
			We will sometimes denote a bi-arrangement~$\B$ in~$\C^n$ by a triple~$(\L,\M,\chi)$, where~$\L$ and~$\M$ are two disjoint arrangements in~$\C^n$, and~$\chi:\s_+(\L\sqcup\M)\rightarrow\{\lambda,\mu\}$ is a function that satisfies~$\chi(L)=\lambda$ for~$L\in\L$,~$\chi(M)=\mu$ for~$M\in\M$, and the K\"{u}nneth condition (\ref{kunnethcondition}).
			\end{notation}
		
			\begin{notation}\label{notationconvention2}
			For~$\B=(\A,\chi)$ a bi-arrangement, we will often forget the underlying arrangement~$\A$ and simply denote it by~$\B$ instead. We will then write~$K\in\B$ for~$K\in\A$,~$S\in\s(\B)$ for~$S\in\s(\A)$, and so on.
			\end{notation}
			
			We will make great use of a natural involution on bi-arrangements.
					
			\begin{defi}
			The \textit{dual} of a bi-arrangement~$\B=(\A,\chi)$ is the bi-arrangement~$\B^\vee=(\A,\chi^\vee)$ where~$\chi^\vee$ is the composition of~$\chi$ with the involution~$\lambda\leftrightarrow\mu$. Equivalently, the dual of~$\B=(\L,\M,\chi)$ is~$\B^\vee=(\M,\L,\chi^\vee)$. We have~$\left(\B^\vee\right)^\vee=\B$.
			\end{defi}
			
			We may also take product of bi-arrangements. This operation is only well-defined if we work up to equivalence (Definition~\ref{defiequivalence}).
			
			\begin{defi}
			If~$\B'=(\A',\chi')$ is a bi-arrangement of hyperplanes in~$\C^{n'}$ and~$\B''=(\A'',\chi'')$ is a bi-arrangement of hyperplanes in~$\C^{n''}$, then we define their \textit{product}~$\B=\B'\times\B''=(\A,\chi)$, whose underlying arrangement of hyperplanes is~$\A=\A'\times \A''$. Its irreducible strata have the form~$S'\times\C^{n''}$ or~$\C^{n'}\times S''$ for~$S'$ (resp.~$S''$) an irreducible stratum of~$\A'$ (resp.~$\A''$). We thus define the coloring by~$\chi(S'\times\C^{n''})=\chi'(S')$ and~$\chi(\C^{n'}\times S'')=\chi''(S'')$.
			\end{defi}

		
			\begin{ex} 
			There are two (dual) ways in which an arrangement~$\A$ may be viewed as a bi-arrangement: by defining the coloring~$\chi$ to be constant equal to~$\lambda$ or~$\mu$. We will simply denote these bi-arrangements by~$(\A,\lambda)$ and~$(\A,\mu)$.
			\end{ex}
			
			\begin{ex}
			 By taking products, we may define bi-arrangements~$(\L,\lambda)\times(\M,\mu)$. They are somewhat trivial examples since the arrangements~$\L$ and~$\M$ \enquote{do not mix}. 
			 \end{ex}
			
			
			\begin{ex}\label{exextremecoloring} Let~$\L$ and~$\M$ be two disjoint arrangements in~$\C^n$. We define the \textit{$\lambda$-extreme} coloring~$e_\lambda$ and the \textit{$\mu$-extreme} coloring~$e_\mu$ so that~$(\L,\M,e_\lambda)$ and~$(\L,\M,e_\mu)$ are bi-arrangements.
			\begin{eqnarray*}
			e_\lambda(S)=\begin{cases} 
			\lambda & \textnormal{ if } S\subset L \textnormal{ for some } L\in\L  \\
			\mu & \textnormal{ otherwise.}
			\end{cases} & \hspace{.5cm} &
			e_\mu(S)=\begin{cases} 
			\mu & \textnormal{ if } S\subset M \textnormal{ for some } M\in\M\\
			\lambda & \textnormal{ otherwise.}
			\end{cases}
			\end{eqnarray*}
			To understand the terminology, let us anticipate and note (see for instance Lemma~\ref{lemfirstobstruction} below) that we will be interested mostly in the bi-arrangements such that for every stratum~$S$, there exists a hyperplane~$K\supset S$ with the same color as~$S$. The~$\lambda$-extreme coloring (resp. the~$\mu$-extreme coloring) is extreme in the sense that we give the color~$\lambda$ (resp. the color~$\mu$) to as many strata as possible while staying in that class of bi-arrangements.
			\end{ex}
	
	\subsection{The formalism of Orlik--Solomon bi-complexes}
	
		\subsubsection{The definition}	
	
		\begin{lem}\label{defiOSbicomplex}
		Let~$\B$ be a bi-arrangement in~$\C^n$. There exists a unique datum of
		\begin{enumerate}[--]
		\item for all~$i, j\geq 0$, for every stratum~$S\in\s_{i+j}(\B)$, a finite-dimensional~$\Q$-vector space~$A_{i,j}^S$;
		\item for every inclusion~$S\stackrel{1}{\hookrightarrow}T$ of strata of codimension~$1$, linear maps
		$$d'_{S,T}:A_{i,j}^S\rightarrow A_{i-1,j}^T \;\;\textnormal{ and }\;\; d''_{S,T}:A_{i,j-1}^T\rightarrow A_{i,j}^S;$$
		\end{enumerate}		 
		such that the following conditions are satisfied:
		\begin{enumerate}[--]
		\item~$A_{0,0}^{\C^n}=\Q$;
		\item for every stratum~$\Sigma$,~$$A_{\bullet,\bullet}^{\leq\Sigma}=\left(\bigoplus_{S\supset\Sigma}A_{\bullet,\bullet}^S,d',d''\right)$$ is a bi-complex, where~$d'$ and~$d''$ respectively denote the collection of the maps~$d'_{S,T}$ and~$d''_{S,T}$ for~$S\supset\Sigma$;
		\item for every strict stratum~$\Sigma\in\s_{i+j}(\B)$ such that~$\chi(\Sigma)=\lambda$, we have exact sequences
		$$0\rightarrow A_{i,j}^\Sigma \stackrel{d'}{\longrightarrow} \bigoplus_{\Sigma\stackrel{1}{\hookrightarrow}S}A_{i-1,j}^S \stackrel{d'}{\longrightarrow} \bigoplus_{\Sigma\stackrel{2}{\hookrightarrow}T}A_{i-2,j}^T;$$
		\item for every strict stratum~$\Sigma\in\s_{i+j}(\B)$ such that~$\chi(\Sigma)=\mu$, we have exact sequences
		$$0\leftarrow A_{i,j}^\Sigma \stackrel{d''}{\longleftarrow} \bigoplus_{\Sigma\stackrel{1}{\hookrightarrow}S}A_{i,j-1}^S \stackrel{d''}{\longleftarrow} \bigoplus_{\Sigma\stackrel{2}{\hookrightarrow}T}A_{i,j-2}^T.$$
		\end{enumerate}
		\end{lem}
		
		\begin{proof}
		We define the bi-complexes~$A_{\bullet,\bullet}^{\leq\Sigma}$ by induction on the codimension of~$\Sigma$. The case of codimension~$0$ is given by definition. If~$\Sigma$ is a strict stratum and~$\chi(\Sigma)=\lambda$ then one is forced to define
		$$A_{i,j}^\Sigma=\mathrm{ker}\left(\bigoplus_{\Sigma\stackrel{1}{\hookrightarrow}S}A_{i-1,j}^S \stackrel{ d'}{\rightarrow} \bigoplus_{\Sigma\stackrel{2}{\hookrightarrow}T}A_{i-2,j}^T\right)$$
		and the differentials~$ d'_{\Sigma,S}:A_{i,j}^\Sigma\rightarrow A_{i-1,j}^S$ to be the components of the natural inclusion.
		This uniquely defines the differentials~$ d''_{\Sigma,S}:A_{i,j-1}^S\rightarrow A_{i,j}^\Sigma$ by filling the dotted arrow in the following commutative diagram.
		$$\xymatrix{
		\displaystyle\bigoplus_{\Sigma\stackrel{1}{\hookrightarrow}S} A_{i,j-1}^S \ar@{.>}[d] \ar[r] & \displaystyle\bigoplus_{\Sigma\stackrel{2}{\hookrightarrow}T} A_{i-1,j-1}^T \ar[d]\ar[r] & \displaystyle\bigoplus_{\Sigma\stackrel{3}{\hookrightarrow}U} A_{i-2,j-1}^U\ar[d]\\
		A_{i,j}^\Sigma \ar[r]& \displaystyle\bigoplus_{\Sigma\stackrel{1}{\hookrightarrow}S} A_{i-1,j}^S \ar[r]& \displaystyle\bigoplus_{\Sigma\stackrel{2}{\hookrightarrow}T} A_{i-2,j}^T}$$
		The case~$\chi(\Sigma)=\mu$ is dual, with the definition
		$$A_{i,j}^{\Sigma}=\mathrm{coker}\left( \bigoplus_{\Sigma\stackrel{2}{\hookrightarrow}T}A_{i,j-2}^T \stackrel{ d''}{\rightarrow} \bigoplus_{\Sigma\stackrel{1}{\hookrightarrow}S}A_{i,j-1}^S\right).$$
		\end{proof}

		\begin{defi}
		The above datum is called the \textit{Orlik--Solomon bi-complex} of the bi-arrangement~$\B$ and denoted by~$A_{\bullet,\bullet}(\B)$, or simply~$A_{\bullet,\bullet}$ when the situation is clear.
		\end{defi}
		
		Visually, we get a bi-complex that is defined inductively, starting in the top right corner and going in the bottom left direction.
		
		$$\xymatrix{
		 \ar@{.>}[r] & A_ {3,0}\ar@{.>}[d] \ar[r] & A_{2,0} \ar[r]\ar[d] & A_{1,0} \ar[r]\ar[d] & A_{0,0} \ar[d] \\
		 &\ar@{.>}[r]&  A_{2,1} \ar@{.>}[d] \ar[r] & A_{1,1}\ar[d] \ar[r] & A_{0,1}\ar[d]  \\
		& & \ar@{.>}[r] & A_{1,2} \ar@{.>}[d]\ar[r] & A_{0,2} \ar[d] \\
		 &&& \ar@{.>}[r] & A_{0,3}\ar@{.>}[d] \\
		 &&&&
		}$$
		
		\begin{rem}
		The Orlik--Solomon bi-complex is a local object: the bi-complex~$A_{\bullet,\bullet}^{\leq\Sigma}(\B)$ is the Orlik--Solomon bi-complex of the bi-arrangement~$\B^{\leq\Sigma}$ consisting of the hyperplanes that contain~$\Sigma$.
		\end{rem}
		
		\begin{lem}\label{lemAbicomplex}
		Let~$\B$ be a bi-arrangement and~$A_{\bullet,\bullet}$ be its Orlik--Solomon bi-complex. The fact that all~$A_{\bullet,\bullet}^{\leq \Sigma}$ are bi-complexes may be translated explicitly into the following identities.
		\begin{enumerate}
		\item For an inclusion~$S\stackrel{2}{\hookrightarrow}U$ we have
		$$\sum_{S\stackrel{1}{\hookrightarrow}T\stackrel{1}{\hookrightarrow}U}d'_{T,U}\circ d'_{S,T}=0 \;\;\textit{ and }\;\, \sum_{S\stackrel{1}{\hookrightarrow}T\stackrel{1}{\hookrightarrow}U}d''_{S,T}\circ d''_{T,U}=0.$$
		\item \begin{enumerate} 
			\item Let~$S\neq U$ be two strata of the same codimension such that there is no diagram~$S\stackrel{1}{\hookrightarrow} T\stackrel{1}{\hookleftarrow} U$. Then for every diagram~$S\stackrel{1}{\hookleftarrow} R\stackrel{1}{\hookrightarrow} U$ we have 
			$$d'_{R,U}\circ d''_{R,S}=0.$$
			\item Let~$S\neq U$ be two strata of the same codimension such that there is a diagram~$S\stackrel{1}{\hookrightarrow} T\stackrel{1}{\hookleftarrow} U$. Then we necessarily have~$T=S+U$ and there is a unique diagram~$S\stackrel{1}{\hookleftarrow} R\stackrel{1}{\hookrightarrow} U$ , which is~$R=S\cap U$. We then have 
			$$d'_{R,U}\circ d''_{R,S}=d''_{U,T}\circ d'_{S,T}.$$
			\item For every stratum~$S$, we have 
			$$\sum_{S\stackrel{1}{\hookrightarrow}T}d''_{S,T}\circ d'_{S,T}=0.$$
			For every inclusion~$R\stackrel{1}{\hookrightarrow} S$ we have
			$$ d'_{R,S}\circ d''_{R,S}=0.$$
			\end{enumerate}
		\end{enumerate}
		\end{lem}		
		
		\begin{proof}
		\begin{enumerate}
		\item It expresses the fact that~$d'\circ d'=0$ and~$ d''\circ d''=0$ in~$A^{\leq S}_{\bullet,\bullet}$.
		\item \begin{enumerate} 
			\item It expresses the fact that the components~$A_{i,j-1}^S\rightarrow A_{i-1,j}^U$ of~$ d'\circ d''$ and~$ d''\circ d'$ in~$A^{\leq R}_{\bullet,\bullet}$ are equal.
			\item Same.
			\item There is no~$ d''_{R,S}$ in~$A^{\leq S}_{\bullet,\bullet}$, hence the component~$A_{i,j-1}^S\rightarrow A_{i-1,j}^S$ of~$ d''\circ d'$ is zero. This gives the first equality. Now for some~$R\stackrel{1}{\hookrightarrow} S$, the second equality follows from the first equality and the fact that the components~$A_{i,j-1}^S\rightarrow A_{i-1,j}^S$ of~$ d'\circ d''$ and~$ d''\circ d'$ in~$A_{\bullet,\bullet}^{\leq R}$ are equal.
			\end{enumerate}
		\end{enumerate}
		\end{proof}
		
		\begin{defi}\label{defiexact}
		Let~$\B$ be an arrangement and~$A_{\bullet,\bullet}$ be its Orlik--Solomon bi-complex. We say that a strict stratum~$\Sigma$ of~$\B$ is \textit{exact} if the following condition, depending on the color of~$\Sigma$, is satisfied:
		\begin{enumerate}[--]
		\item~$\chi(\Sigma)=\lambda$ and all the rows 
		$$0\rightarrow A_{i,j}^\Sigma \stackrel{d'}{\longrightarrow} \bigoplus_{\Sigma\stackrel{1}{\hookrightarrow}S}A_{i-1,j}^S \stackrel{d'}{\longrightarrow} \bigoplus_{\Sigma\stackrel{2}{\hookrightarrow}T}A_{i-2,j}^T \stackrel{d'}{\longrightarrow}\cdots\stackrel{d'}{\longrightarrow} \bigoplus_{\Sigma\stackrel{i}{\hookrightarrow}Z}A_{0,j}^Z \rightarrow 0$$
		of the bi-complex~$A_{\bullet,\bullet}^{\leq\Sigma}$ are exact;
		\item~$\chi(\Sigma)=\mu$ and all the columns 
		$$0\leftarrow A_{i,j}^\Sigma \stackrel{d''}{\longleftarrow} \bigoplus_{\Sigma\stackrel{1}{\hookrightarrow}S}A_{i,j-1}^S \stackrel{d''}{\longleftarrow} \bigoplus_{\Sigma\stackrel{2}{\hookrightarrow}T}A_{i,j-2}^T \stackrel{d''}{\longleftarrow}\cdots\stackrel{d''}{\longleftarrow} \bigoplus_{\Sigma\stackrel{j}{\hookrightarrow}Z}A_{i,0}^Z \leftarrow 0$$
		of the bi-complex~$A_{\bullet,\bullet}^{\leq\Sigma}$ are exact.
		\end{enumerate}
		We say that~$\B$ is \textit{exact} if all its strict strata are exact.
		\end{defi}
		
		The next easy lemma expresses the fact that the definition of the Orlik--Solomon bi-complex is self-dual.
		
		\begin{lem}
		The Orlik--Solomon bi-complexes of~$\B$ and~$\B^\vee$ are dual to each other: we have~$A_{i,j}^S(\B^\vee)=\left(A_{j,i}^S(\B)\right)^\vee$,~$ d'$ being the transpose of~$ d''$ and~$ d''$ the transpose of~$ d'$.~$\B$ is exact if and only if~$\B^\vee$ is exact.
		\end{lem}
		
		\subsubsection{The K\"{u}nneth formula}
		
		Up to now, we haven't used the K\"{u}nneth condition (\ref{kunnethcondition}). This condition is actually crucial since it implies that the Orlik--Solomon bi-complexes behave well with respect to decompositions.
		
		\begin{prop}\label{propkunneth}
		Let~$\B$ be a bi-arrangement and~$\Sigma$ a stratum of~$\B$. Let us assume that~$\Sigma$ has a decomposition~$\Sigma=\Sigma'\pitchfork\Sigma''$. Then we have an isomorphism of bi-complexes (\enquote{K\"{u}nneth formula})
		$$A_{\bullet,\bullet}^{\leq\Sigma}\cong A_{\bullet,\bullet}^{\leq\Sigma'}\otimes A_{\bullet,\bullet}^{\leq\Sigma''}.$$
		More precisely, a stratum~$S\supset\Sigma$ of codimension~$r$ has a unique decomposition~$S=S'\pitchfork S''$ with~$S'\supset\Sigma'$ of codimension~$r'$ and~$S''\supset\Sigma''$ of codimension~$r''$ with~$r=r'+r''$; we then have isomorphisms
		$$A_{i,j}^{S}\cong \bigoplus A^{S'}_{i',j'}\otimes A^{S''}_{i'',j''}$$
		where the sum is over the indices such that~$i'+i''=i$,~$ j'+j''=j$,~$i'+j'=r'$,~$i''+j''=r''$. These isomorphisms are compatible with differentials in the sense that the horizontal (resp. vertical) differential on the left-hand side equals $d'\otimes\mathrm{id}+(-1)^{i'}\mathrm{id}\otimes d'$ (resp. $d''\otimes\mathrm{id}+(-1)^{j'}\mathrm{id}\otimes d''$).
		\end{prop}
		
		\begin{proof}
		We proceed by induction on the codimension of~$\Sigma$. The case of codimension~$0$ is just the isomorphism~$\Q\cong\Q\otimes\Q$. More generally, the result is trivial if~$\Sigma'$ or~$\Sigma''$ is the whole space~$\C^n$. We thus assume that~$\Sigma'$ and~$\Sigma''$ are strict strata. Let us assume that~$\chi(\Sigma)=\lambda$, the case~$\chi(\Sigma)=\mu$ being dual. Then by the K\"{u}nneth condition (\ref{kunnethcondition}), we necessarily have~$\chi(\Sigma')=\lambda$ or~$\chi(\Sigma'')=\lambda$. We consider the complexes
		\begin{equation}\label{eqtrunc1}
		0\rightarrow A_{i',j'}^{\Sigma'} \stackrel{d'}{\rightarrow} \bigoplus_{\Sigma'\stackrel{1}{\hookrightarrow}S'}A_{i'-1,j'}^{S'} \stackrel{d'}{\rightarrow} \bigoplus_{\Sigma'\stackrel{2}{\hookrightarrow}T'}A_{i'-2,j'}^{T'}
		\end{equation}
		and
		\begin{equation}\label{eqtrunc2}
		0\rightarrow A_{i'',j''}^{\Sigma''} \stackrel{d'}{\rightarrow} \bigoplus_{\Sigma''\stackrel{1}{\hookrightarrow}S''}A_{i''-1,j''}^{S''} \stackrel{d'}{\rightarrow} \bigoplus_{\Sigma''\stackrel{2}{\hookrightarrow}T''}A_{i''-2,j''}^{T''}.
		\end{equation}
		The tensor product of these two complexes is necessarily exact, since one of the two is exact. Summing over all possible indices~$(i', j', i'', j'')$ and using the induction hypothesis leads to an exact complex
		\begin{align*}
		0\rightarrow \bigoplus A^{\Sigma'}_{i',j'}\otimes A^{\Sigma''}_{i'',j''}  &\rightarrow \bigoplus_{\Sigma'\stackrel{1}{\hookrightarrow}S'}A_{i-1,j}^{S'\pitchfork\Sigma''}\oplus\bigoplus_{\Sigma''\stackrel{1}{\hookrightarrow}S''} A_{i-1,j}^{\Sigma'\pitchfork S''}\rightarrow  \\
		&\rightarrow \bigoplus_{\Sigma'\stackrel{2}{\hookrightarrow}T'}A_{i-2,j}^{T'\pitchfork\Sigma''}\oplus\bigoplus_{\substack{\Sigma'\stackrel{1}{\hookrightarrow}S'\\\Sigma''\stackrel{1}{\hookrightarrow} S''}} A_{i-2,j}^{S'\pitchfork S''}\oplus \bigoplus_{\Sigma''\stackrel{2}{\hookrightarrow}T''} A_{i-2,j}^{\Sigma'\pitchfork T''}.
		\end{align*}
		This gives the desired isomorphism. One easily checks the compatibilities with the differentials.
		\end{proof}

		
		\begin{coro}\label{coroOSequivalence}
		\begin{enumerate}
		\item The Orlik--Solomon bi-complex~$A_{\bullet,\bullet}(\B)$ of a bi-arrangement~$\B$ only depends on its equivalence class (Definition~\ref{defiequivalence}).
		\item A bi-arrangement~$\B$ is exact if and only if all its irreducible strata of codimension~$\geq 2$ are exact. Thus, the exactness of~$\B$ only depends on its equivalence class.
		\end{enumerate}
		\end{coro}
		
		\begin{proof}
		\begin{enumerate}
		\item Proposition~\ref{propkunneth} implies that for a decomposition into irreducibles~$S=S_1\pitchfork\cdots\pitchfork S_r$,~$A_{\bullet,\bullet}^{\leq S}$ is the tensor product of the bi-complexes~$A_{\bullet,\bullet}^{\leq S_k}$, hence it does not depend on the color~$\chi(S)$.
		\item Let us assume that all the~$S_k$'s are exact, and that~$\chi(S)=\lambda$ (the case~$\chi(S)=\mu$ being dual). By definition, we may then assume that~$\chi(S_1)=\lambda$, and hence the rows of~$A_{\bullet,\bullet}^{\leq S_1}$ are exact. The K\"{u}nneth formula implies that the rows of~$A_{\bullet,\bullet}^{\leq S}$ are exact, hence~$S$ is exact. The claim then follows from the fact that all hyperplanes~$K\in\B$ are exact.
		\end{enumerate}
		\end{proof}
		
		Another way of stating the K\"{u}nneth formula is the following.
		
		\begin{coro}\label{coroproductexact}
		The Orlik--Solomon bi-complex of a product~$\B'\times\B''$ is the tensor product
		$$A_{\bullet,\bullet}(\B'\times\B'')\cong A_{\bullet,\bullet}(\B')\otimes A_{\bullet,\bullet}(\B'').$$
		Furthermore,~$\B'\times \B''$ is exact if and only if~$\B'$ and~$\B''$ are exact.
		\end{coro}

		\subsubsection{Examples}
		
		\begin{ex}\label{exex}
		 The notion of an Orlik--Solomon bi-complex generalizes the construction of the Orlik--Solomon algebra. Indeed, if~$\A$ is an arrangement then the Orlik--Solomon bi-complex of the bi-arrangement~$(\A,\lambda)$ is concentrated in bi-degrees~$(k,0)$ and agrees with the Orlik--Solomon algebra of~$\A$:~$A_{k,0}^S(\A,\lambda)=A_k^S(\A)$ for all~$S\in\s_k(\A)$, and~$d'_{S,T}=d_{S,T}$ the classical differential of the Orlik--Solomon algebra. Dually, the Orlik--Solomon bi-complex of~$(\A,\mu)$ is concentrated in bi-degrees~$(0,k)$ and is the linear dual of the Orlik--Solomon algebra of~$\A$:~$A_{0,k}^S(\A,\mu)=\left(A_k^S(\A)\right)^\vee$. The bi-arrangements~$(\A,\lambda)$ and~$(\A,\mu)$ are thus always exact.
		 \end{ex}

		\begin{ex}
		More generally, for a bi-arrangement~$\B=(\L,\M,\chi)$, if all strata of~$\L$ are colored~$\lambda$ then we have an isomorphism~$A_{\bullet,0}(\L,\M,\chi)\cong A_\bullet(\L)$. Dually, if all strata of~$\M$ are colored~$\mu$ then we have an isomorphism~$A_{0,\bullet}(\L,\M,\chi)\cong \left(A_\bullet(\M)\right)^\vee$.
		\end{ex}		 
		 
		\begin{ex} By Example~\ref{exex} and Corollary~\ref{coroproductexact}, a product~$(\L,\lambda)\times (\M,\mu)$ is always exact, with its Orlik--Solomon bi-complex
		$$A_{\bullet,\bullet}((\L,\lambda)\times(\M,\mu))\cong A_\bullet(\L)\otimes(A_\bullet(\M))^\vee.$$
		\end{ex}
		
		\subsubsection{The first obstruction to exactness}		
		
		Let~$\B=(\L,\M,\chi)$ be a bi-arrangement. By the definition of an Orlik--Solomon bi-complex, we have for each~$L\in\L$ an isomorphism~$A_{1,0}^{L}\stackrel{\cong}{\rightarrow }\Q$, and~$A_{0,1}^{L}=0$. Dually, we get for each~$M\in\M$ an isomorphism~$\Q\stackrel{\cong}{\rightarrow} A_{0,1}^{M}$, and~$A_{1,0}^{M}=0$. This remark gives us the first obstruction to the exactness of a bi-arrangement.
		
		\begin{lem}\label{lemfirstobstruction}
		If a bi-arrangement~$\B=(\L,\M,\chi)$ is exact, then for every strict stratum~$S$, 
		\begin{enumerate}
		\item if~$\chi(S)=\lambda$ then~$S\subset L$ for some~$L\in\L$;
		\item if~$\chi(S)=\mu$ then~$S\subset M$ for some~$M\in\M$.
		\end{enumerate}
		\end{lem}
		
		\begin{proof}
		Let us assume that~$\chi(S)=\lambda$, the case~$\chi(S)=\mu$ being dual. Then the first row of the bi-complex~$A_{\bullet,\bullet}^{\leq S}$ is exact, which means that we have a surjection
		$$\bigoplus_{L\in\L \;|\;S\subset L}A_{1,0}^{L}\rightarrow \Q\rightarrow 0$$
		hence~$S\subset L$ for some~$L\in\L$.
		\end{proof}
		
		\begin{ex}
		The simplest bi-arrangement of hyperplanes that is not exact is made of three lines~$L_1,L_2,L_3$ in~$\C^2$ that meet at the origin~$Z$, with~$\chi(L_1)=\chi(L_2)=\chi(L_3)=\lambda$, and~$\chi(Z)=\mu$.
		\end{ex}
	
	\subsection{The Orlik--Solomon bi-complex of a tame bi-arrangement}\label{partame}
	
		\subsubsection{Tame bi-arrangements}
	
		Let~$\L=\{L_1,\ldots,L_l\}$ and~$\M=\{M_1,\ldots,M_m\}$ be two arrangements of hyperplanes in~$\C^n$. We say that a pair~$(I,J)$ formed by a subset~$I\subset\{1,\ldots,l\}$ and a subset~$J\subset\{1,\ldots,m\}$ is \textit{dependent} if the hyperplanes~$L_i$, for~$i\in I$, and~$M_j$, for~$j\in J$, are linearly dependent, and \textit{independent} otherwise. A \textit{circuit} is a minimally dependent pair~$(I,J)$ in the sense that if~$I'\subset I$ and~$J'\subset J$ are two subsets such that~$(I',J')$ is dependent, then~$I'=I$ and~$J'=J$. We note that if~$(I,J)$ is a circuit, then~$L_I\cap M_J$ is an irreducible stratum.
		
		\begin{defi}\label{defitame}
		Let~$\B=(\L,\M,\chi)$ be a bi-arrangement. A strict stratum~$S$ of~$\B$ is \textit{tame} if the following condition, depending on the color of~$S$, is satisfied:
		\begin{enumerate}
		\item~$\chi(S)=\lambda$ and there exists a hyperplane~$L_i$ that contains~$S$ and such that~$i$ does not belong to any circuit~$(I,J)$ with~$S\subset L_I\cap M_J$ and~$\chi(L_I\cap M_J)=\mu$;
		\item~$\chi(S)=\mu$ and there exists a hyperplane~$M_j$ that contains~$S$ and such that~$j$ does not belong to any circuit~$(I,J)$ with~$S\subset L_I\cap M_J$ and~$\chi(L_I\cap M_J)=\lambda$.
		\end{enumerate}
		A bi-arrangement of hyperplanes is \textit{tame} if all its strict strata are tame.
		\end{defi}
		
		\begin{rem}\label{remtamelocal}
		The tameness is a local condition in the sense that the tameness of a stratum~$S$ of~$\B$ only depends on the bi-arrangement~$\B^{\leq S}$ consisting of the hyperplanes that contain~$S$.
		\end{rem}
		
		\begin{lem}\label{lemtameequivalence}
		A bi-arrangement is tame if and only if all its irreducible strata of codimension~$\geq 2$ are tame. Thus, the tameness of a bi-arrangement only depends on its equivalence class.
		\end{lem}
		
		\begin{proof}
		We note that the hypersurfaces~$K\in\B$ are necessarily tame. Let us assume that all irreducible strata of~$\B$ are tame. Let~$S$ be a reducible stratum of~$\B$ with a decomposition~$S=S_1\pitchfork\cdots \pitchfork S_r$ into irreducibles~$S_j$. Let us assume that~$\chi(S)=\lambda$, the case~$\chi(S)=\mu$ being dual. Then by the K\"{u}nneth condition (\ref{kunnethcondition}) we may assume that~$\chi(S_1)=\lambda$. Thus, there is a hyperplane~$L_i\supset S_1$ such that~$i$ does not belong to any circuit~$(I,J)$ with~$S_1\subset L_I\cap M_J$ and~$\chi(L_I\cap M_J)=\mu$. Then~$L_i$ contains~$S$; furthermore, a circuit~$(I,J)$ containing~$i$ and such that~$S\subset L_I\cap M_J$ necessarily satisfies~$S\subset S_1\subset L_I\cap M_J$, hence~$S$ is tame.
		\end{proof}
		
		\begin{rem}\label{remhamiltonian}
		Let us say that a stratum~$S$ of~$\B$ is \textit{hamiltonian} if it may be written~$S=L_I\cap M_J$ with~$(I,J)$ a circuit. A hamiltonian stratum is irreducible, but the converse is false in general. If~$\B$ is tame, then the color of the hamiltonian strata determine the colors of all irreducible strata, using the following basic fact about connected (=irreducible) matroids \cite[Proposition 4.1.3]{oxleymatroidtheory}.
		\end{rem}

		\begin{lem}\label{lemmatroids}
		Let~$\A=\{K_1,\ldots,K_k\}$ be an arrangement of hyperplanes,~$S$ an irreducible stratum of~$\A$,~$K_i,K_j\in\A$ hyperplanes containing~$S$. Then there exists a circuit~$I$ containing~$i,j$ such that~$S\subset K_I$.
		\end{lem}
		
		\begin{ex}
		\begin{enumerate}
		\item If~$\A$ is an arrangement, then the bi-arrangements~$(\A,\lambda)$ and~$(\A,\mu)$ are tame. 
		\item The class of tame bi-arrangements is closed under products (this is a consequence of Lemma~\ref{lemtameequivalence}).
		\item As a consequence, any product~$(\L,\lambda)\times (\M,\mu)$ is tame.
		\item The tameness condition implies the necessary condition of Lemma~\ref{lemfirstobstruction}. For bi-arrangements in~$\C^2$, these conditions are equivalent.
		\end{enumerate}
		\end{ex}
		
		\begin{lem}
		Let~$\L$ and~$\M$ be disjoint arrangements in~$\C^n$. Then the bi-arrangements~$(\L,\M,e_\lambda)$ and~$(\L,\M,e_\mu)$, equipped with the~$\lambda$-extreme and~$\mu$-extreme colorings (see Example~\ref{exextremecoloring}), are tame.
		\end{lem}
		
		\begin{proof}
		By duality, it is enough to do the proof for~$(\L,\M,e_\lambda)$. 
		\begin{enumerate}[--]
		\item Let~$S$ be a stratum such that~$e_\lambda(S)=\lambda$, then there exists a hyperplane~$L_i$ such that~$S\subset L_i$. Let~$(I,J)$ be a circuit such that~$i\in I$,~$S\subset L_I\cap M_J$, and~$e_\lambda(L_I\cap M_J)=\mu$. Then by definition,~$I=\varnothing$, which is a contradiction.
		\item Let~$S$ be a stratum such that~$e_\lambda(S)=\mu$, then there exists a hyperplane~$M_j$ such that~$S\subset M_j$. Let~$(I,J)$ be a circuit such that~$j\in J$,~$S\subset L_I\cap M_J$, and~$e_\lambda(L_I\cap M_J)=\lambda$. Then there exists a hyperplane~$L_i$ such that~$L_I\cap M_J\subset L_i$. Then~$S\subset L_i$ and~$e_\lambda(S)=\lambda$, which is a contradiction.
		\end{enumerate}
		\end{proof}

		\subsubsection{The Orlik--Solomon bi-complex}
		
		The goal of this section is to give an explicit formula for the Orlik--Solomon bi-complex of a tame bi-arrangement, and to prove at the same time that tame bi-arrangements are exact. Let us fix a tame bi-arrangement~$\B=(\L,\M,\chi)$ with~$\L=\{L_1,\ldots,L_l\}$ and~$\M=\{M_1,\ldots,M_m\}$. We first set
		$$E_{\bullet,\bullet}(\B)=E_\bullet(\L)\otimes E_\bullet(\M)^\vee=\Lambda^\bullet(e_1,\ldots,e_l)\otimes\Lambda^\bullet(f_1^\vee,\ldots,f_m^\vee).$$
		Thus,~$E_{i,j}(\B)$ has a basis consisting of monomials~$e_I\otimes f_J^\vee$ for~$|I|=i$ and~$|J|=j$. We define 
		$$d'= d\otimes \mathrm{id}:E_{\bullet,\bullet}(\B)\rightarrow E_{\bullet-1,\bullet}(\B)$$ 
		and~$$ d''=\mathrm{id}\otimes  d^\vee:E_{\bullet,\bullet-1}(\B)\rightarrow E_{\bullet,\bullet}(\B)$$
		so that~$E_{\bullet,\bullet}(\B)$ is a bi-complex.\\

		We consider on~$E_{\bullet,\bullet}(\B)$ the following homogeneous relations (subspaces of~$E_{\bullet,\bullet}(\B)$) and co-relations (subspaces of the dual space~$E_{\bullet,\bullet}(\B)^\vee$):
		\begin{enumerate}[--]
		\item for a circuit~$(I,J)$ such that~$\chi(L_I\cap M_J)=\lambda$, for all~$J'\supset J$, we consider the relation
		$$( d(e_I))\otimes f_{J'}^\vee$$
		where~$( d(e_I))$ is the ideal of~$\Lambda^\bullet(e_1,\ldots,e_l)$ generated by~$ d(e_I)$.
		\item for a circuit~$(I,J)$ such that~$\chi(L_I\cap M_J)=\mu$, for all~$I'\supset I$, we consider the co-relation
		$$e_{I'}^\vee\otimes ( d(f_J))$$
		where~$( d(f_J))$ is the ideal of~$\Lambda^\bullet(f_1,\ldots,f_m)$ generated by~$ d(f_J)$.
		\end{enumerate}
		
		\begin{defi}
		Let~$A_{\bullet,\bullet}(\B)$ be the subquotient of~$E_{\bullet,\bullet}(\B)$ defined by the above relations and co-relations.
		\end{defi}
		
		The notation will be justified by the fact that~$A_{\bullet,\bullet}(\B)$ is the Orlik--Solomon bi-complex of~$\B$, see Theorem~\ref{thmtameOS} below. It is worth noting that the definition of~$A_{\bullet,\bullet}(\B)$ only uses the colors of the hamiltonian strata, which is not surprising in view of Remark~\ref{remhamiltonian}.
		
		\begin{lem}\label{lemtamerelcorel}
		The differentials~$d'$ and~$d''$ pass to the subquotient and give~$A_{\bullet,\bullet}(\B)$ the structure of a bi-complex.
		\end{lem}
		
		\begin{proof}
		By duality, it is enough to prove that~$ d'$ and~$ d''$ pass to the quotient by the relations. It follows easily from the definitions:
		$$ d'((e_K\wedge d(e_I))\otimes f_{J'}^\vee)=( d(e_K)\wedge  d(e_I))\otimes f_{J'}^\vee$$
		and
		$$ d''((e_K\wedge d(e_I))\otimes f_{J'}^\vee)=\sum_{j\notin J'}\pm (e_K\wedge d(e_I))\otimes f_{J'\cup\{j\}}^\vee.$$
		\end{proof}
		
		For integers~$i,j\geq 0$ and a stratum~$S\in\s_{i+j}(\B)$, let us denote by~$E_{i,j}^S(\B)$ the direct summand of~$E_{i,j}(\B)$ spanned by the~$e_I\otimes f_J^\vee$ such that~$L_I\cap M_J=S$. Note that this implies that~$(I,J)$ is independent. Then we have a direct sum decomposition
		\begin{equation}\label{eqdecompositionE}
		E_{i,j}(\B)=\bigoplus_{S\in\s_{i+j}(\B)}E_{i,j}^S(\B)\oplus\bigoplus_{\substack{|I|=i\\ |J|=j\\ (I,J)\textnormal{ dependent}}}\Q\, e_I\otimes f_J^\vee.
		\end{equation}
		
		\begin{lem}
		The direct sum decomposition (\ref{eqdecompositionE}) passes to the subquotient and induces
		$$A_{i,j}(\B)=\bigoplus_{S\in\s_{i+j}(\B)}A_{i,j}^S(\B).$$
		\end{lem}
		
		\begin{proof}
		We first prove that if~$(I,J)$ is dependent then in the definition of~$A_{\bullet,\bullet}(\B)$ we either have the relation~$e_I\otimes f_J^\vee=0$ or the co-relation~$e_I^\vee\otimes f_J=0$, so that the second direct summand of (\ref{eqdecompositionE}) disappears.\\
		Let~$(I,J)$ be dependent. There exists~$I'\subset I$,~$J'\subset J$ such that~$(I',J')$ is a circuit. We assume that~$\chi(L_{I'}\cap M_{J'})=\lambda$, and show that the relation~$e_I\otimes f_J^\vee=0$ holds in~$A_{\bullet,\bullet}(\B)$ (dually, if~$\chi(L_{I'}\cap M_{J'})=\mu$ we would get the co-relation~$e_I^\vee\otimes f_J=0$). There are two cases to consider.\\
		\textit{First case}:~$I'\neq\varnothing$. For any~$i\in I'$, the Leibniz rule implies that~$e_{I'}=\pm e_i\wedge  d(e_{I'})$, hence~$e_{I'}$ and then~$e_I$ are in the ideal of~$\Lambda^\bullet(e_1,\cdots,e_l)$ generated by~$ d(e_{I'})$. Thus the relation~$( d(e_{I'}))\otimes f_J^\vee$ entails~$e_I\otimes f_J^\vee=0$ in~$A_{i,j}(\B)$.\\
		\textit{Second case:}~$I'=\varnothing$. Let~$L_i$ be a hyperplane containing~$M_{J'}$ and satisfying the condition given in the definition of a tame arrangement. Then one easily shows that there exists a subset~$J''\subset J'$ such that~$(\{i\},J'')$ is a circuit. Since~$M_{J'}\subset L_i\cap M_{J''}$, we necessarily have~$\chi(L_i\cap M_{J''})=\lambda$, and we are reduced to the first case.
		
		We next prove that the relations and co-relations are homogeneous with respect to the grading by~$\s(\B)$. Let~$(I,J)$ be a circuit such that~$\chi(L_I\cap M_J)=\lambda$, and let~$J'\supset J$. Then the corresponding relation reads
		$$\sum_{i\in I}\pm e_{I\setminus\{i\}}\otimes f_{J'}^\vee=0.$$
		For all~$i\in I$,~$(I\setminus \{i\},J)$ is independent, hence~$L_{I\setminus\{i\}}\cap M_J=L_I\cap M_J$ does not depend on~$i$, and~$L_{I\setminus\{i\}}\cap M_{J'}$  does not depend on~$i$. Hence the relations are homogeneous with respect to the grading by~$\s(\B)$. Dually, the same is true for the co-relations.
		\end{proof}
		
		\begin{rem}\label{remtameOSlocal}
		By definition, the component~$A_{\bullet,\bullet}^S(\B)$ only depends on the arrangement~$\B^{\leq S}$, which is tame according to Remark~\ref{remtamelocal}. For a strict stratum~$\Sigma$, we then have~$A_{\bullet,\bullet}^{\leq\Sigma}(\B)\cong A_{\bullet,\bullet}(\B^{\leq\Sigma})$.
		\end{rem}
		
		\begin{ex}
		Let~$\L=\{L_1,L_2\}$ and~$\M=\{M_1\}$ be three distinct lines in~$\C^2$. Let~$Z$ be the origin, we set~$\chi(Z)=\lambda$. This defines a tame bi-arrangement~$\B=(\L,\M,\chi)$. The only circuit is~$(\{1,2\},\{1\})$. Then~$A_{\bullet,\bullet}(\B)$ is the quotient of~$\Lambda^\bullet(e_1,e_2)\otimes \Lambda^\bullet(f_1^\vee)$ by the relations
		$(e_2-e_1) f_1^\vee=0$ and~$e_{12}f_1^\vee=0$. It may be pictured as 
		$$\xymatrix{
		 \Q\, e_{12} \ar[r]  & \Q\, e_1\oplus \Q\, e_2 \ar[r]\ar[d] & \Q\, 1 \ar[d]\\ 
		 & \left( \Q\, e_1f_1^\vee \oplus \Q\, e_2f_1^\vee \right) / (e_1f_1^\vee = e_2f_2^\vee ) \ar[r] & \Q\, f_1^\vee 
		}$$
		and its rows are exact.
		\end{ex}		
		
		\begin{thm}\label{thmtameOS}
		Let~$\B$ be a tame bi-arrangement. Then~$A_{\bullet,\bullet}(\B)$ is the Orlik--Solomon bi-complex of~$\B$, and~$\B$ is exact.
		\end{thm}
		
		\begin{proof}
		Firstly,~$A_{0,0}^{\C^n}(\B)$ is indeed one-dimensional with basis~$1\otimes 1$. Secondly, for every strict stratum~$\Sigma$,~$A_{\bullet,\bullet}^{\leq\Sigma}(\B)=A_{\bullet,\bullet}(\B^{\leq\Sigma})$ is a bi-complex by Remark~\ref{remtameOSlocal} and Lemma~\ref{lemtamerelcorel}.
		Thirdly, let~$\Sigma$ be a strict stratum of~$\B$ such that~$\chi(\Sigma)=\lambda$ (the case~$\chi(\Sigma)=\mu$ being dual). We want to show that all the rows of~$A^{\leq\Sigma}_{\bullet,\bullet}(\B)$ are exact. By the same remark as above, we may assume that~$\Sigma$ is the intersection of all the hyperplanes of~$\B$ and show that all the rows of~$A_{\bullet,\bullet}(\B)$ are exact.
		
		 By the definition of a tame bi-arrangement, there exists a hyperplane~$L_i$ such that~$i$ does not belong to any circuit~$(I,J)$ with~$\chi(L_I\cap M_J)=\mu$. We define~$h:E_{\bullet,\bullet}(\B)\rightarrow E_{\bullet+1,\bullet}(\B)$ by the formula~$h(x\otimes y)=(e_i\wedge x)\otimes y$. Then the Leibniz rule implies that~$ d'\circ h+ h\circ d'=\mathrm{id}$, hence~$h$ is a contracting homotopy for all the rows of~$E_{\bullet,\bullet}(\B)$. Hence we are done if we prove that~$h$ passes to the subquotient and induces~$h:A_{\bullet,\bullet}(\B)\rightarrow A_{\bullet+1,\bullet}(\B)$.\\
		 The fact that~$h$ respects the relations is trivial. Let~$(I,J)$ be a circuit such that~$\chi(L_I\cap M_J)=\mu$. Then by assumption~$i\notin I$. Thus, any subset~$I'\supset I$ that contains~$i$ is of the form~$I'=\{i\}\sqcup I''$ with~$I''\supset I$. Hence we have~$h^\vee(e_{I'}\otimes (f_K\wedge d(f_J)))=\pm e_{I''}\otimes (f_K\wedge  d(f_J))$ and~$h$ respects the co-relations.
		\end{proof}
		
		\begin{rem}
		The definition of~$A_{\bullet,\bullet}(\B)$ is automatically self-dual, viewing~$A_{\bullet,\bullet}(\B^\vee)$ as a subquotient of~$ \Lambda^\bullet(e_1^\vee,\ldots,e_l^\vee)\otimes\Lambda^\bullet(f_1,\ldots,f_m)\cong \Lambda^\bullet(f_1,\ldots,f_m)\otimes \Lambda^\bullet(e_1^\vee,\ldots,e_l^\vee)$.
		\end{rem}
		
		\begin{rem}
		There is a natural structure of graded module over~$E_\bullet(\L)$ on~$E_{\bullet,\bullet}(\L,\M)$. Let~$(\L,\M,e_\lambda)$ be a bi-arrangement equipped with the~$\lambda$-extreme coloring; then this structure passes to the subquotient and induces on~$A_{\bullet,\bullet}(\L,\M,e_\lambda)$ a structure of graded module over the Orlik--Solomon algebra~$A_\bullet(\L)$. Dually,~$A_{\bullet,\bullet}(\L,\M,e_\mu)$ is a graded comodule over~$\left(A_\bullet(\M)\right)^\vee$, which is the same as a graded module over~$A_\bullet(\M)$.
		\end{rem}
		
	\subsection{Examples}
	
		\subsubsection{A non-tame bi-arrangement which is not exact}
		
		To find a non-tame non-exact bi-arrangement, we may choose trivial examples that do not satisfy the necessary condition of Lemma~\ref{lemfirstobstruction}. Here we present a less trivial example.
		
		 Let us consider, in~$\C^3$, a bi-arrangement~$\B=(\L,\M,\chi)$ with~$\L=\{L_1,L_2,L_3\}$ and~$\M=\{M_1,M_2\}$ defined by the equations~$L_1=\{x_1=0\}$,~$L_2=\{x_2=0\}$,~$L_3=\{x_3=0\}$,~$M_1=\{x_1+x_3=0\}$,~$M_2=\{x_2+x_3=0\}$. Apart from the hyperplanes, the irreducible strata are the lines~$L_{13}=\{x_1=x_3=0\}$,~$L_{23}=\{x_2=x_3=0\}$ and the point~$P=\{x_1=x_2=x_3=0\}$. We define~$\chi(L_{13})=\chi(L_{23})=\mu$ and~$\chi(P)=\lambda$. The circuits are~$(\{1,3\},\{1\})$,~$(\{2,3\},\{2\})$ with color~$\mu$, and~$(\{1,2\},\{1,2\})$ with color~$\lambda$. The stratum~$P$ is not tame, thus~$\B$ is not tame.
		 
		 It is easy to check that~$\B$ is not exact. This follows from looking at the first row~$(\bullet,0)$ of its Orlik--Solomon bi-complex. The only non-zero terms are~$A_{1,0}^{L_i}=\Q$ for~$i=1,2,3$, and~$A_{2,0}^{L_{12}}=\ker\left(A_{1,0}^{L_1}\oplus A_{1,0}^{L_2} \rightarrow \Q\right) \cong \Q$. The first row is then 
		$$0\rightarrow 0\rightarrow \Q\rightarrow \Q\oplus\Q\oplus\Q \rightarrow \Q \rightarrow 0$$
		which is not exact.
	
		\subsubsection{A non-tame bi-arrangement which is exact}
			
		 Let us consider the same bi-arrangement as in the previous example, but with the coloring~$\chi(L_{13})=\chi(L_{23})=\lambda$  and~$\chi(P)=\mu$ (it is not its dual, since we have not exchanged~$\L$ and~$\M$). This bi-arrangement~$\B$ is not tame, but it may be checked that it is exact. 

\section{Bi-arrangements of hypersurfaces}\label{sectionhypersurfaces}

	\subsection{Arrangements of hypersurfaces and resolution of singularities}
	
		We fix a complex manifold~$X$. An \textit{arrangement of hypersurfaces} in~$X$ is a finite set~$\A$ of smooth hypersurfaces of~$X$ which is locally an arrangement of hyperplanes. More precisely, it means that around every point~$p\in X$ we may find a system of local coordinates centered at~$p$ such that all hypersurfaces~$K\in\A$ are defined by a linear equation.
		
		\begin{ex}
		A (simple) normal crossing divisor in~$X$ is a special case of an arrangement of hypersurfaces. In this case, we may find local coordinates around every point such that all hypersurfaces are defined by the vanishing of a coordinate.
		\end{ex}
		
		\begin{ex}\label{exarrangementshypersurfaces}
		\begin{enumerate}
		\item An arrangement of hyperplanes in~$\C^n$ is an arrangement of hypersurfaces. More generally, a finite set of hyperplanes of~$\C^n$ that do not necessarily pass through the origin is an arrangement of hypersurfaces.
		\item A finite set of hyperplanes of~$\mathbb{P}^n(\C)$ is an arrangement of hypersurfaces.
		\item If~$Y$ is a Riemann surface and~$X=Y^n$ is the~$n$-fold cartesian power of~$Y$, then there are distinguished hypersurfaces in~$X$: the diagonals~$\{y_i=y_j\}$, and the hypersurfaces~$\{y_i=a\}$ where~$a\in Y$ is a point. Any finite set of such hypersurfaces is an arrangement of hypersurfaces. In the context of motivic periods, these arrangements of hypersurfaces have been studied by S. Bloch~\cite{blochtreeterated}.
		\end{enumerate}
		\end{ex}		
		
		A \textit{stratum} of~$\A$ is a connected component of a non-empty intersection~$K_I=\bigcap_{i\in I}K_i$ of some hypersurfaces~$K_i\in\A$. It is a submanifold of~$X$. For instance, the whole space~$X=K_\varnothing$ is always a stratum of~$\A$, and the other strata are called \textit{strict}. A stratum~$S$ is \textit{reducible} (resp. \textit{irreducible}) if it is reducible (resp. irreducible) locally around every point~$p\in S$. Every hypersurface~$K\in\A$ is irreducible; if they are the only irreducible strata, then~$\A$ is a normal crossing divisor.\\
		
		The class of arrangements of hypersurfaces is closed under blow-ups along a certain class of strata, that we now introduce.
		
		\begin{defi}
		Let~$\A$ be an arrangement of hyperplanes in~$\C^n$. A strict stratum~$Z$ of~$\A$ is \textit{good} if there exists a stratum~$U$ and a decomposition~$Z\pitchfork U$ such that for every hyperplane~$K\in\A$,~$K$ contains~$Z$ or~$U$.
		
		Let~$\A$ be an arrangement of hypersurfaces in~$X$. A strict stratum~$Z$ of~$\A$ is \textit{good} if it is good in the above sense locally around every point~$p\in Z$. 
		\end{defi}
	
		For instance, a stratum of dimension~$0$ (a point) is always good.
		
		\begin{lem}\label{lemminimalirreduciblegood}
		Let~$\A$ be an arrangement of hypersurfaces in~$X$, and~$S$ be a minimal \footnote{For the usual inclusion order.} irreducible stratum of~$\A$. Then~$S$ is good.
		\end{lem}
		
		\begin{proof}
		The statement is local, so we may assume that~$X=\C^n$ and~$\A$ is an arrangement of hyperplanes. Let~$M=\bigcap_{K\in\A}K$ be the minimal stratum of~$\A$ and~$M=S_1\pitchfork \cdots \pitchfork S_r$ be its decomposition into irreducibles. Then the~$S_i$'s are exactly the minimal irreducible strata. We may then assume that~$S=S_1$. Let us define~$U=S_2\pitchfork\cdots\pitchfork S_r$. Then we have a decomposition~$S\pitchfork U$ and every hyperplane~$K\in\A$ contains~$S$ or~$U$, hence~$S$ is a good stratum.
		\end{proof}	
		
		\begin{lem}\label{lemblowuparrangement}
		Let~$\A$ be an arrangement of hypersurfaces in~$X$ and~$Z$ a good stratum of~$\A$ of codimension~$\geq 2$. Let~$\pi:\td{X}\rightarrow X$ be the blow-up of~$X$ along~$Z$ and~$E=\pi^{-1}(Z)$ the exceptional divisor. We write~$\td{Y}$ for the strict transform of a submanifold~$Y\subset X$. Then 
		\begin{enumerate}
		\item The set~$\td{\A}=\{E\}\cup\{\td{K}\,,\,K\in\A\}$ is an arrangement of hypersurfaces in~$\td{X}$.
		\item The strata of~$\td{\A}$ are of the form~$\td{S}$ or~$E\cap\td{S}$, for strata~$S$ of~$\A$ that are not contained in~$Z$.
		\item The irreducible strata of~$\A$ are~$E$ and the strict transforms~$\td{S}$ of the irreducible strata~$S$ of~$\A$ that are not contained in~$Z$.
		\end{enumerate}
		\end{lem}
		
		\begin{defi}
		We call~$\td{\A}=\{E\}\cup\{\td{K}\,,\,K\in\A\}$ the \textit{blow-up} of~$\A$ along~$Z$.
		\end{defi}
		
		\begin{proof}
		The statement is local, so we assume that~$X=\C^n$ and~$\A$ is an arrangement of hyperplanes. Since~$Z$ is a good stratum, we may choose coordinates~$(z_1,\ldots,z_n)$ such that~$Z=\{z_1=\cdots=z_r=0\}$ for some integer~$r$, and such that the hyperplanes~$K\in\A$ are given by equations of the form~$\alpha_1z_1+\cdots+\alpha_rz_r=0$ or~$\alpha_{r+1}z_{r+1}+\cdots+\alpha_nz_n=0$.
		\begin{enumerate}
		\item We have~$r$ local charts for the blow-up~$\pi:\td{X}\rightarrow X$, given for~$k=1,\ldots,r$ by 
		$$\pi_k(z_1,\ldots,z_r)=(z_kz_1,\ldots,z_kz_{k-1},z_k,z_kz_{k+1},\ldots,z_kz_r,z_{z+1},\ldots,z_n).$$
		In such a chart, the exceptional divisor is~$E=\{z_k=0\}$; the strict transform of~$K=\{\alpha_1z_1+\cdots+\alpha_rz_r=0\}$ is~$\td{K}=\{\alpha_1z_1+\cdots+\alpha_{k-1}z_{k-1}+\alpha_k+\alpha_{k+1}z_{k+1}+\cdots+\alpha_rz_r=0\}$; the strict transform of~$K=\{\alpha_{r+1}z_{r+1}+\cdots+\alpha_nz_n=0\}$ is~$\td{K}=\{\alpha_{r+1}z_{r+1}+\cdots+\alpha_nz_n=0\}$. All these equations are linear, hence the result.
		\item For~$S$ a stratum of~$\A$, it is easy to show using the above local charts that we have
		$$\td{S}=\varnothing\,\,\Leftrightarrow\,\, E\cap \td{S}=\varnothing \,\,\Leftrightarrow\,\, S\subset Z$$
		hence the result.
		\item The exceptional divisor~$E$ is obviously irreducible. Now let us fix a stratum~$S$ of~$\A$ not contained in~$Z$. Then it is easy to see using the above local charts that~$E\pitchfork\td{S}$ and that for every~$K\in\A$,~$E\cap\td{S}\subset\td{K}\,\Rightarrow\, S\subset K$; thus,~$E\cap \td{S}$ is reducible if~$S$ is not the whole space~$\C^n$. We are left with proving that~$\td{S}$ is irreducible if and only if~$S$ is irreducible. It it easy to see that a decomposition~$S=A\pitchfork B$ gives a decomposition~$\td{S}=\td{A}\pitchfork\td{B}$ and vice versa, hence the result.
		\end{enumerate}
		\end{proof}
		
		Blow-ups along good strata are enough to resolve the singularities of hypersurface arrangements, as the next theorem shows.
		
		\begin{thm}\label{thmresolutionarrangements}
		Let~$\A$ be an arrangement of hypersurfaces in~$X$. We inductively define a sequence of complex manifolds~$X^{(k)}$ and arrangements of hypersurfaces~$\A^{(k)}$ inside~$X^{(k)}$, via the following process.
		\begin{enumerate}[(a)] 
		\item~$X^{(0)}=X$ and~$\A^{(0)}=\A$;
		\item for~$k\geq 0$, let~$Z^{(k)}$ be a minimal irreducible stratum of~$\A^{(k)}$ of codimension~$\geq 2$,~$X^{(k+1)}\rightarrow X^{(k)}$ the blow-up of~$X^{(k)}$ along~$Z^{(k)}$. We let~$\A^{(k+1)}=\td{\A^{(k)}}$ be the blow-up of~$\A^{(k)}$ along~$Z^{(k)}$.
		\end{enumerate}
		After a finite number of steps, we get a normal crossing divisor~$\A^{(\infty)}$ inside~$X^{(\infty)}$.
		\end{thm}
		
		\begin{proof}
		The process is well-defined according to Lemma~\ref{lemminimalirreduciblegood} and Lemma~\ref{lemblowuparrangement}. For~$k\geq 0$, let~$\mathscr{I}^{(k)}$ be the set of irreducible strata of~$\A^{(k)}$ of codimension~$\geq 2$. Then~$Z^{(k)}$ is a minimal element of~$\mathscr{I}^{(k)}$, and~$\mathscr{I}^{(k+1)}$ consists of the strict transforms of the other elements of~$\mathscr{I}^{(k)}$. Thus, we get~$|\mathscr{I}^{(k+1)}|=|\mathscr{I}^{(k)}|-1$. After a finite number of steps, we end up with an arrangement~$\A^{(\infty)}$ inside~$X^{(\infty)}$ such that~$\mathscr{I}^{(\infty)}$ is empty, hence~$\A^{(\infty)}$ is a normal crossing divisor.
		\end{proof}
		
		\begin{rem}\label{remindependentorderblowups}
		At each step of the process described in Theorem~\ref{thmresolutionarrangements}, we choose a minimal irreducible stratum of codimension~$\geq 2$. The resulting pair~$(\A^{(\infty)},X^{(\infty)})$ is independent of these choices, as follows from the work of Li~\cite{li}. According to \textit{loc. cit.}, Definition 1.1 and Theorem 1.3, the morphism~$\pi:X^{(\infty)}\rightarrow X$ is the wonderful compactification of the arrangement~$\A$ with respect to the building set~$\mathscr{I}$ consisting of the irreducible strata; it is by definition independent of any choice.
		\end{rem}
	
	\subsection{The motive of a bi-arrangement of hypersurfaces}
		
		\begin{defi}
		Let~$X$ be a complex manifold. A \textit{bi-arrangement of hypersurfaces}~$\B=(\A,\chi)$ in~$X$ is the data of an arrangement of hypersurfaces~$\A$ in~$X$ along with a coloring function
		$$\chi:\s_+(\A)\rightarrow\{\lambda,\mu\}$$
		such that the K\"{u}nneth condition (\ref{kunnethcondition}) is satisfied locally around every point of~$X$.
		\end{defi}
		
		As for bi-arrangements of hyperplanes, only the colors of the irreducible strata will matter, and thus we will consider bi-arrangements of hypersurfaces up to equivalence (see Definition~\ref{defiequivalence}).
		
		We will also use the notational conventions~\ref{notationconvention1} and~\ref{notationconvention2} in the context of bi-arrangements of hypersurfaces. When the underlying arrangement of hypersurfaces is a normal crossing divisor, then~$\chi$ is only determined (up to equivalence) by the colors~$\chi(K)$ of the hypersurfaces~$K\in\B$, hence we may simply write~$\B=(\L,\M)$.
		
		We also define the dual~$\B^\vee$ of a bi-arrangement of hypersurfaces.\\
		
		Let~$\B=(\A,\chi)$ be a bi-arrangement of hypersurfaces in a complex manifold~$X$, and~$Z$ be a good stratum of~$\B$ of codimension~$\geq 2$. Let~$\pi:\td{X}\rightarrow X$ be the blow-up of~$X$ along~$Z$, and~$E=\pi^{-1}(Z)$ be the exceptional divisor. Let~$\td{\A}=\{E\}\cup\{\td{K}\,,\,K\in\A\}$ be the blow-up of~$\A$ along~$Z$. Then we define a bi-arrangement of hypersurfaces~$\td{\B}=(\td{\A},\td{\chi})$ in~$\td{X}$ whose underlying arrangement of hypersurfaces is~$\td{\A}=\{E\}\cup\{\td{K}\,,\,K\in\A\}$. We define the coloring~$\td{\chi}$ only on the irreducible strata: we set~$\td{\chi}(E)=\chi(Z)$, and for an irreducible stratum~$S$ not contained in~$Z$, we set~$\td{\chi}(\td{S})=\chi(S)$.
		
		\begin{defi}
		We call~$\td{\B}=(\td{\A},\td{\chi})$ the \textit{blow-up} of~$\B$ along~$Z$.
		\end{defi}
		
		If~$Z$ is irreducible (which will be our main case of interest) then the blow-up is a well-defined operation among equivalence classes of bi-arrangements of hypersurfaces.\\
		
		Let~$\B$ be a bi-arrangement of hypersurfaces in a complex manifold~$X$. We inductively define a sequence of complex manifolds~$X^{(k)}$ and bi-arrangements of hypersurfaces~$\B^{(k)}$ inside~$X^{(k)}$, via the following process.
		\begin{enumerate}[(a)] 
		\item~$X^{(0)}=X$ and~$\B^{(0)}=\B$;
		\item for~$k\geq 0$, let~$Z^{(k)}$ be a minimal irreducible stratum of~$\B^{(k)}$ of codimension~$\geq 2$,~$X^{(k+1)}\rightarrow X^{(k)}$ the blow-up of~$X^{(k)}$ along~$Z^{(k)}$. We let~$\B^{(k+1)}=\td{\B^{(k)}}$ be the blow-up of~$\B^{(k)}$ along~$Z^{(k)}$.
		\end{enumerate}
		As in the case of arrangements of hypersurfaces, we get after a finite number of steps a bi-arrangement of hypersurfaces~$\B^{(\infty)}$ inside~$X^{(\infty)}$, whose underlying arrangement of hypersurfaces is a normal crossing divisor. We write~$\B^{(\infty)}=(\L^{(\infty)},\M^{(\infty)})$, with~$\L^{(\infty)}\cup\M^{(\infty)}$ a normal crossing divisor. By an abuse of notation, we write~$\L^{(\infty)}$ (resp.~$\M^{(\infty)}$) for the union of all the hypersurfaces~$K\in\L^{(\infty)}$ (resp.~$K\in\M^{(\infty)}$).

		\begin{defi}\label{defimotive}
		The \textit{motive} of the bi-arrangement of hypersurfaces~$\B$ is the collection of relative cohomology groups (see (\ref{appmotiveNCD}))
		$$H^\bullet(\B)=H^\bullet(X^{(\infty)}\setminus \L^{(\infty)},\M^{(\infty)}\setminus \M^{(\infty)}\cap \L^{(\infty)}).$$
		If~$X$ is a smooth complex variety, then~$H^\bullet(\B)$ is endowed with a mixed Hodge structure. 
		\end{defi}
		
		\begin{rem}
		According to Remark~\ref{remindependentorderblowups}, the motive of a bi-arrangement of hypersurfaces is independent of the choices made during the blow-up process.
		\end{rem}
		
		
		\begin{ex}
		\begin{enumerate}
		\item If~$\A$ is a hypersurface arrangement in~$X$, we have
		$$H^\bullet(\A,\lambda)\cong H^\bullet(X\setminus \A) \;\;\textnormal{ and }\;\, H^\bullet(\A,\mu)\cong H^\bullet(X,\A).$$
		\item For~$\B=(\L,\M)$ a normal crossing divisor, then there is no blow-up and we simply have 
		$$H^\bullet(\L,\M)=H^\bullet(X\setminus \L, \M\setminus \M\cap\L).$$
		\end{enumerate}
		\end{ex}
		
		\begin{rem}
		There is also the compactly-supported version (see (\ref{appmotivecompactNCD}))
		$$H_c^\bullet(\B)=H_c^\bullet(X^{(\infty)}\setminus \L^{(\infty)},\M^{(\infty)}\setminus \M^{(\infty)}\cap \L^{(\infty)}).$$
		Putting~$n=\mathrm{dim}_\C(X)$, the duality of bi-arrangements is viewed as a Poincar\'{e}--Verdier duality isomorphism (Proposition~\ref{appproppoincareverdier})
		$$H^k(\B^\vee)\cong \left(H_c^{2n-k}(\B)\right)^\vee.$$
		\end{rem}

	\subsection{The Orlik--Solomon bi-complex, and blow-ups}
	
		Let~$\B$ be a bi-arrangement of hypersurfaces in a complex manifold~$X$. The definition of the \textit{Orlik--Solomon bi-complex} of~$\B$ may be repeated word for word from the local case: we start with~$A_{0,0}^X(\B)=\Q$ and define the bi-complexes~$A_{\bullet,\bullet}^{\leq\Sigma}(\B)$ by induction on the codimension of~$\Sigma$. Note that~$A^{\leq\Sigma}_{\bullet,\bullet}(\B)$ only depends on the hypersurfaces that contain~$\Sigma$ and may be computed in a local chart around any point of~$\Sigma$. We say that a bi-arrangement of hypersurfaces is \textit{exact} if all its strict strata are exact in the sense of Definition~\ref{defiexact}.\\
	
		It is worth noting that although every~$A^{\leq\Sigma}_{\bullet,\bullet}(\B)$ is a bi-complex, the direct sum~$\bigoplus_{S}A_{\bullet,\bullet}^S$ is not in general. For instance, if~$\B$ made of two non-intersecting hypersurfaces, one colored~$\lambda$ and the other colored~$\mu$, we get the following non-commutative square.
		$$\xymatrix{
		\Q \ar[r]^{\mathrm{id}}  \ar[d] & \Q \ar[d]^{\mathrm{id}}\\
		0 \ar[r] &\Q
		}$$
		
		Let now~$Z$ be a good stratum of~$\B$,~$\pi:\td{X}\rightarrow X$ be the blow-up along~$Z$,~$E=\pi^{-1}(Z)$ be the exceptional divisor, and~$\td{\B}$ be the blow-up of~$\B$ along~$Z$. The following proposition, which will be crucial in the sequel, expresses the Orlik--Solomon bi-complex of~$\td{\B}$ in terms of that of~$\B$.
		
		\begin{prop}\label{propOSbicomplexblowup}
		Let us assume that~$\chi(Z)=\lambda$. We have isomorphisms, for~$S$ a stratum of~$\B$ that is not contained in~$Z$:
		$$A_{i,j}^{\td{S}}(\td{\B})\cong A_{i,j}^S(\B) \;\;\textit{ and }\;\; A_{i,j}^{E\cap\td{S}}(\td{\B}) \cong A_{i-1,j}^S(\B).$$
		They are compatible with the differentials in that we have the following commutative diagrams.
		\begin{enumerate}
		\item For the inclusions~$\td{S}\stackrel{1}{\hookrightarrow}\td{T}$:\\
		\centerline{\xymatrixcolsep{3pc}\xymatrix{
		A_{i,j}^{\td{S}}(\td{\B}) \ar[r]^{ d'_{\td{S},\td{T}}} \ar@{<->}[d]_{\cong} & A_{i-1,j}^{\td{T}}(\td{\B})  \ar@{<->}[d]^{\cong} &
		 A_{i,j-1}^{\td{T}}(\td{\B}) \ar[r]^{ d''_{\td{S},\td{T}}} \ar@{<->}[d]_{\cong} & A_{i,j}^{\td{S}}(\td{\B})  \ar@{<->}[d]^{\cong}  \\ 
		A_{i,j}^S(\B) \ar[r]_{ d'_{S,T}} & A_{i-1,j}^T(\B) & A_{i,j-1}^T(\B) \ar[r]_{ d''_{S,T}} & A_{i,j}^S(\B)
		}}
		\item For the inclusions~$E\cap\td{S}\stackrel{1}{\hookrightarrow} E\cap\td{T}$:\\
		\centerline{\xymatrixcolsep{4pc}\xymatrix{
		A_{i,j}^{E\cap\td{S}}(\td{\B}) \ar[r]^{ d'_{E\cap\td{S},E\cap\td{T}}} \ar@{<->}[d]_{\cong} & A_{i-1,j}^{E\cap\td{T}}(\td{\B})  \ar@{<->}[d]^{\cong} &
		 A_{i,j-1}^{E\cap\td{T}}(\td{\B}) \ar[r]^{ d''_{E\cap\td{S},E\cap\td{T}}} \ar@{<->}[d]_{\cong} & A_{i,j}^{E\cap\td{S}}(\td{\B})  \ar@{<->}[d]^{\cong}  \\ 
		A_{i-1,j}^S(\B) \ar[r]_{- d'_{S,T}} & A_{i-2,j}^T(\B) & A_{i-1,j-1}^T(\B) \ar[r]_{ d''_{S,T}} & A_{i-1,j}^S(\B)
		}}
		\item For the inclusions~$E\cap\td{S}\stackrel{1}{\hookrightarrow}\td{S}$:\\
		\centerline{\xymatrixcolsep{4pc}\xymatrix{
		A_{i,j}^{E\cap\td{S}}(\td{\B}) \ar[r]^{ d'_{E\cap\td{S},\td{S}}} \ar@{<->}[d]_{\cong} & A_{i-1,j}^{\td{S}}(\td{\B})  \ar@{<->}[d]^{\cong}   &
		A_{i,j-1}^{\td{S}}(\td{\B}) \ar[r]^{ d''_{E\cap\td{S},\td{S}}} \ar@{<->}[d]_{\cong} & A_{i,j}^{E\cap\td{S}}(\td{\B})  \ar@{<->}[d]^{\cong} \\ 
		A_{i-1,j}^S(\B) \ar@{=}[r]_{\mathrm{id}} & A_{i-1,j}^S(\B) & A_{i,j-1}^S(\B) \ar[r]_{0} & A_{i-1,j}^S(\B)
		}}
		\end{enumerate}
		The case~$\chi(Z)=\mu$ is dual.
		\end{prop}
		
		\begin{proof}
		We have an isomorphism~$\pi:\td{X}\setminus E \stackrel{\cong}{\rightarrow} X\setminus Z$. Let us recall that the construction of the Orlik--Solomon bi-complex is local. Let~$S$ be a stratum of~$\B$ that is not contained in~$Z$,~$p\in S\setminus S\cap Z$,~$\td{p}=\pi^{-1}(p)\in\td{S}$. Around the point~$\td{p}$, the local situation is the same as the one around the point~$p$, hence the first isomorphism.\\
		For the second isomorphism, we see using local coordinates as in the Proof of Lemma~\ref{lemblowuparrangement} that the local situation around a point of~$E\cap \td{S}$ is that of a decomposition~$E\pitchfork\td{S}$. Thus, the K\"unneth formula (Proposition~\ref{propkunneth}) implies that we have 
		$$A_{i,j}^{E\cap\td{S}}(\td{\B})\cong \left(A_{1,0}^{E}(\td{\B})\otimes A_{i-1,j}^{\td{S}}(\td{\B})\right) \oplus \left(A_{0,1}^{E}(\td{\B})\otimes A_{i,j-1}^{\td{S}}(\td{\B})\right).~$$
		Since we have~$\chi(E)=\lambda$, we have~$A_{1,0}^E(\td{\B})=\Q$ and~$A_{0,1}^E(\td{\B})=0$. Hence the second isomorphism follows from the first isomorphism~$A_{i-1,j}^{\td{S}}(\td{\B})\cong A_{i-1,j}^S(\B)$.\\
		The compatibility with the differentials is easy. One only has to note the minus sign in front of~$ d'_{S,T}$ which follows from the Koszul sign rule in a tensor product of two (bi-)complexes.
		\end{proof}
		
		\begin{coro}\label{coroblowupexact}
		If~$\B$ is exact, then~$\td{\B}$ is exact.
		\end{coro}
		
		\begin{proof}
		According to Corollary~\ref{coroOSequivalence} it is enough to check the exactness of the strata~$\td{S}$, for~$S$ an irreducible stratum of~$\B$ not contained in~$Z$. Proposition~\ref{propOSbicomplexblowup} implies that we have~$A_{\bullet,\bullet}^{\td{S}}(\td{\B})\cong A_{\bullet,\bullet}^S(\B)$, hence the result.
		\end{proof}

\section{The geometric Orlik--Solomon bi-complex and the main theorem}\label{sectiongeometric}

	\subsection{The geometric Orlik--Solomon bi-complex}
	
		We fix a complex manifold~$X$ and a bi-arrangement of hypersurfaces~$\B$ in~$X$. We fix an integer~$q$. Let us write, for~$S\in\s_{i+j}(\B)$,
		$$\qD^S_{i,j}(\B)=H^{q-2i}(S)(-i)\otimes A_{i,j}^S(\B).$$
		If~$X$ is a smooth complex variety and the hypersurfaces~$K\in\B$ are divisors (we call this the \enquote{algebraic case}), then this is endowed with a mixed Hodge structure. If furthermore~$X$ is projective, it is a pure Hodge structure of weight~$q$.  \\
		
		Let~$\iota_{S}^T:S\stackrel{1}{\hookrightarrow}T$ be an inclusion of strata of~$\B$, with~$S\in\s_{i+j}(\B)$ and~$T\in\s_{i+j-1}(\B)$. We refer the reader to Appendix~\ref{parappB} for details on Gysin morphisms and pull-backs.
		\begin{enumerate}[--]
		\item We have the Gysin morphism~$\left(\iota_S^T\right)_*:H^{q-2i}(S)(-i)\rightarrow H^{q-2i+2}(T)(-i+1)$. We then define a morphism\footnote{We make an abuse of notation by denoting by the same symbols~$d'_{S,T}$ and~$d''_{S,T}$ the differentials in the Orlik--Solomon bi-complex and in the geometric Orlik--Solomon bi-complex; no confusion should arise.}
		$$ d'_{S,T}:\qD^S_{i,j}(\B)\rightarrow \qD^T_{i-1,j}(\B)$$
		by the formula
		$$d'_{S,T}(s\otimes X)=(\iota_S^T)_*(s)\otimes d'_{S,T}(X)$$
		for~$s\in H^{q-2i}(S)(-i)$ and~$X\in A_{i,j}^S(\B)$.
		\item We have the restriction morphism~$\left(\iota_S^T\right)^*:H^{q-2i}(T)(-i)\rightarrow H^{q-2i}(S)(-i)$. We then define a morphism
		$$ d''_{S,T}:\qD^T_{i,j-1}(\B)\rightarrow \qD^S_{i,j}(\B)$$
		by the formula
		$$d''_{S,T}(t\otimes X)=(\iota_S^T)^*(t)\otimes d''_{S,T}(X)$$
		for~$t\in H^{q-2i}(T)(-i)$ and~$X\in A_{i,j-1}^T(\B)$.
		\end{enumerate}
	
		Let us now set~$$\qD_{i,j}(\B)=\bigoplus_{S\in\s_{i+j}(\B)}\qD_{i,j}^S(\B).$$ 
		
		The above morphisms induce 
		$$ d':\qD_{\bullet,\bullet}(\B)\rightarrow \qD_{\bullet-1,\bullet}(\B)$$ and 
		$$ d'':\qD_{\bullet,\bullet-1}(\B)\rightarrow \qD_{\bullet,\bullet}(\B).$$	
		
		If~$X$ is a smooth complex variety,~$ d'$ and~$ d''$ are morphisms of mixed Hodge structures.
		
		\begin{thm}\label{thmDbicomplex}
		The differentials~$ d'$ and~$ d''$ make~$\qD_{\bullet,\bullet}(\B)$ into a bi-complex.
		\end{thm}
	
		\begin{defi}
		We call~$\qD_{\bullet,\bullet}(\B)$ the \textit{geometric Orlik--Solomon bi-complex} of index~$q$ of~$\B$. We will denote by~$\qD_\bullet(\B)$ its total complex, and call it the \textit{geometric Orlik--Solomon complex} of index~$q$:
		$$\qD_n(\B)=\bigoplus_{i-j=n}\qD_{i,j}(\B).$$
		\end{defi}
		
		\begin{ex}
		\begin{enumerate}
		\item Let~$\A$ be an arrangement of hypersurfaces in~$X$. Then the geometric Orlik--Solomon bi-complexes for~$(\A,\lambda)$ are concentrated in bi-degrees~$(n,0)$ with
		$$\qD_{n,0}(\A,\lambda)=\bigoplus_{S\in\s_n(\A)}H^{q-2n}(S)(-n)\otimes A_n(\A).$$
		Up to a shift, it is the same as the Gysin complex defined in~\cite{duponthypersurface}.
		Dually, the geometric Orlik--Solomon bi-complexes for~$(\A,\mu)$ are concentrated in bi-degrees~$(0,n)$ with
		$$\qD_{0,n}=\bigoplus_{S\in\s_n(\A)}H^q(S)\otimes \left(A_n(\A)\right)^\vee.$$
		\item If~$\B=(\L,\M)$ is a normal crossing divisor with~$\L=\{L_1,\ldots,L_l\}$ and~$\M=\{M_1,\ldots,M_m\}$, then we get 
		$$\qD_{i,j}(\L,\M)=\bigoplus_{\substack{|I|=i\\|J|=j}} H^{q-2i}(L_I\cap M_J)(-i)$$
		and the Orlik--Solomon complexes~$\qD_\bullet(\L,\M)$ form the~$E_1$ page of the spectral sequence (\ref{appeqspectralsequence}) described in Appendix~\ref{parappA}.
		\end{enumerate}
		\end{ex}
		
		In the rest of this section, we prove Theorem~\ref{thmDbicomplex} by showing that in~$\qD_{\bullet,\bullet}(\B)$ we have the equalities~$d'\circ d'=0$,~$d''\circ d''=0$ (Lemma~\ref{lemDbicomplex1}) and~$d'\circ d''=d''\circ d'$ (Lemma~\ref{lemDbicomplex2}).
		
		\begin{lem}\label{lemDbicomplex1}
		We have~$ d'\circ d'=0$ and~$ d''\circ d''=0$ in~$\qD_{\bullet,\bullet}(\B)$.
		\end{lem}
		
		\begin{proof}
		We prove that~$d'\circ d'=0$. Let~$s\otimes X\in H^{q-2i}(S)(-i)\otimes A_{i,j}^S(\B)$, we get
		$$( d'\circ d')(s\otimes X)=\sum_{S\stackrel{1}{\hookrightarrow}T\stackrel{1}{\hookrightarrow}U}(\iota_T^U)_*(\iota_S^T)_*(s)\otimes  d'_{T,U} d'_{S,T}(X).$$
		Since~$(\iota_T^U)_*\circ(\iota_S^T)_*=(\iota_S^U)_*$, the above sum decomposes as
		$$\sum_{S\stackrel{2}{\hookrightarrow}U}\left(\iota_S^U\right)_*(s)\otimes\left(\sum_{S\stackrel{1}{\hookrightarrow}T\stackrel{1}{\hookrightarrow}U} d'_{T,U}d'_{S,T}(X)\right).$$
		For~$U$ fixed, the right-hand side of the tensor product is zero because~$A^{\leq S}_{\bullet,\bullet}(\B)$ is a bi-complex (Lemma~\ref{lemAbicomplex}). The result follows. We leave it to the reader to prove that $d''\circ d''=0$ by using exactly the same argument.
		\end{proof}
	
		The task of proving that~$ d'\circ d''= d''\circ d'$ is more intricate. We fix a stratum~$S$ and an element~$s\otimes X \in H^{q-2i}(S)(-i)\otimes A_{i,j}^S(\B)$. Let us write
		$$ d'\circ d''(s\otimes X)=\sum_{S\stackrel{1}{\hookleftarrow}R\stackrel{1}{\hookrightarrow} U} (\iota_R^U)_*(\iota_R^S)^*(s)\otimes  d'_{R,U} d''_{R,S}(X)=\sum_{U\neq S}\Sigma_1(U)+\Sigma_1'$$
		where~$\Sigma_1(U)$ is the sum over diagrams~$S\stackrel{1}{\hookleftarrow}R\stackrel{1}{\hookrightarrow} U$ and~$\Sigma_1'$ is the sum over diagrams~$S\stackrel{1}{\hookleftarrow}R\stackrel{1}{\hookrightarrow} S$. In the same fashion we write
		$$ d''\circ d'(s\otimes X)=\sum_{S\stackrel{1}{\hookrightarrow} T\stackrel{1}{\hookleftarrow} U} (\iota_U^T)^*(\iota_S^T)_*(s)\otimes  d''_{U,T} d'_{S,T}(X)=\sum_{U\neq S}\Sigma_2(U)+\Sigma'_2$$
		where~$\Sigma_2(U)$ is the sum over diagrams~$S\stackrel{1}{\hookrightarrow} T\stackrel{1}{\hookleftarrow} U$ and~$\Sigma_2'$ is the sum over diagrams~$S\stackrel{1}{\hookrightarrow} T\stackrel{1}{\hookleftarrow} S$.
		
		\begin{lem}\label{lemDbicomplex2} 
		We have the following equalities:
		\begin{enumerate}
		\item for every stratum $U\neq S$, $\Sigma_1(U)=\Sigma_2(U)$;
		\item~$\Sigma'_1=0$;
		\item~$\Sigma'_2=0$.
		\end{enumerate}
		Thus,~$ d'\circ d''= d''\circ d'$ in~$\qD_{\bullet,\bullet}(\B)$.
		\end{lem}
		
		\begin{proof}
		\begin{enumerate}
		\item We fix strata~$S\neq U$. There are three cases to consider.
		
		\textit{First case:}~$S\cap U=\varnothing$. Then~$\Sigma_1(U)=0$. For any diagram~$S\stackrel{1}{\hookrightarrow} T\stackrel{1}{\hookleftarrow} U$,~$S$ and~$U$ intersect transversely in~$T$, hence by (\ref{appiitransverse}) the composite~$\left(\iota_U^T\right)^*\circ\left(\iota_S^T\right)_*$ is zero, hence~$\Sigma_2(U)=0$.
		
		\textit{Second case:}~$S\cap U\neq \varnothing$, and there is no diagram~$S\stackrel{1}{\hookrightarrow} T\stackrel{1}{\hookleftarrow} U$. Then~$\Sigma_2(U)=0$. For every diagram~$S\stackrel{1}{\hookleftarrow}R\stackrel{1}{\hookrightarrow} U$ we have~$ d'_{R,U}d''_{S,R}(X)=0$ because~$A^{\leq R}_{\bullet,\bullet}(\B)$ is a bi-complex (Lemma~\ref{lemAbicomplex}), hence~$\Sigma_1(U)=0$.
		
		\textit{Third case:}~$S\cap U\neq \varnothing$, and there is a diagram~$S\stackrel{1}{\hookrightarrow} T\stackrel{1}{\hookleftarrow} U$. Then~$T$ is unique for dimension reasons (locally around a point of~$S\cap U$,~$T$ is the sum~$S+U$). The diagrams~$S\stackrel{1}{\hookleftarrow}R\stackrel{1}{\hookrightarrow} U$ correspond to the connected components of~$S\cap U$. For such a connected component~$R$ we have 
		$$d'_{R,U} d''_{R,S}(X)= d''_{U,T}d'_{S,T}(X)$$
		because~$A^{\leq R}_{\bullet,\bullet}(\B)$ is a bi-complex (Lemma~\ref{lemAbicomplex}). Thus
		$$\Sigma_1(U)=\left(\sum_{S\stackrel{1}{\hookleftarrow}R\stackrel{1}{\hookrightarrow} U} (\iota_R^U)_*(\iota_R^S)^*(s)\right)\otimes  d''_{U,T} d'_{S,T}(X).$$
		Using (\ref{appiitransverse}) we have 
		$$(\iota_U^T)^*(\iota_S^T)_*(s)=\sum_{S\stackrel{1}{\hookleftarrow}R\stackrel{1}{\hookrightarrow} U} (\iota_R^U)_*(\iota_R^S)^*(s)$$
		hence~$\Sigma_1(U)=(\iota_U^T)^*(\iota_S^T)_*(s)\otimes  d''_{U,T} d'_{S,T}(X)=\Sigma_2(U)$.
		
		\item For an inclusion~$R\stackrel{1}{\hookrightarrow}S$, the fact that~$A^{\leq R}_{\bullet,\bullet}(\B)$ is a bi-complex implies that we have~$ d'_{R,S}\circ d''_{R,S}=0$ (Lemma~\ref{lemAbicomplex}). The result then follows.
		
		\item We have~$$\Sigma'_2=\sum_{S\stackrel{1}{\hookrightarrow}T} \left(\iota_S^T\right)^*\left(\iota_S^T\right)_*(s)\otimes  d''_{S,T} d'_{S,T}(X).$$
		By (\ref{appiiChern}),~$\left(\iota_S^T\right)^*\left(\iota_S^T\right)_*(s)=$ is the cup-product~$c_1(N_{S/T})\,.\,s$ where~$c_1(N_{S/T})\in H^2(S)(-1)$ is the first Chern class of the normal bundle of the inclusion~$S\hookrightarrow T$. We first consider a special case.
		
		\textit{Special case:} We assume that the stratum~$S$ is irreducible. For an inclusion~$S\stackrel{1}{\hookrightarrow}T$, there exists a hypersurface~$K\in\B$ such that~$S$ is a connected component of the intersection~$T\cap K$. According to (\ref{appisonormalbundles}), we get~$c_1(N_{S/T})\cong c_1(N_{K/X})_{|S}$. Now Lemma~\ref{lemChernirreducible} below implies that~$c_1(N_{K/X})_{|S}=c$ is independent of~$K$, hence we may write
		$$\Sigma'_2=(c \,.\, s)\otimes \left(\sum_{S\stackrel{1}{\hookrightarrow}T} d''_{S,T} d'_{S,T}(X)\right).$$
		Now the fact that~$A^{\leq S}_{\bullet,\bullet}(L;M;\chi)$ is a bi-complex (Lemma~\ref{lemAbicomplex}) implies that the right-hand side of the tensor product is zero, hence the result .
		
		\textit{General case:} In general there is a (local) decomposition of~$S$ into irreducible strata. Let us assume for simplicity that this decomposition has two terms, i.e. we have a (local) decomposition into irreducibles~$S= S'\pitchfork S''$. Then an inclusion~$S\stackrel{1}{\hookrightarrow}T$ is (locally) either of the form~$T=S'\pitchfork T''$ for~$S''\stackrel{1}{\hookrightarrow}T''$ or of the form~$T=T'\pitchfork S''$ for~$S'\stackrel{1}{\hookrightarrow} T'$. Using the K\"unneth formula (Proposition~\ref{propkunneth}) for the Orlik--Solomon bi-complex, we may then split~$\Sigma_2$ into two sums. One gets the result by applying the same reasoning as in the first case to each of these two sums.
		\end{enumerate}
		\end{proof}
		
		We have used the following lemma.
		
		\begin{lem}\label{lemChernirreducible}
		Let~$\A$ be an arrangement of hypersurfaces in a complex manifold~$X$, and~$S$ an irreducible stratum of~$\A$. Then the line bundles~$\left(N_{K/X}\right)_{|S}$, for~$K\in\A$ such that~$K\supset S$, are all isomorphic.
		\end{lem}
		
		\begin{proof}
		Let us write~$\A^{\leq S}=\{K_1,\ldots,K_r\}$ for the hypersurfaces of~$\A$ that contain~$S$. Let~$i, j\in\{1,\ldots,r\}$. We first consider a special case.
		
		\textit{Special case:} Let us first assume that~$\{1,\ldots,r\}$ is a circuit. Let~$T$ be the connected component of~$K_1\cap\cdots\cap\widehat{K_i}\cap\cdots\cap\widehat{K_j}\cap\cdots\cap K_r$ that contains~$S$. We then have an inclusion~$S\stackrel{1}{\hookrightarrow}T$,~$S$ being at the same time a connected component of~$K_i\cap T$ and~$K_j\cap T$. From (\ref{appisonormalbundles}) we deduce isomorphisms
		$$\left(N_{K_i/X}\right)_{|S}\cong N_{S/T} \cong \left(N_{K_j/X}\right)_{|S}.$$
		\textit{General case:} One may reduce to the special case above by using Lemma~\ref{lemmatroids}.
		\end{proof}
	
	\subsection{Blow-ups and the geometric Orlik--Solomon bi-complex}
	
		We now define a morphism between the geometric Orlik--Solomon bi-complex of a bi-arrangement of hypersurfaces~$\B$ and that of its blow-up~$\td{\B}$. For all the rest of this article, we make the following assumption on bi-arrangements of hypersurfaces: 
		\begin{equation}\label{connectednessassumption}
		\textit{any intersection of strata is connected} 
		\end{equation}
		 (this includes the empty case). Equivalently, this means that the intersection of any number of hypersurfaces~$K\in\B$ is connected. 
		 
		 This assumption is not necessary, and we will sketch in \S\ref{parwithoutconnectedness} how to deal with the general case. However, working under the assumption (\ref{connectednessassumption}) makes the discussion and the computations more accessible to the reader by keeping the notations light. One may note that (\ref{connectednessassumption}) is satisfied by all the examples of arrangements of hypersurfaces introduced in Example~\ref{exarrangementshypersurfaces}, and is stable by blow-up.
		
		\subsubsection{The framework}	
	
		We fix a bi-arrangement of hypersurfaces~$\B$ in a complex manifold~$X$, and a good stratum~$Z$ for~$\B$.
		
		\begin{defi}
		An inclusion~$S\stackrel{1}{\hookrightarrow}T$ of strata of~$\B$ has \textit{parallel type} with respect to~$Z$ if~$Z\cap S\neq\varnothing$ and~$Z\cap S=Z\cap T$. In this case we write~$S\hookrightpar T$.
		\end{defi}
		
		In view of assumption (\ref{connectednessassumption}),~$Z\cap S$ is connected and this may be checked locally. Around any point of~$Z\cap S$, there is a decomposition~$Z\pitchfork W$, hence one has a decomposition~$S=S_\parallel \pitchfork S_\perp$ with~$S_\parallel\supset Z$ and~$S_\perp\supset W$. For an inclusion~$S\stackrel{1}{\hookrightarrow}T$ of strata, we then have two mutually exclusive cases.
		\begin{enumerate}[--]
		\item ~$T=T_\parallel\pitchfork S_\perp$ with~$S_\parallel\stackrel{1}{\hookrightarrow}T_\parallel$;
		\item ~$T=S_\parallel\pitchfork T_\perp$ with~$S_\perp\stackrel{1}{\hookrightarrow}T_\perp$.
		\end{enumerate}
		The parallel type corresponds to the first case:~$Z\cap S=Z\cap T= Z\pitchfork S_\perp$.\\
		
		Let~$\pi:\td{X}\rightarrow X$ be the blow-up along~$Z$. For every stratum~$S$, it restricts to~$\pi_S^{\td{S}}:\td{S}\rightarrow S$ the blow-up along~$Z\cap S$.
		
		In the case~$S\hookrightpar T$, we have~$Z\cap S=Z\cap T$, hence~$\pi$ induces a morphism 
		$$\pi^{E\cap\td{T}}_{Z\cap S}:E\cap\td{T}\rightarrow Z\cap T=Z\cap S.$$
		
		\subsubsection{Definition of~$\Phi$}\label{pardefPhi}
		
		Let us assume that we have~$\chi(Z)=\lambda$. We recall that we have made explicit the Orlik--Solomon bi-complex of a blow-up in Proposition~\ref{propOSbicomplexblowup}. Having this in mind, we define a morphism
		\begin{equation}\label{eqPhi}
		\Phi:\qD_{i,j}(\B)\rightarrow \qD_{i,j}(\td{\B}).
		\end{equation}
		Let~$S\in\s_{i+j}(\B)$ be a stratum. We define, for~$s\otimes X\in H^{q-2i}(S)(-i)\otimes A_{i,j}^S(\B)=\qD_{i,j}^S(\B)$, 
		\begin{equation}\label{eqdefPhi}
		\Phi(s\otimes X)=(\pi_S^{\td{S}})^*(s)\otimes X +\sum_{S\hookrightpar T} (\pi^{E\cap\td{T}}_{Z\cap S})^*(\iota^S_{Z\cap S})^*(s)\otimes d'_{S,T}(X).
		\end{equation}
		
		Let us explain more precisely the meaning of this formula:
		\begin{enumerate}[--]
		\item the term~$(\pi_S^{\td{S}})^*(s)\otimes X$ lives in~$H^{q-2i}(\td{S})(-i)\otimes A_{i,j}^S=\qD_{i,j}^{\td{S}}(\td{\B})$; if~$S\subset Z$ then~$\td{S}=\varnothing$ and this is zero by convention;
		\item the term~$(\pi^{E\cap\td{T}}_{Z\cap S})^*(\iota^S_{Z\cap S})^*(s)\otimes d'_{S,T}(X)$ lives in~$H^{q-2i}(E\cap\td{T})(-i)\otimes A_{i-1,j}^T(\B)=\qD_{i,j}^{E\cap\td{T}}(\td{\B})$.
		\end{enumerate}
		
		In the algebraic case,~$\Phi$ is a morphism of mixed Hodge structures.\\
		
		The motivation for formula (\ref{eqdefPhi}) comes from the case of normal crossing divisors, as the next lemma shows.
		
		\begin{lem}
		If~$\L\cup\M$ is a normal crossing divisor, then the total complex~$\qD_\bullet$ is then~$E_1^{-\bullet,q}$ term of the natural spectral sequence (\ref{appeqspectralsequence}) that computes the relative cohomology groups~$H^\bullet(X\setminus \L,\M\setminus \M\cap\L)$. In this case, the morphism~$\qD_\bullet(\L,\M)\rightarrow \qD_\bullet(\td{\L},\td{\M})$ induced by~$\Phi$ is the natural morphism (\ref{appPhiNCD}) that expresses the isomorphism~$\pi^*:H^\bullet(X\setminus \L,\M\setminus \M\cap\L)\stackrel{\cong}{\longrightarrow} H^\bullet(\td{X}\setminus \td{\L},\td{\M}\setminus \td{\M}\cap\td{\L})$.
		\end{lem}
		
		\begin{proof}
		It follows from a direct comparison of the formulas since by definition
		$$d'(e_I\otimes f_J^\vee)=\sum_{i\in I}\sgn(\{i\},I\setminus\{i\}) \, e_{I\setminus\{i\}}\otimes f_J^\vee.$$
		\end{proof}
		
		If we now assume that~$\chi(Z)=\mu$, then we are in the dual situation and we may define a morphism
		
		\begin{equation}\label{eqPsi}
		\Psi:\qD_{i,j}(\td{\B})\rightarrow \qD_{i,j}(\B).
		\end{equation}
		by the formulas
		$$\Psi(\td{s}\otimes X)=(\pi_S^{\td{S}})_*(\td{s})\otimes X \;\textnormal{ and }\; \Psi(\td{e}\otimes X)= \sum_{S\hookrightpar T} (\iota_{Z\cap S}^S)_*(\pi_{Z\cap S}^{E\cap\td{T}})_*(e)\otimes d''_{S,T}(X)$$
		for~$\td{s}\otimes X\in H^{q-2i}(\td{S})(-i)\otimes A_{i,j}^S(\B)$ and~$e\otimes X\in H^{q-2i}(E\cap\td{T})(-i)\otimes A_{i,j-1}^T(\B)$.

		\subsubsection{The essential case}
		
			\begin{defi}
			Let~$\B$ be an bi-arrangement of hypersurfaces in a complex manifold~$X$. We say that~$\B$ is \textit{essential} if the intersection~$\bigcap_{K\in\B}K$ of all hypersurfaces in~$\B$ is non-empty. 
			\end{defi}
			
			According to assumption (\ref{connectednessassumption}), the intersection~$Z=\bigcap_{K\in\B}K$ is the minimal stratum of~$\B$. It is necessarily a good stratum. In this case, formula (\ref{eqdefPhi}) takes a simpler form. Indeed, we always have~$Z\cap S=Z$, and all inclusions~$S\stackrel{1}{\hookrightarrow}T$ are of the form~$S\hookrightpar T$. Hence we get
			
			$$\Phi(s\otimes X)=(\pi_S^{\td{S}})^*(s)\otimes X +\sum_{S\stackrel{1}{\hookrightarrow} T} (\pi^{E\cap\td{T}}_{Z})^*(\iota^S_Z)^*(s)\otimes d'_{S,T}(X).$$
		
			If~$S=Z$ the formula simply reads, for~$z\otimes X\in H^{q-2i}(Z)(-i)\otimes A_{i,j}^Z(\B)=\qD_{i,j}^Z(\B)$:
			$$\Phi(z\otimes X)=\sum_{Z\stackrel{1}{\hookrightarrow}T}(\iota_Z^{E\cap\td{T}})^*(z)\otimes d'_{Z,T}(X).$$
	
	\subsection{The main theorem}
	
		The following theorem will be proved in \S\ref{parproof}.
	
		\begin{thm}\label{maintheoremtechnical}
		Let~$\B$ be a bi-arrangement of hypersurfaces in a complex manifold~$X$, let~$Z$ be a good stratum of~$\B$ such that~$\chi(Z)=\lambda$, and let~$\td{\B}$ be the blow-up of~$\B$ along~$Z$.
		\begin{enumerate}
		\item Formula (\ref{eqdefPhi}) defines a morphism of bi-complexes~$\Phi:\qD_{\bullet,\bullet}(\B)\rightarrow \qD_{\bullet,\bullet}(\td{\B})$.
		\item If~$Z$ is exact, then the morphism~$\Phi:\qD_\bullet(\B)\rightarrow \qD_\bullet(\td{\B})$ induced on the total complexes is a quasi-isomorphism.
		\end{enumerate}
		\end{thm}

		\begin{rem}\label{remmaintechnicaldual}
		We will also use the dual counterpart of Theorem \ref{maintheoremtechnical}, whose proof is dual and left to the reader. If~$Z$ is a good stratum of~$\B$ such that~$\chi(Z)=\mu$, then we have a quasi-isomorphism~$\Psi: \qD_\bullet(\td{\B})\rightarrow \qD_\bullet(\B)$.
		\end{rem}	
		
		
		\begin{thm}\label{maintheorem}
		Let~$\B$ be an exact bi-arrangement of hypersurfaces in a complex manifold~$X$.
		\begin{enumerate}
		\item There is a spectral sequence
		\begin{equation}\label{eqspectralsequenceOS}
		E_1^{-p,q}(\B)=\qD_p(\B) \;\Longrightarrow \; H^{-p+q}(\B).
		\end{equation}
		\item If~$X$ is a smooth complex variety and all hypersurfaces of~$\B$ are divisors in~$X$, then this is a spectral sequence in the category of mixed Hodge structures.
		\item If~$X$ is a smooth and projective complex variety, then this spectral sequence degenerates at the~$E_2$ term and we have
		$$E_\infty^{-p,q}\cong E_2^{-p+q} \cong \gr_q^W H^{-p+q}(\B).$$
		\end{enumerate}
		\end{thm}
		
		\begin{proof}
		\begin{enumerate}
		\item Let~$X^{(\infty)}=X^{(N)}\rightarrow X^{(N-1)} \rightarrow \cdots \rightarrow X^{(1)} \rightarrow X^{(0)}=X$ be the sequence of blow-ups used to define the motive of~$\B$ (Definition~\ref{defimotive}) and~$\B^{(\infty)}=\B^{(N)},\B^{(N-1)},\ldots, \B^{(1)},\B^{(0)}=\B$ be the corresponding bi-arrangements of hypersurfaces, with~$\B^{(\infty)}=(\L^{(\infty)},\M^{(\infty)})$ a normal crossing divisor. According to Proposition~\ref{appspectralsequenceNCD}, there is a spectral sequence
		\begin{equation}\label{eqssintheproof}
		E_1^{-p,q}=\qD_p(\B^{(\infty) })  \;\Longrightarrow \; H^{-p+q}(\B).
		\end{equation}
		Since~$\B$ is exact, Corollary~\ref{coroblowupexact} implies that for each~$k$,~$\B^{(k)}$ is exact. Then for each~$k$, Theorem~\ref{maintheoremtechnical} or its dual counterpart (see Remark \ref{remmaintechnicaldual}) implies that there is a quasi-isomorphism~$\qD_{\bullet}(\B^{(k)})\sim \qD_\bullet(\B^{(k+1)})$. Thus we have a quasi-isomorphism~$\qD_\bullet(\B)\sim \qD_\bullet(\B^{(\infty)})$; we can thus replace the~$E_1$ page of the spectral sequence (\ref{eqssintheproof}) by the collections of the complexes~$\qD_\bullet(\B)$, and the result follows (the other pages of the spectral sequence are unchanged).
		\item This follows from the analogous statement for normal crossing divisors (Proposition~\ref{appspectralsequenceNCD}) and the fact that the morphisms (\ref{eqPhi}) are morphisms of mixed Hodge structures.
		\item This follows from the analogous statement for normal crossing divisors (Proposition~\ref{appspectralsequenceNCD}).
		\end{enumerate}
		\end{proof}
		
		\begin{rem}
		In the case of an arrangement of hypersurfaces~$(\A,\lambda)$, Theorem~\ref{maintheorem} gives a spectral sequence
		$$E_1^{-p,q}=\bigoplus_{S\in\s_p(\A)}H^{q-2p}(S)(-p)\otimes A_p^S(\A)  \;\Longrightarrow \; H^{-p+q}(X\setminus \A)~$$
		which was first defined in~\cite{looijenga} and studied in~\cite{duponthypersurface} in the context of logarithmic differential forms and mixed Hodge theory.
		\end{rem}
		
		\begin{rem}
		The spectral sequence (\ref{eqspectralsequenceOS}) is independent of the choices made during the blow-up process. This will be proved in a subsequent article, as well as other functoriality properties of this spectral sequence with respect to the change of wonderful compactification and deletion/restriction.
		\end{rem}

		
		
		
		
		

\section{Application to projective bi-arrangements}\label{sectionprojective}

	\subsection{The setup}

		Let~$\A$ be an arrangement in~$\C^{n+1}$ with~$n\geq 1$. We let~$\P\A$ be the corresponding projective arrangement in~$\P^n(\C)$; it is an arrangement of hypersurfaces consisting of the images~$\P K$ of the hyperplanes~$K\in\A$ by the projection~$\C^{n+1}\setminus 0\rightarrow \P^n(\C)$. The strata of~$\P\A$ are the images~$\P S$ of the strata~$S\neq 0$ of~$\A$. We implicitly assume that~$0$ is a stratum of~$\A$.
		
		A partial coloring function~$\chi:\s_+(\A)\setminus\{0\}\rightarrow\{\lambda,\mu\}$ that satisfies the K\"{u}nneth condition (\ref{kunnethcondition}) gives rise to a \textit{projective bi-arrangement}~$\P\B=(\P\A,\chi)$ where we put~$\chi(\P S)=\chi(S)$. It is a bi-arrangement of hypersurfaces in~$X=\P^n(\C)$. 
		
		This projective bi-arrangement does not necessarily come from a bi-arrangement~$\B=(\A,\chi)$ since the color~$\chi(0)$ is not defined. We will write~$\B_\lambda$ (resp.~$\B_\mu$) for the bi-arrangements~$(\A,\chi)$ with~$\chi(0)=\lambda$ (resp.~$\chi(0)=\mu$), if they are well-defined (i.e. if they satisfy the K\"{u}nneth condition for the stratum~$0$).
		
		There is a partial Orlik--Solomon bi-complex~$A_{\bullet,\bullet}(\B)$ where we have vector spaces~$A_{i,j}^S(\B)$ for strata~$S\neq 0$. If~$\B_\lambda$ (resp.~$\B_\mu$) are well-defined, then it can be completed to an Orlik--Solomon bi-complex~$A_{\bullet,\bullet}(\B_\lambda)$ (resp.~$A_{\bullet,\bullet}(\B_\mu)$).
		

		\begin{rem}\label{remcohomologyproj}
		 For a projective space~$\P^r(\C)$ we have canonical isomorphisms~$H^{2k}(\P^r(\C))\cong \Q(-k)$ for~$k=0,\ldots,r$, and~$H^{2k+1}(\P^r(\C))=0$ for all~$k$. Furthermore, for the inclusion~$\iota:\P^{r-1}(\C)\hookrightarrow \P^r(\C)$ of a projective hyperplane:
		\begin{enumerate}[--]
		\item the Gysin morphism~$\iota_*:H^{2(k-1)}(\P^{r-1}(\C))(-1)\rightarrow H^{2k}(\P^r(\C))$ is the identity of~$\Q(-k)$ for~$k=1,\ldots,r$;
		\item the pull-back morphism~$\iota^*:H^{2k}(\P^r(\C))\rightarrow H^{2k}(\P^{r-1}(\C))$ is the identity of~$\Q(-k)$ for~$k=0,\ldots,r-1$.
		\end{enumerate}
		\end{rem}
		
		The next proposition expresses the (geometric) Orlik--Solomon bi-complex of~$\P\B$ in terms of that of~$\B$.
		
		\begin{prop}\label{propOSBPB}
		\begin{enumerate}
		\item We have isomorphisms~$$A^{\P S}_{i,j}(\P\B)\cong A^S_{i,j}(\B)$$
		for~$S\in\s_{i+j}(\B)$,~$S\neq 0$, which induce isomorphisms of bi-complexes 
		$$A^{\leq\P\Sigma}_{\bullet,\bullet}(\P\B)\cong A^{\leq\Sigma}_{\bullet,\bullet}(\B)$$
		for~$\Sigma\neq 0$ a strict stratum of~$\B$. 
		\item We have~${}^{(q)}D_{i,j}^{\P S}(\P\B) =0$ for~$q$ odd.
		For~$k=0,\ldots,n$ we have isomorphisms of pure Hodge structures of weight~$2k$:
		$${}^{(2k)}D_{i,j}^{\P S}(\P\B) \cong \begin{cases} A_{i,j}^S(\B)(-k) & \textnormal{ if } 0\leq i\leq k \textnormal{ and } 0\leq j\leq n-k; \\ 0 & \textnormal{ otherwise.} \end{cases}~$$
		Furthermore, these isomorphisms are compatible with the differentials~$d'$ and~$d''$.
		\end{enumerate}
		\end{prop}
		
		\begin{proof}
		\begin{enumerate}
		\item It is trivial.
		\item The first statement comes from the fact (Remark~\ref{remcohomologyproj}) that the projective spaces do not have cohomology in odd degree. For~$k=0,\ldots,n$, we have~${}^{(2k)}D_{i,j}^{\P S}=H^{2(k-i)}(\P S)(-i)\otimes A_{i,j}^{\P S}$. The cohomology group~$H^{2(k-i)}(\P S)(-i)$ is non-zero if and only if~$0\leq k-i \leq n-\mathrm{codim}(S)=n-i-j$, which amounts to~$i\leq k$ and~$j\leq n-k$. In this range, we have a canonical isomorphism~$H^{2(k-i)}(\P S)(-i)\cong\Q(-k)$, hence the result. The compatibility with the differentials come from Remark~\ref{remcohomologyproj}.
		\end{enumerate}
		\end{proof}
		
		According to Proposition~\ref{propOSBPB},~$\P\B$ is exact if all strict strata~$\Sigma\neq 0$ of~$\B$ are exact. It is actually convenient to ask for more and make the following definition.
		
		\begin{defi}
		We say that~$\P\B$ is \textit{$\lambda$-exact} (resp. \textit{$\mu$-exact}) if~$\B_\lambda$ (resp.~$\B_\mu$) is well-defined and exact. We say that~$\P\B$ is \textit{strongly exact} if it is~$\lambda$-exact and~$\mu$-exact.
		\end{defi}
		
		We define in the same fashion the concepts of~$\lambda$-tame,~$\mu$-tame and strongly tame projective bi-arrangements; for such bi-arrangements, Theorem~\ref{thmtameOS} provides an explicit presentation of the Orlik--Solomon bi-complex.
		
		Let us then write~$\kA_{\bullet,\bullet}(\B)$ for the bi-complex obtained from~$A_{\bullet,\bullet}(\B)$ by keeping only the rectangle~$0\leq i\leq k$,~$0\leq j\leq n-k$. We write~$\kA_\bullet(\B)$ for its total complex, which means that we set 
		$$\kA_r(\B)=\bigoplus_{\substack{i-j=r\\ 0\leq i\leq k\\ 0\leq j\leq n-k}} A_{i,j}(\B).$$
		
		\begin{thm}\label{thmprojective}
		Let~$\P\B$ be a projective bi-arrangement in~$\P^n(\C)$.
		\begin{enumerate} 
		\item If~$\P\B$ is exact then we have isomorphisms, for~$r=0,\ldots,2n$:
		$$\gr_{2k}^W H^r(\P\B)\cong H_{2k-r}(\kA_\bullet(\B)).$$
		\item If~$\P\B$ is~$\lambda$-exact then we have~$H^r(\P\B)=0$ for~$r>n$. For~$r=0,\ldots,n$ we have isomorphisms, for~$k=0,\ldots,r$:
		$$\gr_{2k}^W H^r(\P\B) \cong \mathrm{coker}\left( A_{k+1,r-k-1}(\B_\lambda)\stackrel{d''}{\longrightarrow} A_{k+1,r-k}(\B_\lambda)\right).$$
		\item If~$\P\B$ is~$\mu$-exact then we have~$H^r(\P\B)=0$ for~$r<n$. For~$r=n,\ldots,2n$ we have isomorphisms, for~$k=n-r,\ldots,n$:
		$$\gr_{2k}^W H^r(\P\B) \cong \mathrm{ker}\left( A_{r-n+k,n-k+1}(\B_\mu)\stackrel{d'}{\longrightarrow} A_{r-n+k-1,n-k+1}(\B_\mu)\right).$$
		\item If~$\P\B$ is strongly exact then we have~$H^r(\P\B)=0$ for~$r\neq n$, and we have isomorphisms, for~$k=0,\ldots,n$:
		\begin{eqnarray*}
		\gr_{2k}^W H^n(\P\B)  & \cong  &\mathrm{coker}\left( A_{k+1,n-k-1}(\B_\lambda)\stackrel{d''}{\longrightarrow} A_{k+1,n-k}(\B_\lambda)\right) \\
		 & \cong  & \mathrm{ker}\left( A_{k,n-k+1}(\B_\mu)\stackrel{d'}{\longrightarrow} A_{k-1,n-k+1}(\B_\mu)\right).
		\end{eqnarray*}
		\end{enumerate}
		\end{thm}
		
		\begin{proof}
		\begin{enumerate}
		\item This is a consequence of Theorem~\ref{maintheorem} and Proposition~\ref{propOSBPB}.
		\item The differential~$d':A_{k+1,r-k}(\B_\lambda)\rightarrow A_{k,r-k}(\B_\lambda)=A_{k,i-k}(\B)$ induces a morphism~$A_{k+1,r-k}(\B_\lambda)\rightarrow H_{2k-r}(\kA_\bullet(\B))$. A diagram chase shows that it induces an isomorphism as in the statement.
		\item This is the dual of 3.
		\item This is a consequence of 2 and 3.
		\end{enumerate}
		\end{proof}
		
		\begin{rem}
		In the case of a projective arrangement of hyperplanes~$(\A,\lambda)$, the cohomology group~$H^k(\A,\lambda)=H^k(\P^n(\C)\setminus \A)$ is concentrated in weight~$2k$ and Theorem~\ref{thmprojective} gives the isomorphism
		$$H^k(\P^n(\C)\setminus \A)\cong \mathrm{ker}\left(A_k(\A)\stackrel{d}{\longrightarrow} A_{k-1}(\A)\right)$$
		which is the projective version of the Brieskorn--Orlik--Solomon theorem~\cite{brieskorn,orliksolomon}.
		\end{rem}
		
 	\subsection{Multizeta bi-arrangements}\label{parmultizetabiarrangements}
 	
 		Let~$r\geq 1$ and~$n_1,\ldots,n_r$ be integers with~$n_1,\ldots,n_{r-1}\geq 1$ and~$n_r\geq 2$. Generalizing the zeta values (\ref{eqzetasum}), one defines the multiple zeta values
 		$$\zeta(n_1,\ldots,n_r)=\sum_{1\leq k_1<\cdots <k_r}\dfrac{1}{k_1^{n_1}\cdots k_r^{n_r}}\cdot$$
		 We let~$n=n_1+\cdots+n_r$ and define an~$n$-tuple 
		$$(a_1,\ldots,a_n)=(\underbrace{1,0,\ldots,0}_{n_1},\ldots,\underbrace{1,0,\ldots,0}_{n_r}).$$
		It has been noticed by Kontsevich that the integral formula (\ref{eqzetaintegral}) has the following generalization:
		$$\zeta(n_1,\ldots,n_r)=(-1)^r\int_{0<x_1<\cdots<x_n<1}\frac{dx_1}{x_1-a_1}\cdots\dfrac{dx_n}{x_n-a_n}\cdot$$
		We may then generalize the discussion of \S\ref{parintroperiods} and produce motives that have a certain multiple zeta value as a period, as follow. In~$\P^n(\C)$ with projective coordinates~$(z_0,z_1,\ldots,z_n)$, we define two arrangements of hyperplanes~$\L=\{L_0,\ldots,L_n\}$ and~$\M=\{M_0,\ldots,M_n\}$:
		\begin{enumerate}[--]
		\item we let~$L_0=\{z_0=0\}$ be the hyperplane at infinity, and for~$k=1,\ldots,n$, let~$L_k=\{z_k=a_kz_0\}$; 
		\item we let~$M_0=\{z_1=0\}$,~$M_n=\{z_n=z_0\}$, and for~$k=1,\ldots,n-1$,~$M_k=\{z_k=z_{k+1}\}$.
		\end{enumerate}
		
		 \begin{defi}
		 The \textit{multizeta bi-arrangement}~$\mathscr{Z}(n_1,\ldots,n_r)$ is the projective bi-arrangement~$(\L,\M,\chi)$ where~$\chi$ is defined to be~$\mu$ on all~$\M$-strata, and~$\lambda$ on all other strata.
		 \end{defi}
		 
		 One can check that this indeed defines a projective bi-arrangement. Using the same argument as in \S\ref{parintroperiods}, we can show (see~\cite{dupontPhD}) that the multiple zeta value~$\zeta(n_1,\ldots,n_r)$ is a period of the motive~$H^n(\mathscr{Z}(n_1,\ldots,n_r))$, which is (the Hodge realization of) a mixed Tate motive over~$\mathbb{Z}$.\\
		 
		 The following proposition is easily proved by direct inspection.
		 
		 \begin{prop}
		 The multizeta bi-arrangements~$\mathscr{Z}(n_1,\ldots,n_r)$ are all~$\lambda$-tame, hence~$\lambda$-exact.
		 \end{prop}
		 
		 Thus, Theorem~\ref{thmprojective} gives explicit complexes that compute the motive of multizeta bi-arrangements.\\
		 
		 When studying multiple zeta values, the motives~$H^n(\mathscr{Z}(n_1,\ldots,n_r))$ are alternatives, in the spirit of~\cite{terasoma,goncharovmanin}, to the approach of~\cite{delignedroiteprojective, delignegoncharov} which uses the motivic fundamental group of~$\mathbb{P}^1\setminus\{\infty,0,1\}$. One advantage of such an alternative is that it generalizes to a larger family of integrals. More specifically, let us look at the periods of the moduli spaces~$\overline{\mathcal{M}}_{0,n}$ considered by Brown in~\cite{brownPhD}. They are integrals of a rational function over a simplex~$0<t_1<\cdots<t_n<1$, such as
		 \begin{equation}\label{eqintcellzeta}
		 \int_{0<x<y<z<1}\frac{dx\, dy\, dz}{(1-x)y(z-x)}\cdot
		 \end{equation}
		 The main result of~\cite{brownPhD} is that these integrals are all linear combinations (with rational coefficients) of multiple zeta values, although not in an explicit way. It so happens that the projective bi-arrangement of hyperplanes corresponding to the integral (\ref{eqintcellzeta}) is also~$\lambda$-exact, hence the corresponding motive may be computed explicitly via an Orlik--Solomon bi-complex. This will be studied in more detail in a subsequent article.
	
\section{Proof of the main theorem}\label{parproof}

		The goal of this section is to prove the two points of Theorem~\ref{maintheoremtechnical}. We first deal with the essential case (\S\ref{parPhimorphismessential} and \S\ref{parPhiqisessential}) then with the general case (\S\ref{parPhimorphismgeneral} and \S\ref{parPhiqisgeneral}). The reader is encouraged to focus on the essential case, since the general case reduces to the essential case at the cost of a few technical trivialities.

	\subsection{The essential case:~$\Phi$ is a morphism of bi-complexes}\label{parPhimorphismessential}
	
		In this paragraph, we assume that~$\B$ is essential and that~$Z$ is the minimal stratum. We prove the first point of Theorem~\ref{maintheoremtechnical} by showing that~$\Phi$ is compatible with~$d'$ (Proposition~\ref{propcompd'}) and with~$d''$ (Proposition~\ref{propcompd''}).
	
		\begin{prop}\label{propcompd'}
		We have~$\Phi\circ d'= d'\circ\Phi$.
		\end{prop}
		
		\begin{proof}
		For~$s\otimes X\in H^{q-2i}(S)(-i)\otimes A_{i,j}^S(\B)$, we compute
		\begin{align*}
		(\Phi d')(s\otimes X)= & \sum_{S\stackrel{1}{\hookrightarrow}T} (\pi^{\td{T}}_T)^*(\iota_S^T)_*(s)\otimes  d'_{S,T}(X) \\ 
		& +\sum_{S\stackrel{1}{\hookrightarrow}T\stackrel{1}{\hookrightarrow}U}(\pi^{E\cap\td{U}}_{Z})^*(\iota^{T}_{Z})^*(\iota_S^T)_*(s)\otimes  d'_{T,U} d'_{S,T}(X).
		\end{align*}
		and
		\begin{align*}
		( d'\Phi)(s\otimes X)= &\sum_{S\stackrel{1}{\hookrightarrow}T}(\iota_{\td{S}}^{\td{T}})_*(\pi^{\td{S}}_S)^*(s)\otimes  d'_{S,T}(X) \\
		& + \sum_{S \stackrel{1}{\hookrightarrow} T} (\iota_{E\cap\td{T}}^{\td{T}})_*(\pi_{Z}^{E\cap \td{T}})^*(\iota_{Z}^{S})^*(s)\otimes d'_{S,T}(X)\\
		& -\sum_{S \stackrel{1}{\hookrightarrow} T\stackrel{1}{\hookrightarrow} U} (\iota^{E\cap\td{U}}_{E\cap\td{T}})_*(\pi^{E\cap\td{T}}_{Z})^*(\iota_{Z}^S)^*(s)\otimes d'_{T,U} d'_{S,T}(X).
		\end{align*}
		The terms~$(\cdots)\otimes d'_{S,T}(X)$ cancel because of the equality
		$$(\pi^{\td{T}}_T)^*(\iota_S^T)_*(s)=(\iota_{\td{S}}^{\td{T}})_*(\pi^{\td{S}}_S)^*(s)+(\iota_{E\cap\td{T}}^{\td{T}})_*(\pi_{Z}^{E\cap \td{T}})^*(\iota_{Z}^{S})^*(s)$$
		which is a special case of (\ref{apppullbackpar}) (set~$L=S$ and~$X=T$).
		Thus, it remains to show that we have, for every~$U$ fixed,
		$$\sum_{S\stackrel{1}{\hookrightarrow}T\stackrel{1}{\hookrightarrow}U}\Delta_{T}\otimes  d'_{T,U} d'_{S,T}(X)=0$$
		where we have set 
		$$\Delta_{T}=(\pi^{E\cap\td{U}}_{Z})^*(\iota^{T}_{Z})^*(\iota_S^T)_*(s)+(\iota^{E\cap\td{U}}_{E\cap\td{T}})_*(\pi^{E\cap\td{T}}_{Z})^*(\iota_{Z}^S)^*(s).$$
		For~$S$ and~$U$ fixed, the fact that~$A^{\leq S}_{\bullet,\bullet}(\B)$ is a bi-complex (Lemma~\ref{lemAbicomplex}) implies that we have
		$$\sum_{S\stackrel{1}{\hookrightarrow}T\stackrel{1}{\hookrightarrow}U} d'_{T,U} d'_{S,T}(X)=0.$$
		Thus, we are done if we prove that~$\Delta_{T}$ is independent of~$T$.
		On the one hand we have
		$$(\iota^{T}_{Z})^*(\iota_S^T)_*(s)=(\iota_Z^S)^*(\iota_S^T)^*(\iota_S^T)_*(s)=(\iota_Z^S)^*(s\,.\,c_1(N_{S/T}))$$
		where we have used (\ref{appiiChern}). On the other hand, we may use (\ref{appidentitypar}) to get
		$$(\iota^{E\cap\td{U}}_{E\cap\td{T}})_*(\pi^{E\cap\td{T}}_{Z})^*(\iota_{Z}^S)^*(s) = (\pi_Z^{E\cap\td{U}})^*(\iota_Z^S)^*(s)\,.\,\left((\pi_Z^{E\cap\td{U}})^*(c_1(N_{T/U})_{|Z})-c_1(N_{E\cap\td{U}/\td{U}})\right).$$
		Thus, we may rewrite
		$$\Delta_{T}=(\pi_Z^{E\cap\td{U}})^*(\iota_Z^S)^*(s)\,.\,\Delta'_{T}$$
		with
		$$\Delta'_{T}=(\pi_Z^{E\cap\td{U}})^*(c_1(N_{S/T})_{|Z}+c_1(N_{T/U})_{|Z})-c_1(N_{E\cap\td{U}/\td{U}}).$$
		Using (\ref{apptransChern}) we get~$c_1(N_{S/T})+c_1(N_{T/U})_{|S}=c_1(N_{S/U})$ and hence
		$$\Delta'_{T}=(\pi_Z^{E\cap\td{U}})^*(c_1(N_{S/U})_{|Z})-c_1(N_{E\cap\td{U}/\td{U}})$$
		which is independent of~$T$, hence~$\Delta_{T}$ is independent of~$T$ and we are done.
		\end{proof}
		
		\begin{prop}\label{propcompd''}
		We have~$\Phi\circ d''=d''\circ\Phi$.
		\end{prop}
		
		\begin{proof}
		For~$s\otimes X\in H^{q-2i}(S)(-i)\otimes A_{i,j}^S(\B)$, we compute
		\begin{align*}
		(\Phi d'')(s\otimes X)= & \sum_{S\stackrel{1}{\hookleftarrow}R} (\pi^{\td{R}}_R)^*(\iota_R^S)^*(s)\otimes  d''_{R,S}(X) \\ 
		& +\sum_{S\stackrel{1}{\hookleftarrow}R\stackrel{1}{\hookrightarrow} U}(\pi^{E\cap\td{U}}_{Z})^*(\iota^{R}_{Z})^*(\iota_R^S)^*(s)\otimes  d'_{R,U} d''_{R,S}(X).
		\end{align*}
		and
		\begin{align*}
		( d''\Phi)(s\otimes X)= &\sum_{S\stackrel{1}{\hookleftarrow}R}(\iota_{\td{R}}^{\td{S}})^*(\pi^{\td{S}}_S)^*(s)\otimes  d''_{R,S}(X) \\
		& +\sum_{S \stackrel{1}{\hookrightarrow} T\stackrel{1}{\hookleftarrow} U} (\iota^{E\cap\td{T}}_{E\cap\td{U}})^*(\pi^{E\cap\td{T}}_{Z})^*(\iota_{Z}^S)^*(s)\otimes d''_{U,T} d'_{S,T}(X).
		\end{align*}
		The terms~$(\cdots)\otimes d''_{R,S}(X)$ cancel because of the equality~$(\pi^{\td{R}}_R)^*(\iota_R^S)^*(s)=(\iota_{\td{R}}^{\td{S}})^*(\pi^{\td{S}}_S)^*(s)$, which follows from~$(\iota_R^S)\circ (\pi^{\td{R}}_R)=(\pi^{\td{S}}_S)\circ (\iota_{\td{R}}^{\td{S}})$. Thus it remains to show that, for~$U$ fixed, we have 
		$$\sum_{S\stackrel{1}{\hookleftarrow}R\stackrel{1}{\hookrightarrow} U}(\pi^{E\cap\td{U}}_{Z})^*(\iota^{R}_{Z})^*(\iota_R^S)^*(s)\otimes  d'_{R,U} d''_{R,S}(X)=\sum_{S \stackrel{1}{\hookrightarrow} T\stackrel{1}{\hookleftarrow} U} (\iota^{E\cap\td{T}}_{E\cap\td{U}})^*(\pi^{E\cap\td{T}}_{Z})^*(\iota_{Z}^S)^*(s)\otimes d''_{U,T} d'_{S,T}(X).$$
		Now~$(\pi^{E\cap\td{U}}_{Z})^*(\iota^{R}_{Z})^*(\iota_R^S)^*(s)=(\pi^{E\cap\td{U}}_{Z})^*(\iota_Z^S)^*(s)=(\iota^{E\cap\td{T}}_{E\cap\td{U}})^*(\pi^{E\cap\td{T}}_{Z})^*(\iota_{Z}^S)^*(s)$ which is independent of~$R$ and~$T$. Thus, the claim follows from the equality
		$$\sum_{S\stackrel{1}{\hookleftarrow}R\stackrel{1}{\hookrightarrow} U} d'_{R,U} d''_{R,S}(X)=\sum_{S \stackrel{1}{\hookrightarrow} T\stackrel{1}{\hookleftarrow} U}  d''_{U,T} d'_{S,T}(X)$$
		which is a consequence of the fact that~$A_{\bullet,\bullet}^{\leq Z}(\B)$ is a bi-complex.
		\end{proof}
		
	\subsection{The essential case:~$\Phi$ is a quasi-isomorphism}\label{parPhiqisessential}
	
		In this paragraph, we still assume that~$\B$ is essential and that~$Z$ is the minimal stratum. We further assume that~$Z$ is exact and prove the second point of Theorem~\ref{maintheoremtechnical}.
		
		\subsubsection{The strategy}		
		
		We start with a basic fact of homological algebra.		
		
		\begin{lem}\label{lemqisfiltration}
		Let~$f:C_\bullet\rightarrow C'_\bullet$ be a morphism of complexes, let~$(F_pC_\bullet)_p$ and~$(F_pC'_\bullet)_p$ be finite increasing filtrations on~$C_\bullet$ and~$C'_\bullet$ such that~$f(F_pC_\bullet)\subset F_pC'_\bullet$. Then~$f$ is a quasi-isomorphism if for every~$p$, the induced morphism~$\mathrm{gr}_p^Ff:\mathrm{gr}_p^FC'_\bullet\rightarrow \gr_p^FC'_\bullet$ is a quasi-isomorphism.
		\end{lem}
		
		\begin{proof}
		By induction on the length of the filtration, using the long exact sequence in cohomology and the~$5$-lemma.
		\end{proof}
		
		Let~$\Phi:\qD_{\bullet}(\B)\rightarrow \qD_{\bullet}(\td{\B})$ be the morphism of complexes induced on the total complexes. Using the filtration on the lines and the above lemma, one sees that~$\Phi$ is a quasi-isomorphism if for every~$j$, the morphism
		$$\Phi_{\bullet,j}:\qD_{\bullet,j}(\B) \rightarrow \qD_{\bullet,j}(\td{\B})$$
		induced on the~$j$-th lines is a quasi-isomorphism. In the rest of \S\ref{parPhiqisessential}, we fix an index~$j$. We are reduced to proving that the cone~$C_{\bullet,j}$ of~$\Phi_{\bullet,j}$ is exact.
		
		We have
		$$C_{i,j}=\qD_{i,j}(\B)\oplus \qD_{i+1,j}(\td{\B})$$
		and the differential~$d':C_{i,j}\rightarrow C_{i-1,j}$ is given by
		$$d'(x,\td{x})=(d'(x),\Phi(x)-d'(\td{x})).$$
		
		The strategy is as follows. We define an complex~$B_{\bullet,j}$ and morphisms~$\alpha:B_{i,j}\rightarrow C_{i,j}$; the second point of Theorem~\ref{maintheoremtechnical} then follows from the following facts:
		\begin{enumerate}[--]
		\item~$B_{\bullet,j}$ is exact (Lemma~\ref{lemBexact});
		\item~$\alpha$ is a morphism of complexes (Proposition~\ref{propPsimorphism});
		\item~$\alpha$ is a quasi-isomorphism (Proposition~\ref{propPsiqis}).
		\end{enumerate}
		
		\subsubsection{The exact complex~$B_{\bullet,j}$}
	
		Let~$r$ be the codimension of~$Z$ inside~$X$. For~$S\in\s_{i+j}(\B)$, let us set 
		$$B_{i,j}^S=H^{q-2r+2j}(Z)(r-j)\otimes A_{i,j}^S(\B).$$
		For an inclusion~$S\stackrel{1}{\hookrightarrow} T$, we define~$d'_{S,T}:B_{i,j}^S\rightarrow B_{i-1,j}^T$. For~$z\otimes X\in H^{q-2r+2j}(Z)(r-j)\otimes A_{i,j}^S(\B)$, it is given by
		$$d'_{S,T}(z\otimes X)=z\otimes d'_{S,T}(X).$$
		If we now set~$B_{i,j}=\bigoplus_{S\in\s_{i+j}(\B)} B_{i,j}^S$, we get a complex~$B_{\bullet,j}$.
		
		\begin{lem}\label{lemBexact}
		$B_{\bullet,j}$ is an exact complex.
		\end{lem}
		
		\begin{proof}
	 	$B_{\bullet,j}$ is nothing but the tensor product of~$H^{q-2r+2j}(Z)(r-j)$ with the complex~$(A_{\bullet,j}^{\leq Z}(\B),d')$, which is exact since~$Z$ is exact.
		\end{proof}
		
		\begin{rem}
		The complexes~$(B_{\bullet,j},d')$ are the lines of a bi-complex whose differentials~$d'_{S,T}$ are given by 
		$$d''_{S,T}(z\otimes X)=(z\,.\,c_1(N_{S/T})_{|Z})\otimes d''_{S,T}(X).$$
		\end{rem}
		
		\subsubsection{A quasi-isomorphism~$\alpha:B_{\bullet,j}\rightarrow C_{\bullet,j}$}
		
		A morphism~$\alpha:B_{\bullet,j}\rightarrow C_{\bullet,j}$ is determined by two morphisms~$f:B_{\bullet,j}\rightarrow \qD_{\bullet,j}(\B)$ and~$g:B_{\bullet,j}\rightarrow \qD_{\bullet+1,j}(\td{\B})$.
		
		
		We define~$f:B_{i,j}^S\rightarrow \qD_{i,j}^S(\B)$ by the formula 
		$$f(z\otimes X)=(\iota_Z^S)_*(z)\otimes X$$
		 and~$g:B_{i,j}^S\rightarrow \qD_{i+1,j}^{E\cap\td{S}}(\td{\B})$ by the formula 
		~$$g(z\otimes X)=(\pi_Z^{E\cap\td{S}})^*(z)\,.\,\gamma_S\otimes X$$
		where~$\gamma_S$ is the excess class of the blow-up~$\pi^{\td{S}}_S:\td{S}\rightarrow S$ along~$Z$, defined in \S\ref{appexcessclass}.
		
		\begin{prop}\label{propPsimorphism}
		We have~$f\circ d'=d'\circ f$ and~$g\circ d'+d'\circ g=\Phi\circ f$. Thus,~$f$ and~$g$ define a morphism of complexes~$\alpha:B_{\bullet,j}\rightarrow C_{\bullet,j}$.
		\end{prop}
		
		\begin{proof}
		The first equality is trivial. For the second equality, we compute
		\begin{align*}
		(\Phi\circ f)(z\otimes X)=&(\pi_S^{\td{S}})^*(\iota_{Z}^S)_*(z)\otimes X + \sum_{S\stackrel{1}{\hookrightarrow} T}(\pi_{Z}^{E\cap\td{T}})^*(\iota_{Z}^S)^*(\iota_{Z}^S)_*(z)\otimes d'_{S,T}(X);\\
		(g d'_{S,T})(z\otimes X)=&(\pi_{Z}^{E\cap\td{T}})^*(z)\,.\,\gamma_T\otimes d'_{S,T}(X);\\
		( d'_{E\cap\td{S},\td{S}}g)(z\otimes X)=&(\iota_{E\cap\td{S}}^{\td{S}})_*((\pi_{Z}^{E\cap\td{S}})(z)\,.\,\gamma_S)\otimes X ;\\
		( d'_{E\cap\td{S},E\cap\td{T}}g)(z\otimes X)=&-(\iota_{E\cap\td{S}}^{E\cap\td{T}})_*((\pi_{Z}^{E\cap\td{S}})^*(z)\,.\,\gamma_S)\otimes d'_{S,T}(X).\\
		\end{align*}
		The terms~$(\cdots)\otimes X$ cancel because of the equality
		$$(\pi_S^{\td{S}})^*(\iota_{Z}^S)_*(z)=(\iota_{E\cap\td{S}}^{\td{S}})_*((\pi_{Z}^{E\cap\td{S}})(z)\,.\,\gamma_S)$$
		which is a special case of (\ref{appeqgamma}). For the terms~$(\cdots)\otimes d'_{S,T}(X)$, we have to prove the equality
		$$(\iota_{E\cap\td{S}}^{E\cap\td{T}})_*((\pi_{Z}^{E\cap\td{S}})^*(z)\,.\,\gamma_S)=(\pi_{Z}^{E\cap\td{T}})^*(z)\,.\,\gamma_T-(\pi_{Z}^{E\cap\td{T}})^*(\iota_{Z}^S)^*(\iota_{Z}^S)_*(z).$$
		We have~$(\pi_{Z}^{E\cap\td{S}})^*=(\iota_{E\cap\td{S}}^{E\cap\td{T}})^*\circ (\pi^{E\cap\td{T}}_{Z})^*$, hence the projection formula (\ref{appprojectionformula}) gives 
		$$(\iota_{E\cap\td{S}}^{E\cap\td{T}})_*((\pi_{Z}^{E\cap\td{S}})^*(z)\,.\,\gamma_S)=(\pi^{E\cap\td{T}}_{Z})^*(z)\,.\,(\iota_{E\cap\td{S}}^{E\cap\td{T}})_*(\gamma_S).$$
		Let us write~$r(T)$ for the codimension of~$Z$ inside~$T$. Then~$(\pi_{Z}^{E\cap\td{T}})^*(\iota_{Z}^S)^*(\iota_{Z}^S)_*(z)=(\pi_{Z}^{E\cap\td{T}})^*(c_{r(T)-1}(N_{Z/S})\,.\,z)$. To sum up, we are reduced to proving the equality
		$$(\iota_{E\cap\td{S}}^{E\cap\td{T}})_*(\gamma_S)=\gamma_T-(\pi_{Z}^{E\cap\td{T}})^* (c_{r(T)-1}(N_{Z /S}))$$
		which is a special case of (\ref{appGysingammapar}).
		\end{proof}
		
		\begin{prop}\label{propPsiqis}
		$\alpha:B_{\bullet,j}\rightarrow C_{\bullet,j}$ is a quasi-isomorphism. Thus,~$C_{\bullet,j}$ is exact, and~$\Phi$ is a quasi-isomorphism.
		\end{prop}
		
		\begin{proof}
		We use Lemma~\ref{lemqisfiltration}, defining the filtration~$F_p\alpha:F_pB_{\bullet,j}\rightarrow F_pC_{\bullet,j}$ which corresponds to the terms involving strata~$S$,~$\td{S}$ and~$E\cap\td{S}$ with~$\mathrm{codim}(S)\leq p+j$. All we have to prove is that~$\mathrm{gr}^F_p\alpha:\mathrm{gr}^F_pB_{\bullet,j}\rightarrow \mathrm{gr}^F_pC_{\bullet,j}$ is a quasi-isomorphism for every~$p$. On the one hand,~$\mathrm{gr}_p^FB_{\bullet,j}$ is concentrated in degree~$p$ with
		$$\mathrm{gr}_p^F B_{p,j}=\bigoplus_{S\in\s_{p+j}(\B)}B_{p,j}^S$$
		and differential~$0$. On the other hand,~$\mathrm{gr}_p^FC_{\bullet,j}$ is concentrated in degrees~$\{p,p-1\}$ with
		\begin{align*}
		\mathrm{gr}_p^F C_{p,j}=&\bigoplus_{S\in\s_{p+j}(\B)}D_{p,j}^S(\B) \oplus D_{p+1,j}^{E\cap\td{S}}(\td{\B});\\
		\mathrm{gr}_p^F C_{p-1,j}=&\bigoplus_{S\in\s_{p+j}(\B)} D_{p,j}^{\td{S}}(\td{\B}).
		\end{align*}
		The differential~$D_{p,j}^S(\B)\rightarrow D_{p,j}^{\td{S}}(\td{\B})$ is~$s\otimes X\mapsto (\pi_S^{\td{S}})^*(s)\otimes X$; the differential~$D_{p+1,j}^{E\cap\td{S}}(\td{\B})\rightarrow D_{p,j}^{\td{S}}(\td{\B})$ is given by~$e\otimes X\mapsto -(\iota_{E\cap\td{S}}^{\td{S}})_*(e)\otimes X$.
			We are left with proving that for a fixed stratum~$S\in\s_{p+j}(\B)$ we have a quasi-isomorphism
			$$\xymatrix{
			0 \ar[r] & D_{p,j}^S(\B)\oplus D^{E\cap\td{S}}_{p+1,j}(\td{\B}) \ar[r] & D^{\td{S}} _{p,j}(\td{\B}) \ar[r] & 0\\
			0 \ar[r] & B_{p,j}^S \ar[r]\ar[u] & 0 \ar[r]\ar[u] & 0.}$$
			The above diagram is, up to a Tate twist, the tensor product of~$A_{i,j}^S$ with
			$$\xymatrix{
			0 \ar[r] & H^{q-2p}(S)\oplus H^{q-2p-2}(E\cap\td{S})(-1) \ar[r] & H^{q-2p}(\td{S}) \ar[r] & 0\\
			0 \ar[r] & H^{q-2r+2j}(Z)(p+j-r) \ar[r]\ar[u] & 0 \ar[r]\ar[u] & 0.}$$
			The fact that this is a quasi-isomorphism is a reformulation of the short exact sequence (\ref{appses}).
		\end{proof}
		
	\subsection{The general case:~$\Phi$ is a morphism of bi-complexes}\label{parPhimorphismgeneral}
	
		In this paragraph we prove the general case of the first point of Theorem~\ref{maintheoremtechnical}.
		
		\begin{prop}
		We have~$\Phi\circ d'=d'\circ\Phi$.
		\end{prop}
		
		\begin{proof}
		Here are the details to add in the proof of Proposition~\ref{propcompd'}. We write~$S\hookrightperp T$ for an inclusion which is not of parallel type.
		\begin{enumerate}[--]
		\item The terms~$(\cdots)\otimes d'_{S,T}(X)$ still cancel, but there are two cases to consider. For the terms corresponding to an inclusion~$S\hookrightpar T$, the argument is the same as in the essential case, replacing~$Z$ by~$Z\cap S=Z\cap T$. For the terms corresponding to an inclusion~$S\hookrightperp T$, the cancelation follows from the formula
		$$(\pi_T^{\td{T}})^*(\iota_S^T)_*(s)=(\iota_{\td{S}}^{\td{T}})_*(\pi^{\td{S}}_S)^*(s)$$
		which is a special case of (\ref{apppullbackperp}).
		\item The terms corresponding to chains~$S\hookrightpar T\hookrightpar U$ cancel thanks to the same argument as in the essential case, replacing~$Z$ by~$Z\cap S=Z\cap T=Z\cap U$.
		\item We are left with proving the equality, for~$U$ fixed:
		$$\sum_{S\hookrightperp Q\hookrightpar U}(\pi_{Z\cap Q}^{E\cap\td{U}})^*(\iota_{Z\cap Q}^Q)^*(\iota_S^Q)^*(s)\otimes d'_{Q,U}d'_{S,Q}(X)= -\sum_{S\hookrightpar T\hookrightperp U} (\iota_{E\cap\td{T}}^{E\cap\td{U}})_*(\pi_{Z\cap S}^{E\cap\td{T}})^*(\iota_{Z\cap S}^S)^*(s)\otimes d'_{T,U}d'_{S,T}(X).$$
		
		Let us start with a local decomposition~$S=S_\parallel\pitchfork S_\perp$, and~$U=U_\parallel\pitchfork U_\perp$ with~$S_\parallel\stackrel{1}{\hookrightarrow} U_\parallel$ and~$S_\perp\stackrel{1}{\hookrightarrow} U_\perp$. There is thus a unique diagram~$S\hookrightperp Q\hookrightpar U$ and a unique diagram~$S\hookrightpar T\hookrightperp U$, i.e.~$Q=U_\parallel \pitchfork S_\perp$ and~$T=S_\parallel\pitchfork U_\perp$. Using the K\"{u}nneth formula (\ref{propkunneth}) for~$A_{\bullet,\bullet}^{\leq S}(\B)$ with respect to the decomposition~$S=S_\parallel\pitchfork S_\perp$, the fact that~$d'\circ d'=0$ implies that~$ d'_{Q,U}d'_{S,Q}(X)=-d'_{T,U}d'_{S,T}(X)$. Thus, we are left with proving the equality
		$$(\pi_{Z\cap Q}^{E\cap\td{U}})^*(\iota_{Z\cap Q}^Q)^*(\iota_S^Q)^*(s)=(\iota_{E\cap\td{T}}^{E\cap\td{U}})_*(\pi_{Z\cap S}^{E\cap\td{T}})^*(\iota_{Z\cap S}^S)^*(s).$$
		Since~$Z\cap Q$ and~$S$ are transverse in~$Q$, (\ref{appiitransverse}) implies the identity 
		$$(\iota_{Z\cap Q}^Q)^*(\iota_S^Q)^*(s)=(\iota_{Z\cap S}^{Z\cap Q})_*(\iota_{Z\cap S}^S)^*(s).$$
		Thus, writing~$z=(\iota_{Z\cap S}^S)^*(s)$ and remembering that~$Z\cap S=Z\cap T$, we only need to prove that
		$$(\pi_{Z\cap Q}^{E\cap\td{U}})^*(\iota_{Z\cap T}^{Z\cap Q})_*(z)=(\iota_{E\cap\td{T}}^{E\cap\td{U}})_*(\pi_{Z\cap\td{T}}^{E\cap\td{T}})^*(z)$$
		which is a special case of (\ref{appbasechangeperp}) since~$Z\cap U$ and~$T$ are transverse in~$U$. 
		\end{enumerate}
		\end{proof}
		
		\begin{prop}
		We have~$\Phi\circ d''=d''\circ\Phi$.
		\end{prop}
		
		\begin{proof}
		Here are the details to add in the proof of Proposition~\ref{propcompd''}.
		\begin{enumerate}[--]
		\item The terms~$(\cdots)\otimes d''_{R,S}(X)$ cancel by the same argument as in the essential case. Thus it remains to show that for~$U$ fixed we have
		$$\sum_{S\stackrel{1}{\hookleftarrow}R\hookrightpar U}(\pi^{E\cap\td{U}}_{Z})^*(\iota^{R}_{Z})^*(\iota_R^S)^*(s)\otimes  d'_{R,U} d''_{R,S}(X)=\sum_{S\hookrightpar T\stackrel{1}{\hookleftarrow} U} (\iota^{E\cap\td{T}}_{E\cap\td{U}})^*(\pi^{E\cap\td{T}}_{Z})^*(\iota_{Z}^S)^*(s)\otimes d''_{U,T} d'_{S,T}(X).$$
		\item If~$S\cap\td{U}=\varnothing$ then the left-hand side is zero. For a diagram~$S\hookrightpar T\stackrel{1}{\hookleftarrow} U$ we have~$Z\cap U\subset Z\cap S$, hence~$Z\cap U=\varnothing$ and~$E\cap\td{U}=\varnothing$, thus the corresponding term in the right-hand side is zero.
		\item If~$S\cap\td{U}\neq\varnothing$, the same argument as in the essential case works. To prove the identity
		$$\sum_{S\stackrel{1}{\hookleftarrow}R\hookrightpar U} d'_{R,U} d''_{R,S}(X)=\sum_{S \hookrightpar T\stackrel{1}{\hookleftarrow} U}  d''_{U,T} d'_{S,T}(X)$$
		one has to use the K\"{u}nneth formula (Proposition~\ref{propkunneth}) in addition of the fact that~$A_{\bullet,\bullet}^{\leq S\cap U}(\B)$ is a bi-complex.
		\end{enumerate}
		\end{proof}
	
	\subsection{The general case:~$\Phi$ is a quasi-isomorphism}\label{parPhiqisgeneral}
	
		In this paragraph we prove the general case of the second point of Theorem~\ref{maintheoremtechnical} by reducing to the essential case, already proved in \S\ref{parPhiqisessential}.
		
		\begin{defi} Let~$P$ be a stratum of~$\B$ that is transverse to~$Z$; in particular,~$Z\pitchfork P\neq\varnothing$. Let~$S$ be a stratum such that~$Z\cap S\neq\varnothing$. Then by looking at a local chart around any point of~$Z\cap S$, one sees that we have a decomposition ~$S=S_Z\pitchfork P$ with~$S_Z\supset Z$ and~$P$ transverse to~$Z$. We call~$P$ the \textit{transverse direction} of~$S$.
		\end{defi}
		
		We let~$\B_P$ be the arrangement of hypersurfaces on~$P$ consisting in the intersections of~$P$ and the hypersurfaces~$K\in\B^{\leq Z}$. It is essential, with minimal stratum~$Z\pitchfork P$. The strata of~$\B_P$ are exactly the strata of~$\B$ with transverse direction~$P$. As the coloring is concerned, we ask that~$\chi(S_Z\pitchfork P)=\chi(S_Z)$ for every~$S_Z\supset Z$.
		
		 The Orlik--Solomon bi-complex of~$\B_P$ is related to the one of~$\B$ by
		 \begin{equation} \label{eqisoOStransverseP}
		 A_{\bullet,\bullet}^{\leq S_Z\pitchfork P}(\B_P)\cong A_{\bullet,\bullet}^{\leq S_Z}(\B).
		 \end{equation}
		 In particular, if~$Z$ is exact in~$\B$ then~$Z\pitchfork P$ is exact in~$\B_P$.
			
		Let~$S=S_Z\pitchfork P$ be a stratum with transverse direction~$P$. Combining the K\"{u}nneth formula (Proposition~\ref{propkunneth}) and (\ref{eqisoOStransverseP}), we get an isomorphism
		$$A_{i,j}^S(\B)\cong \bigoplus_{k+l=\mathrm{codim}(P)} A_{i-k,j-l}^{S_Z\pitchfork P}(\B_P)\otimes A_{k,l}(\B).$$
		and hence an isomorphism at the level of the Orlik--Solomon bi-complexes:
		$$\qD_{i,j}^S(\B)\cong \bigoplus_{k+l=\mathrm{codim}(P)} {}^{(q-2k)}D_{i-k,j-l}^{S_Z\pitchfork P}(\B_P) (-k)\otimes A_{k,l}^P(\B).$$

		Summing over all strata~$S\in\s_{i+j}(\B)$ and grouping together the strata having the same transverse direction~$P$, we get a decomposition:

		$$\qD_{i,j}(\B)=\left(\bigoplus_{\substack{S\in\s_{i+j}(\B)\\Z\cap S=\varnothing}} \qD_{i,j}^S(\B)\right) \oplus \left(\bigoplus_{\substack{P \perp Z \\ k+l=\mathrm{codim}(P)}} {}^{(q-2k)}D_{i-k,j-l}(\B_P) (-k)\otimes A_{k,l}^P(\B)\right)$$
		where~$P\perp Z$ means that we sum over all strata~$P$ that are transverse to~$Z$.
			
		Now it is clear that in the blown-up situation we have
			
		$$\qD_{i,j}(\td{\B})=\left(\bigoplus_{\substack{S\in\s_{i+j}(\B)\\Z\cap S=\varnothing}} \qD_{i,j}^{\td{S}}(\td{\B})\right) \oplus \left(\bigoplus_{\substack{P \perp Z \\ k+l=\mathrm{codim}(P)}} {}^{(q-2k)}D_{i-k,j-l}(\td{\B_P}) (-k)\otimes A_{k,l}^P(\B)\right)$$
			
		where~$\td{\B_P}$ is the blow-up of~$\B_P$ along~$Z\pitchfork P$.\\
		
		These decompositions are compatible with~$\Phi$ in the following sense:
		\begin{enumerate}[--]
		\item for~$S\in\s_{i+j}(\B)$ such that~$Z\cap S=\varnothing$,~$\Phi$ is an isomorphism~$\qD_{i,j}^S(\B) \cong \qD_{i,j}^{\td{S}}(\td{\B})$.
		\item for every~$P\perp Z$,~$\Phi:\qD_{i,j}(\B)\rightarrow \qD_{i,j}(\td{\B})$ restricts to 
		$$D_{i-k,j-l}(\B_P) (-k)\otimes A_{k,l}^P(\B) \rightarrow D_{i-k,j-l}(\td{\B_P}) (-k)\otimes A_{k,l}^P(\B)$$
		which is nothing but~$\Phi\otimes\mathrm{id}$.
		\end{enumerate}
		
		\begin{prop}
		If~$Z$ is exact, then~$\Phi:\qD_\bullet(\B)\rightarrow \qD_\bullet(\td{\B})$ is a quasi-isomorphism
		\end{prop}
		
		\begin{proof}
		As in \S\ref{parPhiqisessential}, it is enough to prove that for every line~$j$, the morphism~$\Phi_{\bullet,j}:\qD_{\bullet,j}(\B)\rightarrow \qD_{\bullet,j}(\td{\B})$ is a quasi-isomorphism. 
		
		The index~$j$ being fixed, we define an increasing filtration~$F_p\Phi_{\bullet,j}:F_p\qD_{\bullet,j}(\B)\rightarrow F_p\qD_{\bullet,j}(\td{\B})$. By definition,~$F_p\qD_{\bullet,j}(\B)$ is the sum of the terms corresponding to~$\mathrm{codim}(P)\leq p$. We add the convention~$F_p\qD_{\bullet,j}(\B)=\qD_{\bullet,j}(\B)$ for~$p=\mathrm{dim}(X)+1$ to include the terms corresponding to~$Z\cap S=\varnothing$. We make the analogous definition for~$\qD_{\bullet,j}(\td{\B})$. 
		
		In view of Lemma~\ref{lemqisfiltration}, it is enough to show that for every~$p$, the morphism~$\mathrm{gr}^F_p\Phi_{\bullet,j}:\mathrm{gr}^F_p\qD_{\bullet,j}(\B)\rightarrow \mathrm{gr}^F_p\qD_{\bullet,j}(\td{\B})$ is a quasi-isomorphism. For~$p=\mathrm{dim}(X)+1$,~$\mathrm{gr}^F_p\Phi_{\bullet,j}$ is an isomorphism. For~$p\leq \mathrm{dim}(X)$ we get
		$$\mathrm{gr}^F_p\qD_{\bullet,j}(\B)=\bigoplus_{\substack{P \perp Z \\ \mathrm{codim}(P)=p \\ k+l=p}}{}^{(q-2k)}D_{\bullet-k,j-l}(\B_P) (-k)\otimes A_{k,l}^P(\B)$$
		and the differential on~$D_{\bullet-k,j-l}(\B_P) (-k)\otimes A_{k,l}^P(\B)$ is~$d'\otimes\mathrm{id}$. The same is true for
		$$\mathrm{gr}^F_p\qD_{\bullet,j}(\td{\B})=\bigoplus_{\substack{P \perp Z \\ \mathrm{codim}(P)=p \\ k+l=p}}{}^{(q-2k)}D_{\bullet-k,j-l}(\td{\B_P}) (-k)\otimes A_{k,l}^P(\B).$$
		Thus,~$\mathrm{gr}^F_p\Phi_{\bullet,j}$ is a quasi-isomorphism if and only if every~$\Phi:{}^{(q-2k)}D_{\bullet-k,j-l}(\B_P) \rightarrow {}^{(q-2k)}D_{\bullet-k,j-l}(\td{\B_P})$ is a quasi-isomorphism. Since the arrangements~$\B_P$ are essential with~$Z\pitchfork P$ exact, this follow from the essential case, already proved in \S\ref{parPhiqisessential}.
		\end{proof}
		
	\subsection{Working withouth the connectedness assumption}\label{parwithoutconnectedness}
	
	Let~$\B$ be a bi-arrangement of hypersurfaces in a complex manifold~$X$, and~$Z$ a good stratum of~$X$. If we do not assume (\ref{connectednessassumption}) that the intersection of strata are all connected, then it is still possible to define the morphisms~$\Phi$ as in~\ref{pardefPhi}.
	
	Let us fix a stratum~$S$ of~$\B$. For every~$S\stackrel{1}{\hookrightarrow}T$, we have a decomposition into connected components
	$$Z\cap T=\bigsqcup_{\alpha\in I_\parallel(T)}(Z\cap T)_\alpha \sqcup \bigsqcup_{\beta\in I_\perp(T)}(Z\cap T)_\beta$$
	where for each~$\alpha\in I_\parallel(T)$,~$(Z\cap T)_\alpha\subset S$, and for each~$\beta\in I_\perp(T)$,~$(Z\cap T)_\beta \not\subset S$.
	In the same fashion, we have a decomposition into connected components 
	$$E\cap\td{T}=\bigsqcup_{\alpha\in I_\parallel(T)}(E\cap \td{T})_\alpha \sqcup \bigsqcup_{\beta\in I_\perp(T)}(E\cap \td{T})_\beta$$
	and for each~$\alpha$ we have a morphism~$\pi_{T,\alpha}:(E\cap\td{T})_\alpha\rightarrow (Z\cap T)_\alpha$.
	
	We then define 
	$$\Phi(s\otimes X)=(\pi^{\td{S}}_S)^*(s)\otimes X +\sum_{\substack{S\stackrel{1}{\hookrightarrow}T \\\alpha\in I_\parallel(T)}}(\pi_{T,\alpha})^*(\iota^S_{(Z\cap T)_\alpha})^*(s)\otimes d'_{S,T}(X).$$
	
	We leave it to the reader to check that the proof of Theorem~\ref{maintheoremtechnical} can be adapted in that setting.

\appendix
	
	\section{Normal crossing divisors and relative cohomology}\label{parappA}
	
			In this appendix, we fix~$X$ a complex manifold,~$\L$ and~$\M$ two simple normal crossing divisors in~$X$ that do not share an irreducible component and such that~$\L\cup\M$ is a normal crossing divisor. We will denote by~$L_1,\ldots,L_l$ (resp.~$M_1,\ldots,M_m$) the irreducible components of~$\L$ (resp.~$\M$). For~$I\subset\{1,\ldots,l\}$ (resp.~$J\subset\{1,\ldots,m\}$), we will write~$L_I=\bigcap_{i\in I}L_i$ (resp.~$M_J=\bigcap_{j\in J}M_j$), with the convention~$L_\varnothing=M_\varnothing=X$. For every~$I$ and~$J$,~$L_I\cap M_J$ is a disjoint union of submanifolds of~$X$.
			
			\subsection{The spectral sequence}
		
			We let
			\begin{equation}\label{appmotiveNCD}
			H^\bullet(\L,\M)=H^\bullet(X\setminus\L,\M\setminus \M\cap\L)
			\end{equation}
			be the corresponding relative cohomology group. It is endowed with a canonical mixed Hodge structure if~$X$ is a complex variety and~$\L$,~$\M$ are complex subvarieties of~$X$.
		
			\begin{prop}\label{appspectralsequenceNCD}
			\begin{enumerate}
			\item There is a spectral sequence\footnote{Here,~$(-i)$ denotes the Tate twist of weight~$2i$. It is important in the algebraic case; otherwise it should be ignored.}
			\begin{equation}\label{appeqspectralsequence}
			E_1^{-p,q}(\L,\M)=\bigoplus_{\substack{i-j=p\\ |I|=i \\ |J|=j}} H^{q-2i}(L_I\cap M_J)(-i) \;\Longrightarrow \; H^{-p+q}(\L,\M).
			\end{equation}
			The differential~$d_1:E_1^{-p,q}\rightarrow E_1^{- p+1,q}$ is that of the total complex of a double complex, where
				\begin{enumerate}[--]
				\item the horizontal differential is the collection of the morphisms 
				$$H^{q-2i}(L_I\cap M_J)(-i)\rightarrow H^{q-2i+2}(L_{I\setminus \{r\}}\cap M_J)(-i+1)$$ 
				for every~$r\in I$, which are the Gysin morphisms of the inclusions~$L_I\cap M_J\hookrightarrow L_{I\setminus\{r\}}\cap M_J$, multiplied by the signs~$\sgn(\{r\},I\setminus\{r\})$;
				\item the vertical differential is the collection of the morphisms
				$$H^{q-2i}(L_I\cap M_J)(-i)\rightarrow H^{q-2i}(L_I\cap M_{J\cup\{s\}})(-i)$$ 
				for every~$s\notin J$, which are the pull-back morphisms of the inclusions~$L_I\cap M_{J\cup\{s\}}\hookrightarrow L_I\cap M_J$, multiplied by the signs~$\sgn(\{s\},J\setminus\{s\})$.
				\end{enumerate}
			\item If~$X$ is a smooth complex variety and~$\L$,~$\M$ are complex subvarieties of~$X$, then this is a spectral sequence in the category of mixed Hodge structures.
			\item If furthermore~$X$ is projective, then this spectral sequence degenerates at the~$E_2$ term and we have
			$$E_\infty^{-p,q}\cong E_2^{-p+q} \cong \gr_q^W H^{-p+q}(\L,\M).$$
			\end{enumerate}
			\end{prop}
			
			\begin{proof}
			\begin{enumerate}
			\item We will use the following notations :$j_U^Y:U\hookrightarrow Y$ for an open immersion,~$\mathbb{Q}_Y$ for the constant sheaf with stalk~$\Q$ on a space~$Y$, and~$d_Y$ for the complex dimension of~$Y$, when this makes sense. Let us write 
			$$\mathscr{F}=\mathscr{F}(\L,\M)=(j_{X\setminus\L}^X)_*(j_{X\setminus \L\cup\M}^{X\setminus \L})_!\Q_{X\setminus \L\cup\M}[d_X]\ ,$$
			viewed as an object of the bounded derived category~$\mathrm{D}^b_c(X)$ of constructible sheaves on~$X$. Then we have
			$$H^\bullet(\L,\M)=H^{\bullet-d_X}(X,\mathscr{F}(\L,\M)).$$
			We recall that for a complex manifold~$Y$, the objetct~$\mathbb{Q}_Y[d_Y]$ is in the abelian category~$\mathrm{Perv}(Y)\subset \mathrm{D}^b_c(Y)$ of perverse sheaves on~$Y$. We note that both~$j_{X\setminus\L}^X$ and~$j_{X\setminus \L\cup\M}^{X\setminus \L}$ are affine open immersions, so that the functors~$(j_{X\setminus\L}^X)_*$ and~$(j_{X\setminus \L\cup\M}^{X\setminus \L})_!$ preserve the categories of perverse sheaves. Let us write~$\mathscr{F}'=(j_{X\setminus \L\cup\M}^{X\setminus \L})_!\Q_{X\setminus \L\cup\M}[d_X]$, viewed as a perverse sheaf on~$X\setminus \L$. We have an exact sequence of sheaves
			$$0\rightarrow \mathscr{F}' \rightarrow \Q_{X\setminus \L} \rightarrow \bigoplus_{|J|=1} (\iota_{M_J\setminus M_J\cap \L}^{X\setminus \L})_*\Q_{M_J\setminus M_J\cap\L} \rightarrow \bigoplus_{|J|=2} (\iota_{M_J\setminus M_J\cap \L}^{X\setminus \L})_*\Q_{M_J\setminus M_J\cap\L} \rightarrow\cdots$$
			where the arrows are the alternating sums of the natural restriction morphisms. For every subset~$J\subset\{1,\ldots,m\}$ of cardinality~$j$, the object~$\mathbb{Q}_{M_J\setminus M_J\cap \L}[d_X-j]$ is a perverse sheaf. By cutting the above spectral sequence into short exact sequences and rotating the corresponding triangles in the derived category, we see that there is a finite decreasing filtration
			$$\mathscr{F}'=F^0\mathscr{F}'\supset F^1\mathscr{F}' \supset  F^2\mathscr{F}'\supset\cdots$$
			on~$\mathcal{F}'$ in the abelian category~$\mathrm{Perv}(X\setminus\L)$, such that the successive quotients are given by 
			$$\mathrm{gr}^j_F\mathscr{F}' \cong \bigoplus_{|J|=j}(\iota_{M_J\setminus M_J\cap\L}^{X\setminus\L})_*\Q_{M_J\setminus M_J\cap \L}[d_X-j]\ .$$ 
			Furthermore, the extension datum~$\mathrm{gr}^j_F\mathscr{F}'\rightarrow \mathrm{gr}^{j+1}_F\mathscr{F}'[1]$ is given by the alternating sum of the natural restriction morphisms
			$$(\iota_{M_J\setminus M_J\cap\L}^{X\setminus\L})_*\Q_{M_J\setminus M_J\cap \L}[d_X-j] \rightarrow (\iota_{M_{J\cup\{r\}}\setminus M_{J\cup\{r\}}\cap\L}^{X\setminus\L})_*\Q_{M_{J\cup\{r\}}\setminus M_{J\cup\{r\}}\cap \L}[d_X-j]\ .$$
			By applying the functor~$(j_{X\setminus\L}^X)_*$, this induces a finite decreasing filtration
			$$\mathscr{F}=F^0\mathscr{F}\supset F^1\mathscr{F} \supset  F^2\mathscr{F}\supset\cdots$$
			on~$\mathcal{F}$ in the abelian category~$\mathrm{Perv}(X)$, whose successive quotients are given by 
			$$\mathrm{gr}^j_F\mathscr{F} \cong \bigoplus_{|J|=j}(\iota_{M_J}^X)_*(j_{M_J\setminus M_J\cap\L}^{M_J})_*\Q_{M_J\setminus M_J\cap \L}[d_X-j]\ .$$ 
			Let us write~$\mathscr{F}(J):=(j_{M_J\setminus M_J\cap\L}^{M_J})_*\Q_{M_J\setminus M_J\cap \L}[d_X-j]$, viewed as a perverse sheaf on~$M_J$. By applying the same argument (dualized) as in the case of~$\mathscr{F}'$, we see that there is a finite increasing filtration
			$$F_0\mathscr{F}(J)\subset F_1\mathscr{F}(J)\subset F_2\mathscr{F}(J) \subset \cdots \subset \mathscr{F}(J)$$
			on~$\mathscr{F}(J)$, whose successive quotients are given by
			$$\mathrm{gr}_i^F\mathscr{F}(J) \cong \bigoplus_{|I|=i}(\iota_{L_I\cap M_J}^{M_J})_*\Q[d_X-i-j] \ .$$
			Furthermore, the extensions datum~$\mathrm{gr}_{i+1}^F\mathscr{F}(J)[-1]\rightarrow \mathrm{gr}_{i}^F\mathscr{F}(J)$ is given by the alternating sum of the natural Gysin morphisms
			$$(\iota_{L_I\cap M_J}^{M_J})_*\Q[d_X-i-j-2] \rightarrow (\iota_{L_{I\cup\{r\}}\cap M_J}^{M_J})_*\Q[d_X-i-j]$$
			This induces a filtration on the quotients~$\mathrm{gr}^j_F\mathscr{F}$; by pulling it back to a filtration on~$\mathscr{F}$ and forming the total filtration, one gets a finite increasing filtration
			$$\cdots \subset T_{-1}\mathscr{F} \subset T_0\mathscr{F}\subset T_1\mathscr{F} \subset \cdots \subset \mathscr{F}$$
			whose successive quotients are given by 
			$$\mathrm{gr}^T_p\mathscr{F}\cong \bigoplus_{\substack{i-j=p \\ |I|=i \\ |J|=j}} (\iota_{L_I\cap M_J}^{X})_*\Q_{L_I\cap M_J}[d_X-i-j]$$
			The hypercohomology spectral sequence is thus given by
			$$E_1^{-p,q}=\mathbb{H}^{-p+q}(X,\mathrm{gr}^T_p\mathscr{F})\cong\bigoplus_{\substack{i-j=p \\ |I|=i \\ |J|=j}} H^{d_X+q-2i}(L_I\cap M_J) \;\Longrightarrow\; H^{-p+q}(X,\mathscr{F})$$
			and the desired spectral sequence is obtained by shifting the degree~$q$ by~$d_X$.
			\item If we work in the category of mixed Hodge modules~\cite[\S 14]{peterssteenbrink}, then the above proof works and gives the compatibility of the spectral sequence with the mixed Hodge structures.
			\item If~$X$ is smooth and projective, then all~$L_I\cap M_J$ are (disjoint unions of) smooth projective varieties. Thus,~$E_1^{-p,q}$ is a pure Hodge structure of weight~$q$. The degeneration then comes from the fact that in the category of mixed Hodge structures, a morphism between two pure Hodge structures of different weights is zero.
			\end{enumerate}
			\end{proof}
		
			\begin{rem}
			In the case~$\M=\varnothing$, one recovers the spectral sequence
			$$E_1^{-p,q}=\bigoplus_{|I|=p} H^{q-2p}(L_I)(-p) \;\Longrightarrow \; H^{-p+q}(X\setminus \L)$$
			where the differential is the alternating sum of the Gysin morphisms of the inclusions~$L_{I}\hookrightarrow L_{I\setminus \{r\}}$. This spectral sequence was first studied by Deligne in the smooth and projective case~\cite[Corollary 3.2.13]{delignehodge2}. If~$\L$ is a smooth submanifold of~$X$ (i.e.~$l=1$), then this spectral sequence is nothing but the residue/Gysin long exact sequence:
			$$\cdots\rightarrow H^{k-2}(\L) (-1)\rightarrow H^k(X) \rightarrow H^k(X\setminus \L) \rightarrow \cdots$$ 
			In the case~$\L=\varnothing$, one recovers the spectral sequence
			$$E_1^{p,q}=\bigoplus_{|J|=p} H^q(M_J) \;\Longrightarrow\; H^{p+q}(X,\M)$$
			where the differential is the alternating sum of the pull-back morphisms of the inclusions~$M_{J\cup\{s\}}\hookrightarrow M_J$. If~$\M$ is a smooth submanifold of~$X$ (i.e.~$m=1$), then this spectral sequence is nothing but the long exact sequence in relative cohomology:
			$$\cdots \rightarrow H^k(X,\M) \rightarrow H^k(X) \rightarrow H^k(\M)\rightarrow \cdots$$
			\end{rem}		
		
			\begin{rem}\label{appremhodge}
			There is a way of proving the first and third points of Proposition~\ref{appspectralsequenceNCD} which does not make use of mixed Hodge modules, but only of mixed Hodge theory \textit{\`a la} Deligne~\cite{delignehodge2,delignehodge3}, i.e. with complexes of holomorphic differential forms. After tensoring by~$\C$,~$\mathscr{F}(\L,\M)$ is isomorphic to the total complex of the double complex 
			\begin{equation}\label{eqcomplexOmega}
			0\rightarrow \Omega^\bullet_X(\log\L) \rightarrow \bigoplus_{|J|=1}(\iota_{M_J}^X)_*\Omega^\bullet_{M_J}(\log\L\cap M_J)\rightarrow \bigoplus_{|J|=2}(\iota_{M_J}^X)_*\Omega^\bullet_{M_J}(\log\L\cap M_J)\rightarrow \cdots
			\end{equation}
			On each component~$\Omega^\bullet_{M_J}(\log\L\cap M_J)$ there is the filtration~$P$ by the order of the pole~\cite{delignehodge2} such that we have the Poincar\'e residue isomorphisms
			$$\gr_k^P\Omega^\bullet_{M_J}(\log\L\cap M_J) \cong \bigoplus_{|I|=k}(\iota_{L_I\cap M_J}^{M_J})_*\Omega^{\bullet-k}_{L_I\cap M_J}.$$
			Suitably shifted, this gives a filtration~$W$ on (\ref{eqcomplexOmega}) whose hypercohomology spectral sequence is the spectral sequence of Proposition~\ref{appspectralsequenceNCD} tensored with~$\C$. If~$X$ is projective, the formalism of mixed Hodge complexes~\cite{delignehodge3} allows one to prove that it is defined over~$\Q$ and compatible with the mixed Hodge structures.
			\end{rem}
		
		\subsection{Duality}		
		
			There is also the compactly-supported version of (\ref{appmotiveNCD})
			\begin{equation}\label{appmotivecompactNCD}
			H_c^\bullet(\L,\M)=H_c^\bullet(X\setminus\L,\M\setminus \L\cap \M).
			\end{equation}
			This has to be understood as the compactly supported cohomology groups of the sheaf~$\mathscr{F}(\L,\M)$ defined in the proof of Proposition~\ref{appspectralsequenceNCD}. If~$X$ is compact, then it is the same as (\ref{appmotiveNCD}).
		
			\begin{prop}\label{appproppoincareverdier}
			Let~$n=\mathrm{dim}_\C(X)$. Then~$H^\bullet(\L,\M)$ and~$H^\bullet(\M,\L)$ are dual to each other in the sense that we have a Poincar\'{e}--Verdier duality
			$$\left(H^k(\L,\M)\right)^\vee \cong H^{2n-k}_c(\M,\L)$$
			that is compatible with the mixed Hodge structures in the algebraic case. The corresponding spectral sequences of Proposition~\ref{appmotiveNCD} are also dual to each other.
			\end{prop}
			
			\begin{proof}
			Let~$\mathrm{D}$ denote the Verdier duality operator on~$\mathrm{Perv}(X)$. We have, using the notations of the proof of Proposition~\ref{appspectralsequenceNCD}:
			$$\mathrm{D}\mathscr{F}(\L,\M) \cong (j_{X\setminus \L}^X)_!(j_{X\setminus \L\cup\M}^{X\setminus \L})_* \Q_{X\setminus\L\cup\M}[d_X]\ .$$
			The natural morphism 
			$$(j_{X\setminus \L}^X)_!(j_{X\setminus \L\cup\M}^{X\setminus \L})_* \Q_{X\setminus\L\cup\M}[d_X] \rightarrow (j_{X\setminus \M}^X)_*(j_{X\setminus \L\cup\M}^{X\setminus \M})_! \Q_{X\setminus\L\cup\M}[d_X] = \mathscr{F}(\M,\L)$$
			is easily seen to be an isomorphism by working in local coordinates. This implies the duality statement. The duality between the spectral sequences follows from the compatibility between Verdier duality and the filtrations in the proof of Proposition \ref{appspectralsequenceNCD}.
			\end{proof}
			
		\subsection{Blow-ups}
		
			We study the functoriality of the spectral sequence (\ref{appeqspectralsequence}) with respect to the blow-up of a stratum. For simplicity we assume that all~$L_I\cap M_J$'s are connected. Let~$Z=L_{I_0}\cap M_{J_0}$ be a stratum, with~$I_0\neq\varnothing$ so that~$Z\subset \L$. Let~$\pi:\td{X}\rightarrow X$ be the blow-up along~$Z$,~$E=\pi^{-1}(Z)$ be the exceptional divisor. We set~$\td{\L}=E\cup\td{L}_1\cup\cdots\cup\td{L_l}$ and~$\td{\M}=\td{M}_1\cup\cdots\cup \td{M}_m$. We then have a natural isomorphism
			\begin{equation}\label{appeqblowup}
			\pi^*:H^\bullet(X\setminus \L,\M\setminus \M\cap\L)\stackrel{\cong}{\rightarrow} H^\bullet(\td{X}\setminus \td{\L},\td{\M}\setminus\td{\M}\cap\td{\L}).
			\end{equation}
			
			
			\begin{prop}\label{apppropblowup}
			The spectral sequence (\ref{appeqspectralsequence}) is functorial with respect to the blow-up morphism (\ref{appeqblowup}) via a morphism of spectral sequences
			\begin{equation}\label{appPhiNCD}
			E_1^{-p,q}(\pi):E_1^{-p,q}(\L,\M)\rightarrow E_1^{-p,q}(\td{\L},\td{\M})
			\end{equation}
			given for~$s\in H^{q-2p}(L_I\cap M_J)(-p)$ by 
			$$s\mapsto \left(\pi_{L_I\cap M_J}^{\td{L}_I\cap\td{M}_J}\right)^*(s) + \sum_{i\in I\cap I_0} \sgn(\{i\},I\setminus\{i\})\, \left(\pi_{Z\cap L_{I\setminus\{i\}}\cap M_J}^{E\cap\td{L}_{I\setminus\{i\}}\cap\td{M}_J}\right)^*\left(\iota_{Z\cap L_{I}\cap M_J}^{L_I\cap M_J}\right)^*(s).$$
			\end{prop}

			\begin{proof}
			We sketch the proof for the case~$\M=\varnothing$ (see~\cite[Theorem 5.5]{duponthypersurface} for more details); the general case is similar if one uses the complexes of Remark~\ref{appremhodge}. In this special case, the spectral sequence is Deligne's spectral sequence
			\begin{equation}\label{appeqssDeligne}
			E_1^{-p,q}=\bigoplus_{|I|=p} H^{q-2p}(L_I)(-p) \;\Longrightarrow \; H^{-p+q}(X\setminus \L).
			\end{equation}
			One works over~$\C$ with the complex of logarithmic forms~$\Omega_X^\bullet(\log \L)$. By definition, we have a pull-back morphism
			$$\pi^*:\Omega_X^\bullet(\log \L)\rightarrow \Omega_{\td{X}}^\bullet(\log \td{\L}).$$
			The claim follows from the following local statement. In~$\C^n$ with coordinates~$(z_1,\ldots,z_n)$, let us write ~$\omega_i=\frac{dz_i}{z_i}$ and for~$I=\{i_1<\cdots<i_k\}$,~$\omega_I=\omega_{i_1}\wedge\cdots\wedge \omega_{i_k}$. In any standard affine chart~$\pi:\C^n\rightarrow \C^n$ of the blow-up of the linear space~$Z=\{z_1=\cdots=z_r=0\}$, one write~$z_E$ for the coordinate corresponding to the exceptional divisor, and~$\omega_E=\frac{dz_E}{z_E}$. One then has the formula
			$$\pi^*(\omega_I)=\omega_I+\sum_{\substack{i\in I\\ 1\leq i\leq r}} \sgn(\{i\},I\setminus\{i\})\, \omega_E\wedge\omega_{I\setminus\{i\}}.$$
			\end{proof}
			
	
	\section{Chern classes, blow-ups, and some cohomological identities}\label{parappB}
	
		\subsection{Chern classes of normal bundles}
		
			Let~$\iota_Z^X:Z\hookrightarrow X$ be the inclusion of a closed submanifold~$Z$ of codimension~$r$ of a complex manifold~$X$. We denote by~$N_{Z/X}$ the normal bundle of~$\iota_Z^X$ and by~$c_k(N_{Z/X})\in H^{2k}(Z)$,~$k=0,\ldots,r$ its Chern classes.
			
			For inclusions~$A\hookrightarrow B\hookrightarrow C$ we have a short exact sequence
			\begin{equation}\label{appsesnormalbundles}
			0\rightarrow N_{A/B} \rightarrow N_{A/C}\rightarrow \left(N_{B/C}\right)_{|A}\rightarrow 0
			\end{equation}
			which implies the following transitivity property of Chern classes:
			\begin{equation}\label{apptransChern}
			c_k(N_{A/C})=\sum_{j=0}^k c_j(N_{A/B})\,.\,c_{k-j}(N_{B/C})_{|A}.
			\end{equation}
			
			If~$A$ and~$B$ are two closed submanifolds of a complex manifold~$X$ that are transverse, and~$R$ a connected component of the intersection~$A\cap B$, we also have an isomorphism
			\begin{equation}\label{apptransnormalbundles}
			N_{R/X}\cong \left(N_{A/X}\right)_{|R}\oplus\left(N_{B/X}\right)_{|R}.
			\end{equation}
			By combining it with (\ref{appsesnormalbundles}) for~$R\hookrightarrow A\hookrightarrow X$, one gets an isomorphism
			\begin{equation}\label{appisonormalbundles}
			N_{R/A}\cong \left(N_{B/X}\right)_{|R}.
			\end{equation}
			
		\subsection{Gysin morphisms and pull-backs}
		
			Let~$\iota_Z^X:Z\hookrightarrow X$ be the inclusion of a closed submanifold~$Z$ of codimension~$r$ of a complex manifold~$X$. We have a pull-back morphism~$(\iota_Z^X)^*:H^\bullet(X)\rightarrow H^\bullet(Z)$ and a Gysin morphism~$(\iota_Z^X)_*:H^\bullet(Z)\rightarrow H^{\bullet+2r}(X)$. We have the projection formula
			\begin{equation}\label{appprojectionformula}
			(\iota_Z^X)_*(z.(\iota_Z^X)^*(x))=(\iota_Z^X)_*(z)\,.\,x.
			\end{equation}
			We have the following compatibilities:
			\begin{equation}\label{appiiChern}
			(\iota_Z^X)^*(\iota_Z^X)_*(z)=z\,.\,c_r(N_{Z/X});
			\end{equation}
			\begin{equation}\label{appiiclass}
			(\iota_Z^X)_*(\iota_Z^X)^*(x)=x\,.\,[Z]_X.
			\end{equation}
			Here~$N_{Z/X}$ is the normal bundle of~$Z$ inside~$X$, and~$c_k(N_{Z/X})\in H^{2k}(Z)$,~$k=0,\ldots,r$ are its Chern classes;~$[Z]_X\in H^{2r}(X)$ is the cohomology class of~$Z$ in~$X$.\\
			If~$A$ and~$B$ are two closed submanifolds of a complex manifold~$X$ that are transverse, then we have 
			\begin{equation}\label{appiitransverse}
			(\iota_A^X)^*\circ(\iota_B^X)_*=(\iota_{A\cap B}^A)_*\circ(\iota_{A\cap B}^B)^*.
			\end{equation}
			This includes the case~$A\cap B=\varnothing$ for which the right-hand side is~$0$, and the case where~$A\cap B$ is not connected for which the right-hand side is the sum of~$(\iota_R^A)_*\circ(\iota_R^B)^*$ for~$R$ a connected component of~$A\cap B$.
			
		\subsection{Blow-ups}
		
			Let~$X$ be a complex manifold and~$Z$ a closed submanifold of~$X$, of codimension~$r$. We let~$\pi:\td{X}\rightarrow X$ be the blow-up of~$X$ along~$Z$. We let~$\pi^E_Z:E\rightarrow Z$ be the morphism induced by~$\pi$, it is the projectified normal bundle of~$Z$ inside~$X$. For~$S$ a submanifold of~$X$, we denote by~$\td{S}$ its strict transform along~$\pi$, and~$\pi_S^{\td{S}}:\td{S}\rightarrow S$ the morphism induced by~$\pi$. It is the blow-up of~$S$ along~$Z\cap S$.\\
			
			Let~$L$ be a smooth hypersurface of~$X$ that contains~$Z$. We have the identity
			\begin{equation}\label{apppullbackpar}
			\pi^*\circ(\iota_L^X)_*=(\iota_{\td{L}}^{\td{X}})_*\circ(\pi_L^{\td{L}})^*+(\iota_E^{\td{X}})_*\circ(\pi_Z^E)^*\circ(\iota_Z^L)^*
			\end{equation}
			between morphisms~$H^\bullet(L)\rightarrow H^{\bullet+2}(\td{X})$. When applied to the element~$1\in H^0(L)$, one recovers
			\begin{equation}\label{apppullbackparclass}
			\pi^*([L])=[\td{L}]+[E].
			\end{equation}		
			We also have the following identity, for any~$z\in H^\bullet(Z)$:
			\begin{equation}\label{appidentitypar}
			(\iota_{E\cap\td{L}}^E)_*(\pi_Z^{E\cap\td{L}})^*(z)=(\pi_Z^E)^*(z)\,.\,\left((\pi_Z^E)^*c_1(N_{L/X})_{|Z}-c_1(N_{E/\td{X}})\right).
			\end{equation}
			
			\begin{proof}[Proof (of (\ref{appidentitypar}))]
			We have~$(\pi_Z^{E\cap\td{L}})^*=(\iota_{E\cap\td{L}}^E)^*\circ(\pi_Z^E)^*$, hence using (\ref{appiiclass}) we get
			$$(\iota_{E\cap\td{L}}^E)_*(\pi_Z^{E\cap\td{L}})^*(z)=(\pi_Z^E)^*(z)\,.\,[E\cap\td{L}]_E$$
			where~$[E\cap\td{L}]_E$ denotes the class of~$E\cap\td{L}$ in the cohomology of~$E$. Since~$\td{L}$ and~$E$ are transverse in~$\td{X}$, we may use (\ref{appiitransverse}) to get
			$$[E\cap\td{L}]_E=(\iota_{E\cap\td{L}}^E)_*(\iota_{E\cap\td{L}}^{\td{L}})^*(1)=(\iota_{E}^{\td{X}})^*(\iota_{\td{L}}^{\td{X}})_*(1)=(\iota_{E}^{\td{X}})^*([\td{L}]_{\td{X}}).$$
			Now using (\ref{apppullbackparclass}) we get 
			$$[\td{L}]_{\td{X}}=\pi^*([L]_X)-[E]$$
			and thus
			$$[E\cap\td{L}]_E= (\iota_{E}^{\td{X}})^*\pi^*[L]-(\iota_E^{\td{X}})^*[E]=(\pi_Z^E)^*(\iota_Z^X)^*[L]-c_1(N_{E/\td{X}}).$$
			The claim then follows from the computation~$(\iota_Z^X)^*[L]=(\iota_Z^L)^*(\iota_L^X)^*(\iota_L^X)_*(1)=c_1(N_{L/X})_{|Z}$ where we have used (\ref{appiiChern}).
			\end{proof}			
			
			Now if~$L$ is a smooth hypersurface of~$X$ such that~$Z$ and~$L$ are transverse in~$X$, we have the simpler identities
			\begin{equation}\label{apppullbackperp}
			\pi^*\circ(\iota_L^X)_*=(\iota_{\td{L}}^{\td{X}})_*\circ(\pi_L^{\td{L}})^*;
			\end{equation}
			\begin{equation}\label{apppullbackperpclass}
			\pi^*([L])=[\td{L}].
			\end{equation}		
			We also have 
			\begin{equation}\label{appbasechangeperp}
			(\pi_Z^E)^*\circ(\iota_{Z\cap L}^Z)_*=(\iota_{E\cap\td{L}}^E)_*\circ(\pi_{Z\cap L}^{E\cap\td{L}})^*.
			\end{equation}
			
		\subsection{The excess class~$\gamma$}\label{appexcessclass}
		
			Let~$X$ be a complex manifold and~$Z$ a closed submanifold of~$X$, of codimension~$r$. We let~$\pi:\td{X}\rightarrow X$ be the blow-up of~$X$ along~$Z$. The excess class of~$\pi$ is by definition
			\begin{equation}\label{appdefgamma}
			\gamma=c_{r-1}\left((\pi^E_Z)^*(N_{Z/X})/N_{E/\td{X}}\right) \in H^{2(r-1)}(E).
			\end{equation}
			It appears in the formula
			\begin{equation}\label{appeqgamma}
			\pi^*(\iota_Z^X)_*(z)=(\iota_E^{\td{X}})_*((\pi_Z^E)^*(z)\,.\,\gamma).
			\end{equation}			
			
			We have a short exact sequence
			\begin{equation}\label{appses}
			0\rightarrow H^{k-2r}(Z)(-r) \stackrel{\alpha}{\longrightarrow} H^{k-2}(E)(-1)\oplus H^k(X)\stackrel{\beta}{\longrightarrow} H^k(\td{X}) \rightarrow 0
			\end{equation}
			where~$\alpha$ and~$\beta$ are defined by~$\alpha(z)=\left((\pi_Z^E)^*(z)\,.\,\gamma,(\iota_Z^X)_*(z)\right)$ and~$\beta(e,x)=-(\iota_E^{\td{X}})_*(e)+\pi^*(x)$.\\
			
			If~$S$ is a submanifold of~$X$ such that (for simplicity)~$Z\cap S\neq\varnothing$ is connected, we let~$\gamma_S$ be the excess class of~$\pi_S^{\td{S}}:\td{S}\rightarrow S$. It lives in~$H^{2(r(S)-1)}(E\cap\td{S})$ where~$r(S)$ is the codimension of~$Z\cap S$ inside~$S$.\\
			Let~$L$ be a smooth hypersurface of~$X$ that contains~$Z$. We have the identity
			\begin{equation}\label{appGysingammapar}
			(\iota_{E\cap\td{L}}^E)_*(\gamma_L)=\gamma-(\pi_Z^E)^*c_{r-1}(N_{Z/L});
			\end{equation}

			
			\begin{proof}[Proof (of (\ref{appGysingammapar}))]
			Let us write~$\xi=-c_1(N_{E/\td{X}})$ and~$\xi_L=-c_1(N_{E\cap\td{L}/E})$. We then have 
			$$\gamma=\sum_{k=0}^{r-1}(\pi_Z^E)^*(c_{r-1-k}(N_{Z/X}))\,.\,\xi^k \textnormal{ and } \gamma_L=\sum_{k=0}^{r-2}(\pi_Z^{E\cap\td{L}})^*(c_{r-2-k}(N_{Z/L}))\,.\,\xi_L^k.$$
			Using (\ref{appiiclass}) and (\ref{appiitransverse}) ($E$ and~$\td{L}$ are transverse in~$X$) we get
			$$(\iota_{E\cap\td{L}}^E)^*\xi=-(\iota_{E\cap\td{L}}^E)^*(\iota_E^{\td{X}})^*(\iota_E^{\td{X}})_*(1)=-(\iota_{E\cap\td{L}}^{\td{L}})^*(\iota_{\td{L}}^{\td{X}})^*(\iota_E^{\td{X}})_*(1)=-(\iota_{E\cap\td{L}}^E)^*(\iota_{E\cap\td{L}}^E)_*(\iota_{E\cap\td{L}}^{\td{L}})^*(1)=\xi_L.$$
			Repeated applications of the projection formula (\ref{appprojectionformula}) then give
			$$(\iota_{E\cap\td{L}}^E)_*(\gamma_L)=\sum_{k=0}^{r-2}(\iota_{E\cap\td{L}}^E)_*(\pi_Z^{E\cap\td{L}})^*(c_{r-2-k}(N_{Z/L}))\,.\,\xi^k.$$
			Using (\ref{appidentitypar}) we get 
			$$(\iota_{E\cap\td{L}}^E)_*(\pi_Z^{E\cap\td{L}})^*(c_{r-2-k}(N_{Z/L}))=(\pi_Z^E)^*(c_{r-2-k}(N_{Z/L})\,.\,c_1(N_{L/X})_{|Z})+(\pi_Z^E)^*(c_{r-2-k}(N_{Z/L}))\,.\,\xi.$$
			Replacing in the above sum and doing a change of summation index, one gets
			$$(\iota_{E\cap\td{L}}^E)_*(\gamma_L)=\sum_{k=0}^{r-1}(\pi_Z^E)^*(c_{r-2-k}(N_{Z/L})\,.\,c_1(N_{L/X})_{|Z}+c_{r-1-k}(N_{Z/L}))\,.\,\xi^k-(\pi_Z^E)^*(c_{r-1}(N_{Z/L}))$$
			Now using (\ref{apptransChern}) we get~$c_{r-2-k}(N_{Z/L})\,.\,c_1(N_{L/X})_{|Z}+c_{r-1-k}(N_{Z/L})=c_{r-1-k}(N_{Z/X})$, hence the claim.
			\end{proof}			
			
			

\bibliographystyle{alpha}
\bibliography{biblio}

\end{document}